\documentclass[english,10pt,leqno]{article}
\usepackage[utf8]{inputenc}
\usepackage{tipa}
\usepackage{dsfont}
\usepackage{amsxtra}
\usepackage{amsmath}
\usepackage{amssymb}
\usepackage{amsfonts}
\usepackage[all]{xy}
\usepackage{mathrsfs}
\usepackage{amsthm}
\usepackage{enumitem}
\usepackage{mathabx}
\usepackage{bbm}
\usepackage{dsfont}
\usepackage{esint}    
\usepackage{tikz-cd}
\usepackage{babel}

\usepackage{aliascnt}
\usepackage[final]{hyperref}

\makeatletter
\hypersetup{
    unicode=true,           
    pdftoolbar=true,        
    pdfmenubar=true,        
    pdffitwindow=false,     
    pdfstartview={FitH},    
    pdftitle={\@title},     
    pdfauthor={\@author},   
    pdfsubject={},          
    pdfcreator={},          
    pdfproducer={},         
    pdfkeywords={},         
    pdfnewwindow=true,      
    colorlinks,             
    linkcolor=black,        
    citecolor=black,        
    filecolor=black,        
    urlcolor=black          
}
\makeatother

\numberwithin{equation}{subsection}
\def\NewTheorem#1#2{%
    \newaliascnt{#1}{equation}
    \newtheorem{#1}[#1]{#2}
    \aliascntresetthe{#1}
    \expandafter\def\csname #1autorefname\endcsname{#2}
}

\NewTheorem{thm}{Theorem}
\NewTheorem{cor}{Corollary}
\NewTheorem{lem}{Lemma}
\NewTheorem{prop}{Proposition}

\newtheoremstyle{example}{\topsep}{\topsep}%
{}
{}
{\bfseries}
{.}
{2pt}
{\thmname{#1}\thmnumber{ #2}\thmnote{ #3}}

\theoremstyle{example}
\NewTheorem{nota}{Notation}
\NewTheorem{Defi}{Definition}
\NewTheorem{rem}{Remark}
\NewTheorem{rems}{Remarks}
\NewTheorem{prob}{Problem}
\NewTheorem{exas}{Examples}
\NewTheorem{ex}{Example}
\NewTheorem{exa}{Example}
\NewTheorem{version}{Version}

\newlist{assertions}{enumerate}{1}
\setlist[assertions]{label={\rm(\alph*)},  ref={\rm(\alph*)},leftmargin=0.8em, itemindent=1.5em,labelwidth=\itemindent,labelsep=0.8em,align=left}

\newlist{statements}{enumerate}{1}  
\setlist[statements]{topsep=0pt, beginpenalty=10000, label={\rm(\alph*)},  ref={\rm(\alph*)}, leftmargin=0.8em, itemindent=1.5em,labelwidth=\itemindent,labelsep=0.8em,align=left}

\newcommand{\highbox}[2]{\raisebox{0pt}[1.2\height]{$#1#2$}}
\newcommand{\highol}[1]{\ol{\mathpalette{\highbox}{#1}}}

\def\eps{{\varepsilon}}

\def\AAA{\mathbb{A}}

\def\CC{\mathbb{C}}
\def\DD{\mathbb{D}}

\def\GG{\mathbb{G}}

\def\JJ{{\mathbb{J}}}
\def\LL{{\mathbb{L}}}

\def\PP{\mathbb{P}}
 \def\QQ{{\mathbb{Q}}}
\def\RR{\mathbb{R}}

\def\ZZ{\mathbb{Z}}

\def\scrA{\mathscr{A}}
\def\scrB{\mathscr{B}}
\def\scrC{\mathscr{C}}

\def\scrF{\mathscr{F}}

\def\gen{\mathfrak{g}}

\def\len{\mathfrak{l}}

\def\wen{{\mathfrak {w}}}

\def\Cen{\mathfrak{C}}

\def\Sen{\mathfrak{S}}
\def\Uen{\mathfrak{U}}
\def\Ven{\mathfrak{V}}

\def\Ac{\mathcal{A}}
\def\Bc{\mathcal{B}}
\def\Cc{\mathcal{C}}
\def\Kc{\mathcal{K}}
\def\Dc{\mathcal{D}}
\def\Ec{\mathcal{E}}
\def\Fc{\mathcal{F}}
\def\Gc{\mathcal{G}}
\def\Ic{{\mathcal{I}}}
\def\Lc{\mathcal{L}}
\def\Mc{\mathcal{M}}
\def\Nc{\mathcal{N}}

\def\Oc{\mathcal{O}}
\def\Pc{\mathcal{P}}
\def\Qc{\mathcal{Q}}
\def\Rc{\mathcal{R}}
\def\Sc{\mathcal{S}}

\def\Uc{{\mathcal {U}}}

\def\Yc{\mathcal{Y}}

\def\Sen{\mathfrak{S}}

\def\Ad{{\on{Ad}}}
\def\AGS{{\on{AGS}_\k}}
\def\Alg{{\on{Alg}}}

\def\an{{\on{an}}}
\def\At{{\on{At}}}
\def\Aut{\operatorname{Aut}\nolimits}

\def\basic{{\on{basic}}}

\def\c{{\on{c}}}
\def\can{{\on{can}}}

\def\Cat{{\on{Cat}}}
\def\cdga{{\on{CDGA}}}
\def\CE{{\on{CE}}}
\def\CEW{{\on{CEW}}}

\def\Cob{{\on{Cob}}}
\def\Codec{{\on{Codec}}}

\def\Coh{{\on{Coh}}}

\newcommand{\corr}{\mathrm{corr}}
\def\Cosh{{\on{Cosh}}}

\def\dash{\text{-}}
 \newcommand{\Dcorr}{\Dc^\corr}
\def\Dec{{\on{Dec}}}
\def\Der{{\on{Der}}}

\def\dgMod{{\on{dgMod}}}
\def\dgVect{{\on{dgVect}}}
\def\Diff{{\on{Diff}}}

\def\dla{{\twoheadleftarrow}}
\def\DM {{\strict{\on{Mod}}_{[\![\Dc]\!]} }}
\def\DsM{{\strict{\on{Mod}}_\Dc^! }}
\def\DR{{\on{DR}}}
\def\dra{{\twoheadrightarrow}} 
\def\DS{{\on{DS}}}

\def\eb{{\mathbf {e}}}
\def\End{\operatorname{End}\nolimits}
\def\Eq{{\on{Eq}}} 

\def\ev{{\on{ev}}}

\def\FA{\mathrm{FA}}
\def\finsurj{{\mathscr{S}}}   
\def\Fun{{\on{Fun}}}

\def\gl{{\gen\len}}
\def\GL{{\GG\LL}}
\def\GR{{\on{GR}}}
\def\Groth{\on{G\highol{\mathrm{ro}}}}
\def\Grothco{\on{Gro}}
\def\hati{{\widehat{\imath}}}
\def\hatj{{\widehat{\jmath}}}
\def\hatU{{\widehat{U}}}

\def\hocofib{\on{hocofib}}

\DeclareMathOperator*{\hocolim}{\underrightarrow{\mathrm{holim}}}
\DeclareMathOperator*{\holim}{\underleftarrow{\mathrm{holim}}}
\def\hol{{\on{hol}}}
\def\Hol{{\on{Hol}}}
\def\Hom{\operatorname{Hom}\nolimits}

\def\hr{\mathrm{hr}}
\def\hra{{\hookrightarrow}}

\def\Id{\operatorname{Id}\nolimits}
\DeclareMathOperator*{\ilim}{``\underrightarrow{\mathrm{lim}}''}
\def\Im{\operatorname{Im}\nolimits}
\def\Ind{\on{Ind}}

\def\k{\mathbf k}
\def\Ker{\operatorname{Ker}\nolimits}

\DeclareMathOperator*{\laxlim}{\underleftarrow{\mathrm{lax}}}
\DeclareMathOperator*{\oplaxlim}{\underleftarrow{\mathrm{l\highol{\mathrm{ax}}}}}
\def\lc{{\on{lc}}}
\def\lday{{\circledast}}
\def\LDM {{\on{Mod}_{[\![\Dc]\!]} }}
\def\LDsM {{\on{Mod}^!_\Dc }} 
\def\Lie{{\operatorname{Lie}\nolimits}}
\def\lra{\longrightarrow}
\def\lla{\longleftarrow}

\def\Map{\operatorname{Map}}

\def\Mod{{\on{Mod}}}
\def\Mor{\operatorname{Mor}\nolimits} 

\def\Mult{\on{Mult}}

\def\Nm{{\on{Nm}}}

\def\Ob{\operatorname{Ob}\nolimits}
\def\ol{\overline}
\def\on{\operatorname}
\def\oo{{\infty}}
\def\op{{\on{op}}}
\def\Op{{\on{Op}}}

\def\Perf{\on{Perf}}

\DeclareMathOperator*{\plim}{``\underleftarrow{\mathrm{lim}}''}
\def\pol{{\on{pol}}}
\def\pr{{\on{pr}}}
\def\Pro{\on{Pro}}
\def\pt{{\on{pt}}}

\def\QCoh{{\on{QCoh}}}

\def\Ran{{\on{Ran}}}
\def\rday{{\mathbin{\highol{\circledast}}}}
\def\reg{{\on{reg}}}

\def\RHom{{\on{RHom}}}

\def\Sch{{\on{Sch}}}
\def\Sect{{\on{Sect}}}
\def\Set{{\Sc et}}

\def\Sing{{\on{Sing}}}
\def\Spec {\on{Spec}}
\def\Sh{{\on{Sh}}}

\def\sk{{\on{sk}}}

\def\sset{{\Delta^\circ \Sc et}}
\def\st{{\on{st}}}
\def\strict#1{\underline{\underline{#1}}}

\def\Sym{\operatorname{Sym}\nolimits}

\def\Th{{\on{Th}}}
\def\top{{\on{top}}}
\def\Top{{\on{Top}}}
\def\Tot{{\on{Tot}}}

\def\Tw{\mathrm{Tw}}
\def\TwS{\on{Tw}(\finsurj)}

\def\ul{\underline}

\newcommand{\UTwc}{U_\Tw^\mathrm{c}}

\def\Var{{\on{Var}}}
\def\Vect{\on{Vect}}

\def\wb{\widebar}
\def\wc{\widecheck}
\def\We{{\on{We}}}
\def\wh{ \widehat}
\def\wt{\widetilde}

\def\Zar{{\on{Zar}}}

\def\1{{\mathbf{1}}}
\def\lra{\longrightarrow}
\def\lla{\longleftarrow}
\def\(({(\hskip -1mm (}
\def\)){)\hskip -1mm )}
\def\be{\begin{equation}}
    \def\ee{\end{equation}}
\def\ed{\end{document}}

\def\-{{\, \setminus\, }}
\def\={{\,\simeq\,}}

\let\originalleft\left
\let\originalright\right
\renewcommand{\left}{\mathopen{}\mathclose\bgroup\originalleft}
\renewcommand{\right}{\aftergroup\egroup\originalright}

\title{Gelfand-Fuchs cohomology in algebraic geometry and factorization algebras}
\author{Benjamin Hennion, Mikhail Kapranov}

\begin{document}
    
 \maketitle

        \begin{abstract}
        Let $X$ be a smooth affine variety over a field $\k$ of characteristic $0$ and $T(X)$ be the Lie algebra
        of regular vector fields on $X$. We compute the Lie algebra cohomology of $T(X)$ with coefficients in $\k$.
        The answer is given in topological terms relative to any embedding $\k\subset\CC$ and is analogous to
        the classical Gelfand-Fuchs computation for smooth vector fields on a $C^\infty$-manifold. Unlike the
        $C^\infty$-case, our setup is purely algebraic: no topology on $T(X)$ is present.  The proof is based on
        the techniques of factorization algebras, both in algebro-geometric and topological contexts. 
    \end{abstract}
    
    \tableofcontents
    \addtocounter{section}{-1}
    
    \section*{Introduction}
    
    \paragraph{Description of the result.} 
    Let $\k$ be a field of characteristic $0$ and $X$ be a smooth affine algebraic variety over $\k$. 
    Denote by $T(X)=\Der\, \k[X]$  the Lie algebra of regular
    vector fields on $X$. In this paper we determine $H^\bullet_\Lie(T(X))$, the Lie algebra cohomology of $T(X)$ with coefficients in $\k$.
    
    \vskip .2cm

    Clearly, extending the  field of definition of $X$ from $\k$ to $\k'\supset \k$ results in extending the scalars in $H^\bullet_\Lie(T(X))$
    from $\k$ to $\k'$. Since any $X$ can be defined over a field $\k$ finitely generated over $\QQ$ and any such field embeds into the
    complex field $\CC$,  the problem  of finding  $H^\bullet_\Lie(T(X))$ reduces to the case $\k=\CC$, when we can speak about 
    $X_\an$, the space of complex points of $X$ with the transcendental topology. 
    In this case
    our main result, \autoref{thm:main}, implies that $H^\bullet_\Lie(T(X))$ is finite-dimensional in each degree and is an invariant
    of $\dim(X)$, of the rational homotopy type of $X_\an$ 
    and of its rational Chern classes. More precisely, it is identified with
    $H^\bullet_\top(\Sect(\ul Y_X/X_\an))$, the $\CC$-valued cohomology of the space of continuous sections of a natural fibration
    $\ul Y_X\to X_\an$ over $X_\an$. 
    This allows one to easily compute $H^\bullet_\Lie(T(X))$ in many examples, using elementary
    rational homotopy type theory, cf. \S \ref{subsec:explicit}. 
    \vskip .2cm
    
    The interest and importance of this problem stems from its relation to the algebro-geometric study of 
    higher-dimensional analogs of
    vertex algebras, in particular, of Kac-Moody \cite{HiKM} \cite{gwilliam-williams} and Virasoro algebras. While
    the full study eventually involves non-affine varieties  (see n$^\circ$ \ref{par:non-aff} below), the affine case
    already presents considerable difficulties which we address in this paper. Thus, 
    we learned that   \autoref{thm:main}  was conjectured by 
    B. L. Feigin  back in the 1980's but  there has been no proof even in the case of curves, despite some work 
    for holomorphic vector fields  and continuous cohomology \cite{kawazumi} \cite{wagemann}, \cite{wagemann-CMP}. 
    
    \vskip .2cm

    \paragraph{Relation to Gelfand-Fuchs theory.} 
    \autoref{thm:main} is an analog for algebraic varieties of the famous results by Gelfand-Fuchs \cite{gelfand-fuks} \cite{fuks},
    Haefliger \cite{haefliger-ens} and Bott-Segal \cite{bott-segal} on the cohomology of $\Vect(M)$, the Lie algebra
    of smooth vector fields on a $C^\infty$-manifold $M$. We recall the main features of that theory.
    
    \begin{itemize}
        \item[(1)] First, one considers $W_n= \Der \, \RR[\![ z_1, \dots, z_n]\!]$, the Lie algebra of formal vector fields on $\RR^n$,
        with its adic topology. Its cohomology is identified with the cohomology of a certain CW-complex $Y_n$ with
        action of $GL_n(\RR)$, see \cite{fuks} \S 2.2.
        
        \item[(2)] Given an $n$-dimensional $C^\infty$-manifold $M$, the tangent bundle of $M$ gives an associated fibration
        $\ul Y_M\to M$, and $H^\bullet_\Lie(\Vect(M))$ is identified with the cohomology of $\Sect(\ul Y_M/M)$,
        the space of continuous sections \cite{haefliger-ens}  \cite{bott-segal},  \cite{fuks} \S 2.4. 
    \end{itemize}
    
    \vskip .2cm
    
    We notice that $Y_n$ can be realized as a complex algebraic variety  acted upon by $GL_n(\CC)\supset GL_n(\RR)$
    and so any complex manifold $X$ carries the associated fibration $\ul Y_X$ with fiber $Y_n$ (even though the real
    dimension of $X$ is $2n$). It is this fibration that is used in  \autoref{thm:main}. 
    While in the $C^\infty$-theory $\Vect(M)$ is considered with its natural Fr\'echet topology and $H^\bullet_\Lie$
    is understood accordingly (continuous cochains), in our approach $T(X)$ is considered purely algebraically.

    \paragraph{Method of proof: factorization homology.}  In order to prove  \autoref{thm:main}, we use the theory of factorization algebras and factorization homology,
    both in topological \cite{lurie-ha} \cite{CG} \cite{ginot} and algebro-geometric \cite{BD} \cite{G} \cite{FG}  \cite{GL} 
    contexts. In particular, we work systematically with the algebro-geometric version of the Ran space  (\S \ref{subsec:Ran-AG}). 
    
    \vskip .2cm
    
    This theory provides, first of all, a simple treatment of the $C^\infty$-case. That is, the correspondence
    \[
    U \,\mapsto \,\Ac(U) \,=\,\CE^\bullet(\Vect(U))
    \]
    (Chevalley-Eilenberg complex of continuous cochains) is a locally constant factorization algebra $\Ac$ on $M$. As $\Ac$ is
    natural in $M$, it is, by Lurie's theorem \cite{lurie-ha}    \cite{ginot} \S 6.3, 
    determined by an algebra $A_n$ over the little disk operad $E_n$ with a homotopy
    action of $GL_n(\RR)$, so that $H^\bullet_\Lie(\Vect(M))$ is identified with $\int_M(A_n)$, the factorization
    homology of $M$ with coefficients in $A_n$. The Gelfand-Fuchs computation of $H^\bullet_\Lie(W_n)$ identifies $A_n$
    with $C^\bullet(Y_n)$, the cochain algebra of $Y_n$, and the identification with the cohomology of the
    space of sections follows from {\em non-abelian Poincar\'e duality}, see \cite{lurie-ha} \S 5.5.6,
    \cite{GL} \cite{ginot}
    and \autoref{prop:NAPD-1} and \autoref{thm:NAPD} below.
    
    \vskip .2cm
    
    Passing to the algebraic case, we find that $H^\bullet_\Lie(T(X))$ can also be interpreted as the factorization homology
    on the algebro-geometric Ran space (cf. \cite{BD} \S 4.8 for $n=1$ and  \cite{FG} Cor. 6.4.4 in general)
    but the corresponding factorization algebra $\wc \Cc^\bullet$
    is far from being locally constant. Already for $n=1$ it corresponds (after dualizing) to the vertex algebra $\on{Vir}_0$
    (the vacuum module over the Virasoro algebra with central charge $0$) which
    gives a holomorphic 
    but not at all locally constant factorization algebra. See \ref{subsec:CEFA} \ref{par:chiral-envelops} below.
    
    \vskip .2cm
    
    The crucial ingredient in our approach is the {\em covariant Verdier duality} of Gaitsgory and Lurie \cite{GL}
    which is a correspondence $\psi$ between (ordinary, or $*$-) sheaves and $!$-sheaves on the Ran space. 
    For a sheaf $\Fc$ its covariant Verdier dual $\psi(\Fc)$ is the collection $(i_p^!\Fc)_{p\geq 1}$
    where $i_p$ is the embedding of the $p^\mathrm{th}$ diagonal skeleton of the Ran space. In our case $\psi(\wc\Cc^\bullet)$ is
    the algebro-geometric analog of the {\em diagonal filtration} of Gelfand-Fuchs \cite{gelfand-fuks}. It turns out
    that $\psi(\wc\Cc^\bullet)$ is a locally constant factorization algebra even though $\wc\Cc^\bullet$ itself is not.
    This appearance of locally constant objects from holomorphic ones is perhaps the most surprising phenomenon that
    we came across in this work.

    \vskip .2cm
    
    By using non-abelian Poincar\'e duality we show that the factorization homology of $\psi(\wc\Cc^\bullet)$ is identified
    with $H^\bullet_\top(\Sect(\ul Y_X/X))$, and our main result follows from comparing the factorization homology of $\psi(\wc \Cc^\bullet)$
    and $\wc\Cc^\bullet$ for an affine $X$ (``completeness of the diagonal filtration"). 
    
    \vskip .2cm

    We find it remarkable that the classical Gelfand-Fuchs theory has anticipated, in many ways, the modern theory of
    factorization algebras. Thus, the Ran space appears  and is used 
    explicitly (under the name ``configuration space") in the 1977
    paper of Haefliger \cite{haefliger-ens} while the diagram of diagonal embeddings of Cartesian powers $M^I$ 
    is fundamental in the analysis of \cite{gelfand-fuks}.

    \paragraph{Non-affine varieties and future directions.} \label{par:non-aff}
    When $X$ is an arbitrary (not necessarily affine) smooth variety, we can understand $T(X)$ as a dg-Lie algebra
    $R\Gamma(X, T_X)$ and its Lie algebra cohomology is also of great interest. If $X$ is projective, $H^\bullet_\Lie(T(X))$
    plays a fundamental role on Derived Deformation theory (DDT), see  \cite{feigin-icm} \cite{hinich-schechtman} \cite{lurie-dagX}
    \cite{calaque-grivaux}. The corresponding Chevalley-Eilenberg complex is identified
    \[
    \CE^\bullet(T(X)) \,\simeq \, \bigl(\wh \Oc^\bullet_{\Mc, [X]}, d\bigr)
    \]
    with the commutative dg-algebra of functions on the formal germ of the derived moduli space $\Mc$ of complex structures
    on $X$ so $H^0_\Lie(T(X))$ is the space of formal functions on the usual moduli space. In particular, $H^\bullet_\Lie(T(X))$
    is no longer a topological invariant of $X$ and $c_i(T_X)$. 
    
    \vskip .2cm
    
    In the case of arbitrary $X$ we still have an interpretation of $H^\bullet_\Lie(T(X))$ as the factorization homology of $\wc\Cc^\bullet$
    and $\psi(\wc\Cc^\bullet)$ is still locally constant. So our analysis (\autoref{thm:top-approx})  gives a canonical map
    \[
    \tau_X: H^\bullet_\top(\Sect(\ul Y_X)/X) \lra H^\bullet_\Lie(T(X)).
    \]
    While $\tau_X$ may no longer be an isomorphism, it provides an interesting supply of ``topological" classes  in 
    $ H^\bullet_\Lie(T(X))$. For example, when $X$ is projective of dimension $n$, ``integration over the fundamental class of $X$"
    produces,  out of $\tau_X$,  a map
    \[
    H^{2n+1}(Y_n) \lra H^1_\Lie(T(X)),
    \]
    i.e., a supply of {\em characters}  ($1$-dimensional representations)
    of $T(X)$. Cohomology of $T(X)$ with coefficients in such representations
    should describe, by extending the standard DDT, formal sections of natural determinantal bundles on $\Mc$, cf. \cite{feigin-icm}. 
    We recall that
    \[
    H^{2n+1}(Y_n) \,\simeq \, \CC[x_1,\dots, x_n]^{S_n}_{\deg = n+1}, \quad \deg(x_i)=1
    \]
    is identified with the space of symmetric polynomials in $n$ variables of degree $n+1$,  which have the
    meaning of polynomials in the Chern classes.

    \vskip .2cm
    
    In a similar vein, for $X=\AAA^n-\{0\}$ (the ``$n$-dimensional punctured disk''), the space $H^{2n+1}(Y_n)$ maps to
    $H^2_\Lie(T(X))$, i. e., we get a supply of {\em central extensions} of $T(X)$, generalizing the classical
    Virasoro extension for $n=1$.

    \paragraph{The structure of the paper.}
    In Chapter \ref{sec:C-inf-GF} we reformulate, using the point of view of factorization algebras, the classical Gelfand-Fuchs theory.
    We first recall, in \S  \ref{subsec:fac-top}, the theory of factorization algebras  and factorization homology
    on  a $C^\infty$-manifold $M$, in the form given in \cite{CG}, i.e., as dealing with pre-cosheaves on $M$ itself, rather than on
    the Ran space or on an appropriate category of disks. In our case the factorization algebras carry
    additional structures of commutative dg-algebras (cdga's),  so the theory simplifies
    and reduces to cosheaves of cdga's  with no further structure. This simplification is due to 
    \cite{ginot} (Prop. 48), and
    in \S \ref{subsec:fact-cdga} we review  its applications. Unfortunately, we are not aware of
    a similar simplification   in the algebro-geometric setting.

    In \S \ref{subsec:equiv-cdga} we review the concept of factorization homology of $G$-structured manifolds
    in the setting of $G$-equivariant cdga's, where    a self-contained treatment is possible. In particular, 
    we review  non-abelian Poincar\'e duality which will be
    our main tool in relating global objects to the cohomology of the spaces of sections.  It allows us  to give a
    concise 
    proof of the Haefliger-Bott-Segal theorem  in \S \ref{subsec:clas-GF}. The identifications in   \S \ref{subsec:clas-GF}
    are formulated in such a way that they can be re-used later, in \S \ref{subsec:case-tan}, 
    with the full $GL_n(\CC)$-equivariance taken into account. 
    
    \vskip .2cm
    
    Chapter \ref {sec:Dmod} is dedicated to the formalism of $\Dc$-modules which we need in a form more flexible than it is
    usually done. More precisely, our factorization algebras, in their $\Dc$-module incarnation, are not holonomic, but we
    need functorialities that are traditionally available only for holonomic modules.  
    So in \S \ref{subsec:nonstand}
    we introduce  two ``non-standard"
    functorialities on the category of pro-objects. Thus, for a map $f: Z\to W$ of varieties we introduce the functor
    $f^{[\![*]\!]}$ (formal  inverse image), which for $f$ a closed embedding and an induced $\Dc$-module $\Fc\otimes_\Oc\Dc$
    on $W$ corresponds to restriction of sections of $\Fc$ to the formal neighborhood of $Z$. We also introduce the functor
    $f_{[\![!]\!]}$ (formal direct image with proper support) which for an induced $\Dc$-module on $Z$ corresponds to the
    functor $f_!$  on pro-coherent sheaves introduced by Deligne \cite{deligne}. With such definitions we have, for instance, 
    algebraic  Serre duality on non-proper algebraic varieties. 
    
    \vskip .2cm
    
    In Chapter \ref{sec:Ran} we review the algebro-geometric Ran space (\S \ref{subsec:Ran-AG}) and define, in 
    \S \ref{sec:laxmodules}, 
    two main types of $\Dc$-modules on it, corresponding to the concepts of  $*$-sheaves and $!$-sheaves.
    Since we understand the $*$-inverse image in the formal series sense (for not necessarily holonomic modules),
    this understanding propagates into the definition of $\Dc$-module analogs of
    $*$-sheaves, which we call $[\![\Dc]\!]$-modules. We also make a
    distinction between lax and strict modules of both types. In practice, lax modules are more easy to construct
    and are of more finite nature. They can be strictified which  usually produces
    much larger objects but with the
    same factorization homology. In \S \ref{subsec:coVerdieranddiag}  we adapt to our situation the
    concept of covariant Verdier duality from \cite{GL}.

    \vskip .2cm
    
    Chapter \ref{sec:FA} is devoted to the theory of factorization algebras in our algebro-geometric
    (and higher-dimensional) context. Here the main technical issue is to show that covariant Verdier
    duality preserves factorizable objects. This is not obvious in the standard setting when the Ran space is
    represented by the diagram of  the  $X^I$ for all nonempty finite sets $I$ and their surjections. 
    In fact,  this necessitates an
    alternative approach to factorization  algebras themselves:   defining them as collections of data not  on  the $X^I$ 
    (as it is usually done and as recalled in \S \ref{subsec:FA-finsurj}) but  
    on varieties labelled by all  surjective maps $I \twoheadrightarrow J$ of finite sets. This is done in Sections \ref{subsec:straightarrows}
    and \ref{subsec:twistedarrows} (we need, moreover, {\em two forms} of  such a definition, each one good for a particular
    class of properties).  This allows us to prove, in \S \ref {subsec:CVD-FA},  that covariant Verdier duality indeed
    preserves factorization algebras (\autoref{thm:coVerdier-fact}). 
    
    \vskip .2cm
    
    After these preparations, in Chapter \ref{sec:GFAG} we study the  factorization algebras 
    $\Cc_\bullet=\Cc_\bullet(X,L)$ and
    $\wc\Cc^\bullet =\wc\Cc^\bullet(X,L)$ associated to a local Lie algebra $L$ on a smooth
    variety $X$. They can be seen as localized versions of
      the (homological and cohomological)
    Chevalley-Eilenberg complexes of $R\Gamma(X, L)$. In other words, these complexes
    (of vector spaces)
    are obtained from $\Cc_\bullet$ and $\wc\Cc^\bullet$ by 
  applying the factorization (co)homology functor.      
     The factorization algebras $\Cc_\bullet$ and $\wc\Cc^\bullet$  are introduced  in \S \ref{subsec:CEFA}.  In \S \ref{subsec:diag-aff} we specialize to the case of affine $X$ and prove
    \autoref{thm:affinecase} and \autoref{cor:affine-L} which imply that $\psi(\wc\Cc^\bullet)$
    and  $\wc \Cc^\bullet$
    have the same factorization homology.

    \vskip .2cm
    
    Finally, in Chapter \ref{sec:top-pic} we compare the algebro-geometric theory with the topological one. 
    After recalling  in  \S \ref{subsec:C-oo-infty} the concept of $C^\oo$-factorization algebras in the  $\oo$-categorical context  \cite{lurie-ha} \cite{ginot}, we
     outline a general procedure of comparison in \S \ref{subsec:D!-e2n}. We make a particular emphasis the holonomic regular
    case, when we can pass between de Rham cohomology on the Zariski topology
    (which will eventually be related to the purely algebraic object $H^\bullet_\Lie(T(X))$) and
    the cohomology on the complex topology (which will be eventually related to
    the cohomology of $\Sect(\ul Y_X/X_\an)$). In \S \ref{subsec:case-tan} we specialize to the case of the tangent bundle
    where, as we show,  the factorization algebra $\psi(\wc\Cc^\bullet)$ is indeed holonomic regular.
    This allows us to identify the corresponding locally constant factorization algebra 
    and in \S \ref{subsec:main-res} we prove our main result, \autoref{thm:main}. 
    The final \S \ref{subsec:explicit} contains some explicit computations of $H^\bullet_\Lie(T(X))$
    following from  \autoref{thm:main}.

    \paragraph{Acknowledgements.}  We are grateful to D. Gaitsgory, B. L.  Feigin, J. Francis, A. Khoroshkin, D. Lejay,  L. Miaskiwskyi, F. Petit, M. Robalo and  E. Vasserot
    for useful discussions and correspondence.
     We also thank G. Ginot for his careful reading of the draft and his precious comments. Our further gratitude is due to the referees for
  their numerous remarks that helped us improve the exposition. 
    
     The research of the second author was supported by  World Premier International Research Center Initiative (WPI Initiative), MEXT, Japan and by the IAS School of Mathematics.
    
    
    \section{Notations and conventions} \label{sec:notations}
    
    \numberwithin{equation}{section}
    
    \paragraph{Basic notations.}

    $\k$:  a field of characteristic $0$,  specialized to $\RR$ or $\CC$ as needed. 
    
    \vskip .2cm
    
    $\dgVect$: the category of cochain complexes (dg-vector spaces) $V=\{V^i, d_i: V^i \to V^{i+1}\}$ 
    over $\k$, with its standard symmetric monoidal structure
    $\otimes_\k$ (tensor product over $\k$).  

    \vskip .2cm
    
    We use the abbreviation {\em cdga} for ``commutative dg-algebra with unit'',

    \vskip .2cm
    
    $\cdga$: the category of cdga's over $\k$.
    It is also symmetric monoidal with respect to $\otimes_\k$. 
    
    \vskip .2cm
    
    $\Top$: the category of compactly generated topological spaces, see, e.g., \cite[Def. 2.4.21]{hovey}. 
    
    \vskip .2cm
    
    $\Delta$: the standard simplex category with objects the finite nonempty
    ordinals $[p]=\{0,\cdots, p\}$ and morphisms being monotone maps. 
    
    \vskip .2cm
    
    $\Delta^p\subset \RR^{p+1}$: the standard $p$-simplex, i.e., the convex
    hull of the basis vectors. 
    
    \vskip .2cm
    
    $\Delta^\circ \Cc$ resp. $\Delta \Cc$: the category of simplicial,
    resp. cosimplicial objects in a category $\Cc$. In particular, we use the category $\sset$
    of simplicial sets. 
    
    \vskip .2cm
    
    $\Sing_\bullet(T)$: the singular simplicial set of a topological space $T$. 
    If $T$ is a $C^\oo$-manifold, then we have a weakly equivalent simplicial
    subset $\Sing_\bullet^{C^\oo}(T)\subset \Sing_\bullet(T)$ formed
    by  singular simplices $\sigma: \Delta^p\to T$ which are
    $C^\oo$-maps.

    \paragraph{Categorical language.}

    The categories $\dgVect$, $\cdga$, $\Top$, $\sset$ are symmetric monoidal {\em model categories}, see   \cite{hovey} \cite{lurie-htt}
    for background on model structures.  
    We denote by $W$ the classes of weak equivalences in these categories
    (thus $W$ consists of quasi-isomorphisms for $\dgVect$ and $\cdga$). 
    
    \vskip .2cm

    We will mostly use the weaker structure: that of a {\em homotopical category}
    \cite{DHKS} which is a category $\Cc$  with just one  class $W$ of  morphisms, called weak equivalences and satisfying suitable axioms. 
    A homotopical category $(\Cc, W)$ gives rise to a simplicially enriched category $L_W(\Cc)$ (Dwyer-Kan localization).
    Taking $\pi_0$ of the simplicial  Hom-sets in
    $L_W(\Cc)$ gives the usual localization $\Cc[W^{-1}]$. 
    We refer to $L_W(\Cc)$ as the {\em homotopy category} of $(\Cc, W)$ (often, this term is reserved for $\Cc[W^{-1}]$). 
    In  particular,  $(\Cc, W)$ has standard notions of homotopy limits and colimits which we denote
    $\holim$ and $\hocolim$. By an {\em equivalence of homotopical categories} $(\Cc, W) \to (\Cc', W')$
    we mean a functor $F: \Cc\to\Cc'$ such that:
    \begin{itemize}
        \item[(1)] $F(W)\subset W'$.
        
        \item[(2)] The induced functor of simplicially enriched
        categories $L_W(\Cc)\to L_{W'}(\Cc')$ is a {\em quasi-equivalence}, that is:
        \begin{itemize}
            \item[(2a)] It  gives an equivalence of the usual categories
            $\Cc[W^{-1}]\to \Cc'[(W')^{-1}]$.  
            
            \item[(2b)] It induces weak equivalences on the simplicial Hom-sets. 
            
        \end{itemize}  
    \end{itemize}
    \vskip .2cm

    We will freely use the language of $\infty$-categories  \cite{lurie-htt}. In particular, any  simplicially enriched category gives
    rise, in a standard way,  to an $\infty$-category with the same objects, and we will simply consider it as an $\infty$-category. This applies
    to the Dwyer-Kan localizations $L_W(\Cc)$ above. For instance, various derived categories
     will be 
    ``considered as $\infty$-categories'' when needed. 
    
    \paragraph{Thom-Sullivan cochains.} 
    We recall  the {\em Thom-Sullivan functor}
    \[ 
    \Th^\bullet:  \Delta \dgVect \lra \dgVect.
    \]
    Explicitly, for a cosimplicial dg-vector space $V^\bullet$, the dg-vector space
    $\Th^\bullet(V^\bullet)$ is defined as the end, in the sense of
    \cite{maclane}, of the simplicial-cosimplicial dg-vector space $\Omega^\bullet_\pol(\Delta^\bullet)\otimes V^\bullet$
    where $\Omega^\bullet_\pol(\Delta^\bullet)$ consists of polynomial differential forms on the standard simplices,
    see \cite{hinich-schechtman:thom} \cite{felix}. 
    \vskip .2cm
    
    Note that $\Th^\bullet(V^\bullet)$ is quasi-isomorphic to the naive total complex  
    \be\label{eq:Tot}
    \Tot(V^\bullet) \,=\, \left(\bigoplus V^n[-n],  \, d_V + \sum (-1)^i\delta_i\right), 
    \ee
    where $\delta_i$ are
    the coface maps of $V^\bullet$. The quasi-isomorphism is given by the Whitney forms on the 
    $\Delta^n$, see \cite{getzler} \S 3. 
    
    \vskip .2cm
    
   In the case $\k=\RR$ or $\CC$ we can  also use 
   the bigger simplicial-cosimplicial vector space
  $\Omega^\bullet_{C^\oo}(\Delta^\bullet)$ 
      formed by $\k$-valued $C^\oo$-differential formed on the $\Delta^n$. 
      This leads to the $C^\oo$-version of the Thom-Sullivan construction
       of $V^\bullet$ which we denote $\Th^\bullet_{C^\oo}(V^\bullet)$
      and which is quasi-isomorphic to $\Th^\bullet(V^\bullet)$.

    \vskip .2cm

    The functor
    $\Th^\bullet$ (as well as $\Th^\bullet_{C^\oo}$ for $\k=\RR,\CC$)
     is compatible with symmetric monoidal
    structures and so sends cosimplicial cdga's  to cdga's.
    It can, therefore, 
    be used to represent the homotopy limit of any 
    diagram of cdga's as an explicit cdga.
    Indeed, the limit of any diagram can be expressed in terms of
     finite products, cofiltered limits and cosimplicial limits. The first two cases are easy to deal with, and the third one is directly handled by the Thom-Sullivan complex.
    In particular, we have a cdga structure on the cohomology of
    any sheaf of cdga's,  a structure of a sheaf of cdga's in any 
    direct images of a sheaf of cdga's and so on. We note some particular cases. 
    
    \vskip .2cm

    Let $S_\bullet$ be a simplicial set. We write 
    $
    \Th^\bullet (S_\bullet,
    \k)=  \Th^\bullet(\k^{S_\bullet}),
    $
   where $\k^{S_\bullet}$ is the   cosimplicial cdga   $(\k^{S_p})_{p\geq 0}$ (simplicial cochains). The cdga
    $\Th^\bullet (S_\bullet,\k)$
    is called the {\em Thom-Sullivan cochain algebra} of $S_\bullet$. It consists of compatible systems of
    polynomial  differential forms on all the geometric simplices of $S_\bullet$.  
    If $\k=\RR$ or $\CC$, we have a bigger but quasi-isomorphic cdga
     $
     \Th^\bullet_{C^\oo} (S_\bullet,
    \k)=  \Th^\bullet_{C^\oo}(\k^{S_\bullet})
    $
    consisting of compatible systems of
    $C^\oo$  differential forms on all the geometric simplices of $S_\bullet$.
    
    \vskip .2cm
    
    Let $T$ be  a topological space. We write 
    $\Th^\bullet(T) = \Th^\bullet(\Sing_\bullet(T))$. 
    This is a cdga model for the singular cochain algebra of
     $T$ with coefficients in $\k$. 
      If $\k=\RR$ or $\CC$ and $T$ is a $C^\oo$-manifold, then we have
     a chain of quasi-isomorphisms starting  from  the standard $C^\oo$ de Rham
     complex of $T$:
     \be\label{eq:DR=Thom}
     \begin{gathered}
     \Omega^\bullet_{C^\oo}(T,\k) \to 
     \Th^\bullet_{C^\oo}(\Sing_\bullet^{C^\oo}\!(T),\k) \leftarrow \Th^\bullet(\Sing_\bullet^{C^\oo}\!(T), \k)  
     \leftarrow \Th^\bullet(\Sing_\bullet(T),\k). 
     \end{gathered}
     \ee


    \section{\texorpdfstring{$C^\infty$}{C∞} Gelfand-Fuchs cohomology and factorization algebras}\label{sec:C-inf-GF}
    
    \numberwithin{equation}{subsection}

    \subsection{Factorization algebras on \texorpdfstring{$C^\infty$}{C∞} manifolds}\label{subsec:fac-top}
    
    \paragraph{Factorization algebras.} \label{par:fact-alg}

    We follow the approach of \cite{CG} as developed by \cite{ginot}. 
    
    \begin{Defi} Let $M$ be a $C^\infty$-manifold and $(\Cc, \otimes ,\1)$ be a symmetric monoidal category. 
        A {\em pre-factorization algebra}  on $M$ with values in $\Cc$ is a rule   $\Ac$   associating:
        \begin{itemize}
            \item [(1)]  To each open subset $U\subset M$ an object $\Ac(U)\in\Cc$, with $\Ac(\emptyset)=\1$. 
            
            \item[(2)] To each finite family of open sets $U_0, U_1, \dots, U_r$, $r\geq 0$,
            such that $U_1,\dots, U_r$ are disjoint and contained in $U_0$, a permutation invariant morphism  in $\Cc$
            \[
            \mu_{U_1,\dots, U_r}^{U_0}:  \Ac(U_1) \otimes \cdots \otimes \Ac(U_r) \lra \Ac( U_0), 
            \]
            these morphisms satisfying the obvious associativity conditions.
        \end{itemize}
    \end{Defi}

    Taking $r=1$ in (2),  we see that a pre-factorization algebra defines an $\Cc$-valued pre-cosheaf on $M$,
    i.e., a covariant functor from the poset of opens in $M$ to $\Cc$.
    
    \vskip .2cm
    
    Let now $\Cc=\dgVect$ with its standard structure of 
      a symmetric monoidal homotopical category (weak equivalences are quasi-isomorphisms of complexes).

         Let $U\subset M$ be an open subset and $\Uen = (U_i)_{i\in I}$ be an open cover of $U$.
         Let us think of $\Uen$ as a set of open subsets in $M$. 
        Denote $P\Uen$ the set formed by all finite  collections (i.e., sets) of mutually
        disjoint opens $\alpha =\{U_{i_1}, \dots, U_{i_n}\}$ from $\Uen$. 
         Let $\Ac$ be a $\dgVect$-valued
        pre-factorization algebra on $M$.   Following \cite[\S 4.1]{ginot}, we define the
        {\em multiplicative \v Cech complex} of $\Ac$ with respect to $\Uen$ as the simplicial
        dg-vector space

        \be\label{eq:codescent}
           \xymatrix{
    \Cen_\bullet(\Uen, \Fc) \,\,=\,\,
        \biggl\{  \cdots     
          \ar@<.6ex>[r] \ar@<-.6ex>[r] \ar[r] & \displaystyle
          \bigoplus_{\alpha,\beta \in P\Uen} \biggl(  \bigotimes_{U\in\alpha, V\in\beta} \Fc(U\cap V) \biggr)
     \ar@<.4ex>[r] \ar@<-.4ex>[r] & \displaystyle
   \bigoplus_{\alpha\in P\Uen}  \biggl( \bigotimes_{U\in\alpha} \Fc(U) \biggr)
   \biggr\}
   }     
        \ee
    and a morphism
    \be\label{eq:gamma-u}
    \gamma_\Uc:   \hocolim \,  \Cen_\bullet(\Uen, \Fc) \, \lra \, \Fc(U). 
    \ee

    \begin{Defi}\label{Defi:factorizing-cover} Let $U\subset M$ be an open subset and $\Uen = (U_i)_{i\in I}$ be an open cover of $U$.
        We say that $\Uen$ is a
        {\em factorizing cover}, if for any finite subset $\{x_1, \dots, x_n\} \subset U$ there are pairwise
        disjoint opens $U_{i_1}, \dots, U_{i_n}\in\Uen$ s.t. $x_\nu\in U_{i_\nu}$.
       \end{Defi}
    
 Clearly, factorizing  covers are typically very large (consist of infinitely many opens). 
    
    \begin{Defi} \label{Defi:cinfty-factorization-algebra} 
        A pre-factorization algebra $\Ac$ on $M$ with values in $\dgVect$ is called a {\em factorization algebra},
         if the following descent condition holds:
        \begin{itemize}
   
        \item[(3)] For any open $U\subset M$ and any factorizing cover  $\Uen = (U_i)_{i\in I}$  of $U$, the 
            morphism $\gamma_\Uen$ is a  quasi-isomorphism. 
        \end{itemize}
        For a factorization algebra $\Ac$ the object of global cosections will be also denoted by
        \[
        \int_M\Ac \,=\, \Ac(M)
        \]
        and called the {\em factorization homology} of $M$ with coefficients in $\Ac$. 
    \end{Defi}
    
    As pointed out in \cite[Rem.20]{ginot}, a factorization algebra automatically satisfies the following
    property:
    \begin{itemize}
     \item[(4)] For any disjoint open $U_1,\dots, U_r\subset M$, the morphism $\mu_{U_1,\dots, U_r}^{U_1\cup \cdots \cup U_r}$ in (2) 
            is a  quasi-isomorphism.  
            \end{itemize}

    \begin{Defi}~
        \begin{statements}
            \item By a {\em disk} in $M$ we mean an open subset homeomorphic to 
            $\RR^n$ for $n=\dim (M)$.
            \item A (pre-)factorization algebra $\Ac$ with values in  $\Cc$ 
            is called {\em locally constant}, if for any embeddings $U_1 \subset U_0$ of disks in $M$   
            the morphism $\mu_{U_1}^{U_0}$ is a weak equivalence. 
        \end{statements}
    \end{Defi}
    
    
    \subsection{Factorization algebras of cdga's}\label{subsec:fact-cdga}

    \paragraph{Sheaves and cosheaves.} 
    
    We start with a more familiar sheaf-theoretic analog of the formalism of \S \ref{subsec:fac-top}.

    Let $T$ be a topological space. Denote by $\Op(T)$ the poset of open sets in $T$ considered as a
    category.   Let $U\in\Op(T)$ be open and $\Uen=(U_i)_{i\in I}$ be an open cover of $U$;
    we write $U_{ij}=U_i\cap U_j$ etc.

      Let $(\Cc,W)$  be a homotopical category.  Thus
    $\Cc$ has coproducts, denoted $\coprod$. A   $\Cc$-valued pre-cosheaf $\Ac: \Op(T)\to\Cc$  on $T$ gives then the usual \v Cech complex (or codescent diagram)
    which is the simplicial object 
    in $\Cc$ given by
    
        \be\label{eq:codescent2}
    \xymatrix{
        \Nc_\bullet(\Uen, \Fc) \,\,=\,\,
        \biggl\{  \cdots   \ar@<.9ex>[r]   \ar@<.3ex>[r]   \ar@<-.3ex>[r]   \ar@<-.9ex>[r] &
        \displaystyle \coprod_{i,j,k\in I} \Fc(U_{ijk})  \ar@<.6ex>[r] \ar@<-.6ex>[r] \ar[r] &
        \displaystyle \coprod_{i,j\in I} \Fc(U_{ij}) \ar@<.4ex>[r] \ar@<-.4ex>[r] &  
        \displaystyle \coprod_{i\in I} \Fc(U_i) 
        \biggr\} 
    }
    \ee
    and equipped with a canonical morphism
     \[
    \gamma_\Uc^{\text{\v Cech}}:   \hocolim \,  \Nc_\bullet(\Uen, \Fc) \, \lra \, \Fc(U). 
    \]

    \begin{Defi}~
        \begin{statements}
            \item Let $(\Cc, W)$ be a homotopical category. 
            A  $\Cc$-valued pre-cosheaf $\Ac: \Op(T)\to\Cc$  on $T$ is called 
            a {\em homotopy cosheaf} , if, for each
            $U\in\Op(T)$ and each cover $\Uen$ of $U$, 
            the canonical morphism $\gamma_\Uen^{\text{\v Cech}}$  
              is a 
            weak equivalence. 
            \item By a $\Cc$-valued {\em homotopy sheaf} on $T$ we mean a homotopy cosheaf with values in
            $\Cc^\circ$. 
        \end{statements}
    \end{Defi} 
    
    In the sequel we will drop the word ``homotopy" when discussing homotopy sheaves and cosheaves. 
    We denote by $\Sh_T(\Cc)$ and $\Cosh_T(\Cc)$ the categories of $\Cc$-valued sheaves and cosheaves. 

    \begin{Defi}
        Let $M$ be a $C^\infty$ manifold. A $\Cc$-valued cosheaf $\Ac$ on $\Mc$ is called {\em locally constant}
        if for any two disks $U_1\subset U_0\subset M$ the co-restriction map $\Ac(U_1)\to \Ac(U_0)$
        is a weak equivalence. $\Cc$-valued locally constant sheaves on $M$ are defined similarly. 
    \end{Defi}
    
    We denote by $\Sh^\lc_M(\Cc)$ and $\Cosh^\lc_M(\Cc)$ the
    homotopical  categories of locally constant $\Cc$-valued sheaves and cosheaves. 
    
    \begin{prop}\label{prop:sh-cosh-lc}
        The homotopical categories  $\Sh^\lc_M(\Cc)$ and $\Cosh^\lc_M(\Cc)$ are equivalent. 
    \end{prop}
    
    \begin{proof}
        Let $D(M)\subset \Op(M)$ be the poset of disks in $M$. 
        As disks form a basis of topology in $M$,
        any sheaf $\Bc$ on $M$ is determined by its values on $D(M)$. More precisely,
        $\Bc$ is the (homotopy)  right Kan extension   \cite{DHKS}  of $\Bc|_{D(M)}$. This implies that
        $\Sh^\lc_M(\Cc)$ is identified with the category formed by 
        contravariant functors, 
        $D(M)\to \Cc$ sending each morphism to a weak equivalence, i.e., with the category of
        simplicially enriched functors  $ L(D(M)^\circ) \lra L_W(\Cc)$. Here and below $L$ without
        subscript
        stands for  the Dwyer-Kan localization with respect to all morphisms.
        
        Dually, any cosheaf $\Ac$ on $M$ is the homotopy left Kan extension of $\Ac|_{D(M)}$. 
        This means that $\Cosh^\lc_M(\Cc)$  is identified with the category formed by covariant 
        functors
        $D(M)\to \Cc$ sending each morphism to a weak equivalence, i.e., with the category of
        simplicially enriched functors  $ L(D(M))\lra L_W(\Cc)$.
        Now notice that $L(D(M))$ and $L(D(M)^\circ)$ are canonically identified. Indeed, consider the functor $D(M) \to L(D(M)^\circ)$ which is the identity on objects and maps an open immersion $j$ of disks in $M$ to the inverse $j^{-1}$ in $L(D(M)^\circ)$. It factors through $L(D(M))$ and induces the announced equivalence $L(D(M)) \simeq L(D(M)^\circ)$.
    \end{proof}

    \begin{Defi}\label{Defi:inverse} For a locally constant sheaf $\Bc$ we will denote by $\Bc^{-1}$  and call
        the {\em inverse} of $\Bc$ the locally constant
        cosheaf corresponding to $\Bc$.
    \end{Defi}

    \paragraph{Cosheaves of cdga's as factorization algebras.}
    The category $\cdga$  of cdga's over $\k$ is a symmetric monoidal homotopical category, with monoidal operation
    $\otimes_\k$ (tensor product of cdga's) and weak equivalences being quasi-isomorphisms of cdga's.  
    
    Colimits in $\cdga$ follow the general pattern holding for algebraic structures of
    any kind. Let us note,  in particular:
    \begin{itemize}
 \item   [(1)] Filtering colimits in $\cdga$ are calculated at the level of 
  the underlying vector spaces, i.e., as colimits in
   $\Vect$  which, in this
    case, have natural cdga structures.
    
    \item[(2)] A finite   coproduct $\coprod_{i\in I} A_i$, $|I|<\oo$,  is the tensor product 
     $\bigotimes_{i\in I} A_i$ over $\k$. In general,  $\coprod_{i\in I} A_i$
    can be seen, applying (1),  as the restricted tensor product spanned by decomposable tensors
    whose almost all factors are equal to $1$. 
    \end{itemize}
    
 \noindent   The structure of a homotopical category on $\cdga$ gives the concept of
    homotopy colimits. In particular, we will speak about (homotopy) {\em cosheaves}
    of cdga's on a topological space $T$. Thus a precosheaf $\Fc$ of cdga's 
    on $T$ is a cosheaf if and only if, for any  open cover $(U_i)_{i\in I}$
    of an open $U\subset T$, the cdga $\Fc(U)$ is the homotopy colimit of the
    codescent diagram \eqref{eq:codescent2}, with $\coprod$ understood as the
    restricted tensor product.

        This means that
         a cosheaf  of cdga's is typically not a cosheaf of dg-vector spaces, as the coproduct in $\dgVect$ is
        $\oplus$, not $\otimes_\k$. On the other hand, a sheaf of cdga's is indeed a sheaf of dg-vector spaces, as the forgetful functor $\cdga \to \dgVect$ preserves homotopy limits.
  
  \vskip .2cm
  Let $M$ be a $C^\oo$-manifold. By a {\em factorization algebra on $M$ 
  enriched in cdga's} we will mean a factorization algebra $\Ac$ with
  values in $\dgVect$ together with a structure of a cdga on each $\Ac(U)$
  so that each $\mu_{U_1,\dots, U_n}^{U_0}$ is a morphism of cdga's. 

\pagebreak[2]
    \begin{prop}\cite[Prop. 48]{ginot}\label{prop:fact-cdga}
        \begin{statements}
            \item\label{ass:fact-cdga-a} Let $\Ac$ be a factorization algebra on $M$ 
            enriched in cdga's. Then $\Ac$ is a  
            cosheaf on $M$ with values in $\cdga$. 
            
            \item The construction in \ref{ass:fact-cdga-a} (i.e., forgetting all the $\mu_{U_1,\dots, U_r}^{U_0}$ for $r>1$)
            establishes an equivalence between the category of  factorization algebras on $M$  enriched in cdga's and the
            category of    $\Cosh_M(\cdga)$. Under this equivalences, locally constant
            factorization algebras correspond to locally constant cosheaves. \qed
        \end{statements}
    \end{prop}

    \paragraph{Non-abelian Poincar\'e duality I.}
    Let $p: Z\to M$ be a Serre fibration with base a $C^\infty$ manifold $M$ and fiber $Y$.
    We then have the  following pre-sheaf
    and pre-cosheaf of cdga's on $M$:
    \be
    \begin{gathered}
        Rp_*(\ul\k_Z): U\mapsto \Th^\bullet(p^{-1}(U), \k), \\
        p_*^\Sect(\ul\k_Z):  U\mapsto\Th^\bullet( \Sect(p^{-1}(U)/U), \k), 
    \end{gathered}
    \ee
    where  $\Sect(p^{-1}(U)/U)$ denotes the space of continuous sections of the fibration $p^{-1}(U)\to U$. 
    Thus  $Rp_*(\ul\k_Z)$ is in fact a   locally constant sheaf of cdga's, namely the direct image of the constant sheaf $\ul\k_Z$,
    made into a sheaf of cdga's by using Thom-Sullivan cochains. 
    
    The pre-cosheaf $p_*^\Sect(\ul\k_Z)$ is not, in general, a cosheaf of cdga's but we describe
    a class of situations in which it is very closed to being a cosheaf. For this, denote
    $Rp_*^\otimes(\ul\k_Z)$ the cosheaf of cdga's associated to the pre-cosheaf 
    $p_*^\Sect(\ul\k_Z)$. 
    
    Let $U\subset M$ be open. We will say that $U$ is {\em of finite type}, if $U$ admits
    a finite cover by disks in which each nonempty intersection is also a disk. 
    
    \begin{prop}\label{prop:NAPD-1}
    (a) $Rp_*^\otimes(\ul\k_Z)$ is inverse  to $Rp_*(\ul\k_Z)$  (see  \autoref{prop:sh-cosh-lc}).
    In particular, $Rp_*^\otimes(\ul\k_Z)$ is a
        locally constant  cosheaf of cdga's. 
    
    \vskip .2cm
    
    (b) Suppose $Y$ is $n$-connected, where $n=\dim(M)$.
    Then $p_*^\Sect(\ul\k_Z)(U) = Rp_*^\otimes(\ul\k_Z)(U)$ for any open $U\subset M$
    of finite type. 
     \end{prop}
    
    \begin{proof} (a) For a disk $U\subset X$ the space $\Sect(p^{-1}(U)/U)$ is
        homotopy equivalent to $p^{-1}(x)$ for any $x\in U$. Thus  the contravariant
        functor and the covariant functor from $D(M)$ to $\cdga$ given by  $p^{-1}(x)$
        and  $\Sect(p^{-1}(U)/U)$ send each embedding to an equivalence, and these
        equivalences are quasi-inverse to each other. So our statement follows from
        definition of $\Bc^{-1}$ and Proposition \ref{prop:sh-cosh-lc}.
        
        \vskip .2cm
        
        (b) This is a consequence of \cite{bott-segal} Cor. 5.4. 
             \end{proof}
    
    
    \subsection{The classifying space and Cartan-Weil approaches to equivariant cdga's}\label{subsec:equiv-cdga}
    
    We  compare two notions of a topological or Lie group acting on a cdga. 
    
    \paragraph{Classifying space approach.} 
    Let $(\Cc, W)$ be a homotopical category. 
    \begin{Defi}
        Let $T_\bullet$ be a simplicial topological space so that for each morphism $s: [p]\to[q]$ in 
        $\Delta$ we have a morphism of topological spaces $s^*: T_q\to T_p$.
        A $\Cc$-valued sheaf on $T_\bullet$ is a datum $\Fc$ consisting of sheaves $\Fc_p$ on $T_p$
        for each $p$ and of weak equivalences of sheaves $\alpha_s: (s^*)^{-1}(\Fc_p)\to \Fc_q$  given for each  $s: [p]\to[q]$
        and compatible with the compositions. We denote by $\Sh_{T_\bullet}(\Cc)$ the
        category of $\Cc$-valued sheaves on $T_\bullet$. 
    \end{Defi}

    Let $G$ be a topological group and $N_\bullet G\in\Delta^\circ\Top$ its simplicial nerve.
    Thus $N_\bullet G = G\backslash E_\bullet G$, where $E_\bullet G$ is the ``simplex"
    generated by $G$, i.e., the simplicial topological space with
    \be\label{eq:EG}
    E_q G= G^{q+1} = \Map_\Top([q], G)
    \ee
    and the simplicial structure given by composing maps $[q]\to G$ with
    monotone maps $s: [p]\to [q]$. 
    We denote $BG =  |N_\bullet G|$  the classifying space of $G$.  
    We further denote, following \cite{BL} Appendix B,
    \be
    D_G(\pt) = \Sh_{N_\bullet G}(\dgVect),\quad \cdga_G = \Sh_{N_\bullet G}(\cdga).
    \ee

    \begin{prop} For any object  $\Fc$ of $D_G(\pt)$ (resp. of $\cdga_G$) we have the following:
        \begin{assertions}
            \item\label{ass:lca}  Each $\Fc_p$ is a locally constant (in fact, constant) sheaf  of dg-vector
            spaces resp. cdga's on $N_pG$. 
            \vskip .2cm
            
            \item\label{ass:lcb}  $\Fc$   gives a locally constant sheaf $|\Fc|$ of dg-vector
            spaces (resp. cdga's) on $BG$.
        \end{assertions}
    \end{prop}
    
    \begin{proof}
        \ref{ass:lca} Consider any morphism $s: [0]\to [q]$, so $s^*: N_qG\to N_0G=\pt$. 
        Then $\alpha_s$ identifies $\Fc_q$ with the constant sheaf $(s^*)^{-1} (\Fc_0)$. 
        Part \ref{ass:lcb} follows from \ref{ass:lca}.
    \end{proof}

    \begin{cor}\label{cor:hom-inv-D_G(pt)}
        
        A  morphism $ G'\to G$ of topological groups which is a homotopy equivalence,
        induces  equivalences of homotopical categories $  D_G(\pt)\to D_{G'}(pt)$ and
        $\cdga_{G}\to \cdga_{G'}$.
        
    \end{cor}
    
    \begin{proof}
        Follows from homotopy invariance of locally constant sheaves.
    \end{proof}

    \vskip .2cm
    
    Given  $\Fc\in D_G(\pt)$, the component $V^\bullet = \Fc_0\in\Sh_{N_0 G}$ is just a dg-vector space,
    
    \begin{Defi}
        Let $V^\bullet$ be a dg-vector space (resp. a cdga). A  {\em 
        Bernstein-Lunts action} ({\em BL-action} for short) of $G$  on $V^\bullet $
        is an object $\Fc$ in $D_G(\pt)$ (resp. in $\cdga_G)$ together  with identification of dg-vector
        spaces (resp. of cdga's)  $\Fc_0 \simeq V^\bullet$. 
    \end{Defi}
    
    \vskip .2cm

    For a cdga $A$ we denote by $\dgMod_A$, resp. $\cdga_A$ the category of dg-modules over $A$
    resp. commutative differential graded $A$-algebras. 
    
    Let $H^\bullet(BG,\k)$ be the cohomology ring of $BG$ with coefficients in $\k$, i.e., the cohomology
    algebra of the cdga $\Th^\bullet(BG,\k)$. Any $\Fc \in D_G(\pt)$
    gives rise to the dg-space of cochains  $C^\bullet(BG, |\Fc|)$. We note that  $C^\bullet(BG, |\Fc|)$
    is  a kind of homotopy limit (dg-vector space associated to a cosimplicial dg-vector space)
    and so  we can define it using the Thom-Sullivan construction:
    \begin{equation}\label{eq:C(BG,-)}
     C^\bullet(BG, |\Fc|) := \Th^\bullet(C^\bullet(G^\bullet, \Fc_\bullet)),
    \end{equation}
    where $C^\bullet(G^\bullet, \Fc_\bullet)$ is the cosimplicial dg-vector space $[n] \mapsto C^\bullet(G^n, \Fc_{n})$.
    Thus defined,
    $C^\bullet(BG, |\Fc|)$ is a dg-module over $\Th^\bullet(BG, \k)= C^\bullet(BG, \ul\k_{BG})$. Further, if $\Fc$ is a sheaf of cdga's,
    then $C^\bullet(BG, |\Fc|)$ is a  cdga over $\Th^\bullet(BG, \k)$.

    \begin{prop}\label{prop:BL-dg-mod}
        Let $G$ be a connected compact Lie group. Then:
        \begin{assertions}
            \item\label{ass:liegpa} We have an isomorphism 
             $H^\bullet(BG, \k)\to \Th^\bullet(BG, \k)$ in the homotopy category of
             CDGA. 
            
            \item\label{ass:liegpb} The functor $\digamma: \Fc\mapsto C^\bullet(BG, |\Fc|)$ defines symmetric monoidal equivalences of homotopical categories
            \[
            \digamma: D^+_G(\pt) \stackrel{\sim}{\lra} \dgMod^+_{H^\bullet(BG)}, \quad \digamma: \cdga_G^{\geq 0}  \stackrel{\sim}{\lra} \cdga_{H^\bullet(BG)}^{\geq 0},
            \]
            where the subscript ``$+$''  signifies the subcategories formed by sheaves and dg-modules
            bounded below as complexes, and ``$\geq 0$'' signifies the subcategories formed by
            dg-algebras graded by $\ZZ_{\geq 0}$.
        \end{assertions}
    \end{prop}
    
    \begin{proof}
        Part \ref{ass:liegpa} is classical (inclusion of invariant forms on $BG$ 
        into the de Rham complex plus the quasi-isomorphisms \eqref{eq:DR=Thom}). 
         The first equivalence in \ref{ass:liegpb}
        is \cite{BL} Th. 12.7.2. The second equivalence follows from the first by passing to
        commutative algebra objects.
    \end{proof}
    
    \paragraph{Cartan-Weil approach.} 
    In this paper, all algebraic groups will be assumed connected.
    By a {\em proalgebraic group} over $\k$  we will mean a group $\k$-scheme $\GG$
    represented as a projective limit of affine algebraic groups
    \[
   \GG \,=\, \varprojlim{}^{\Sch_\k} \bigl\{ 
   \cdots\buildrel q_{43}\over\lra \GG_3
    \buildrel q_{32}\over\lra \GG_2\buildrel q_{21}\over\lra \GG_1\bigr\},
    \]
    with $\GG_1$ reductive and each $q_{d+1, d}$ surjective with unipotent kernel. 
    
    To such $\GG$ we associate the topological group $\GG(\k)=\varprojlim^\Top \GG_d(\k)$
    homotopy equivalent to $\GG_1(\k)$ and the pro-finite-dimensional Lie algebra
    $\gen = \Lie(\GG) = \varprojlim \gen_d$, $\gen_d=\Lie(\GG_d)$. We equip $\gen$
    with the projective limit topology (taking each $\gen_d$ discrete). 
    
    Besides finite-dimensional examples (such as $\GG=\GG_1$ reductive), the
    following will be used. 
    
    \begin{ex}\label{ex:group-J}
    Let $\JJ_n=\ul\Aut\, \k[\![z_1,\dots, z_n]\!]$ be the group scheme of automorphisms of
    the formal disk. It is a proalgebraic group in our sense, as $\JJ_n = \varprojlim_d \JJ_{n,d}$
    where $\JJ_{n,d}=\ul\Aut(\k[z_1,\dots, z_n]/(z_1, \dots, z_n)^{d+1})$ is the 
    algebraic group of $d$-jets of coordinate changes, so that $\JJ_{n,1}=\GL_{n,\k}$. 
    We have $\Lie(\JJ_n)=W_n^{\geq 1}\subset W_n$, the subalgebra of vector
    fields with zero constant term.
    \end{ex}
    
    For a proalgebraic group $\GG$ we have the coordinate algebra $\Oc(\GG)=\varinjlim 
    \Oc(\GG_d)$ which is a commutative Hopf algebra and the de Rham algebra
    of regular forms $\Omega^\bullet_\reg(\GG)=\varinjlim\Omega^\bullet_\reg(\GG_d)$
    which is a Hopf cdga with $\Omega^0_\reg(\GG)=\Oc(\GG)$.

    \begin{Defi}\label{def:G*}
        (cf.  \cite{guillemin-sternberg} Def.  2.3.1  and \cite{bott-segal} Def. 3.1.) 
        Let $\GG$ be a proalgebraic group over $\k$  with Lie algebra $\gen$ and $V^\bullet\in \dgVect$ be a cochain complex.   A $\GG^*$-{\em action} on  $V^\bullet$  is 
        a datum consisting of:
        \begin{itemize}
            \item[(1)] A regular action of $\GG$ on $V^\bullet$. 
            In particular, for each $\xi\in\gen$ we have the infinitesimal automorphism
            $L_\xi\in \End^0_\k (V^\bullet)$. 
            
            \item[(2)] A continuous $\k$-linear map $i: \gen\to \End^{-1}_\k (V^\bullet)$  such that;
            \begin{itemize}
                \item[(2a)] $i$ is  $\GG$-equivariant with respect to the adjoint action of $\GG$ on $\gen$ and
                the $\GG$-action on $\End^{-1}_\k(V^\bullet)$ coming from (1). 
                
                \item[(2b)] We have
                $[d, i(\xi)] = L_\xi$ for each $\xi\in\gen$. 
                
                \item[(2c)]   For any $\xi_1, \xi_2\in\gen$ we have  $[i(\xi_1), i(\xi_2)]=0$. 
            \end{itemize}
        \end{itemize} 
          The continuity in (2) is meant with respect to the projective limit topologies on $\gen$
    and $\End(V) = \varprojlim_{V'}  \Hom(V', V)$
    where $V'$ runs over finite-dimensional subspaces in $V$. 
    \end{Defi}

    We denote by $\GG^*\dash\dgVect$ the category of  cochain complexes with $\GG^*$-actions.
    This category has a symmetric monoidal structure $\otimes_\k$
    with the operators $i(\xi)$ defined on the tensor products  by the Leibniz rule.
    Commutative algebra objects in $\GG^*\dash\dgVect$ will be called $\GG^*$-{\em cdga's}. Such an algebra
    is a cdga with $\GG$ acting  regularly by automorphisms, so $\gen$ acts by derivations of
    degree $0$, and with $i(\xi)$ being derivations of degree $(-1)$. We denote by
    $\GG^*\dash\cdga$ the category of $\GG^*$-cdga's.

    \begin{prop}\label{prop:G*-coact}~
        \begin{statements}
            \item\label{ass:G*actiona} A $\GG^*$-action on a   cochain complex $V^\bullet$ is the same as a structure  of
            a comodule over the dg-coalgebra $\Omega^\bullet_\reg(\GG)$. This identification is compatible with tensor
            products.
            
            \item\label{ass:G*actionb} A structure of a $\GG^*$-cdga on a cdga $A$ is the same as a  coaction $A \to A\otimes \Omega^\bullet_\reg(\GG)$  from \ref{ass:G*actiona} which is a morphism of cdga's. 
        \end{statements}
    \end{prop}
    
    \begin{proof}
        We prove \ref{ass:G*actiona}, since \ref{ass:G*actionb} follows by passing to commutative
        algebra objects. 
        
        A regular action of $\GG$  is  by definition,  a coaction of the coalgebra
        $\Oc(\GG) =\Omega^0_\reg(\GG)$ of regular functions. More explicitly, the  
        coaction map $c_0: V^\bullet \to\Omega^0_\reg(\GG)\otimes V^\bullet $ is just the  
        action map
        $\rho: \GG\to \End_\k(V^\bullet )$ considered as an element of 
        $\Omega^0_\reg(\GG)\otimes\End_\k (V^\bullet)$. 
        The fact that $c_0$ is a coaction, i.e., that $\rho$ is multiplicative, is equivalent to the fact
        that $\rho$ is equivariant with respect to $\GG$ acting on itself by right translations and on $\End_\k(V^\bullet)$
        by the action induced by $\rho$.
        
        Let us view $\gen$ as consisting of left invariant vector fields on $\GG$. 
        The adjoint action of $\GG$ on $\gen$ is the action  on  such vector fields by right translations.  
        Let $\omega = g^{-1}dg \in \Omega^1_\reg(\GG)\otimes\gen$ be the canonical $\gen$-valued
        left invariant $1$-form on $\GG$. Composing $\omega$ with $i$ we get an element
        $
        c_1\in \Omega^1_\reg(\GG)\otimes\End^{-1}_\k(V^\bullet)
        $
        which is equivariant with respect to $\GG$ acting on $\Omega^1_\reg(\GG)$ by right translations and on
        $\End^{-1}_\k(V^\bullet)$ by the action induced by $\rho$. Further, for each $p\geq 1$ we define, using (2c):
        \[
        c_p\,=\,\Lambda^p(c_1) \,\in\,  \Omega^p_\reg(\GG) \otimes \End_\k^{-p}(V^\bullet). 
        \]
        We claim that the 
        \[
        c=\sum c_p \, \in\,  \bigoplus_p   \, \Omega^p_\reg(\GG) \otimes \End_\k^{-p}(V^\bullet) \,=\, 
        \Hom_\k^0(V^\bullet, \Omega^\bullet_\reg (\GG)\otimes V^\bullet)
        \]
        is a coaction of $\Omega^\bullet_\reg(\GG)$ on $V^\bullet$.  Indeed, 
        considering $c_1$ as a morphism of dg-vector space
        $ V^\bullet \to \Omega^1_\reg(\GG)\otimes V^\bullet $,
        we translate its equivariance property above into saying that
        $c_1$ is compatible with $c_0$. Further compatibilities follow since the $c_p$ are defined
        as exterior powers of $c_1$.
    \end{proof}

  \paragraph{Comparison.} 
     Given a $\GG^*$-action on $V^\bullet$, we form, in a standard way,  the  {\em cobar-construction}
    of the $\Omega^\bullet_\reg(\GG)$-comodule structure.  This is 
    the cosimplicial object
    \[
    \Cob (V^\bullet) = \biggl\{ 
    \xymatrix@C=1em{
        V^\bullet  \ar@<.4ex>[rr]^{\hskip -1cm c }\ar@<-.4ex>[rr]_{\hskip -1cm 1\otimes -} &&\Omega^\bullet_\reg(\GG)\otimes V^\bullet   \ar@<.7ex>[rr] \ar@<-.7ex>[rr]\ar[rr] && 
        \Omega^\bullet_\reg({\GG^2})
        \otimes V^\bullet   \ar@<.9ex>[r] \ar@<.3ex>[r] \ar@<-.3ex>[r] \ar@<-.9ex>[r] & \cdots
    }
    \biggr\}. 
    \]
    For example, the first two coface maps are given by the coaction $c$ and by multiplication with $1\in\Omega^\bullet_\reg(\GG)$
    respectively. 
    
    \vskip .2cm
    Recall that $\k=\RR$ or $\CC$ so we have the Lie groups $G_d=\GG_d(\k)$ and a topological group $G=\varprojlim G_d$ with the canonical projections $q_d: G\to G_d$. 
    For $p\geq 1$ let $\ul\Omega^\bullet_{G_d^p}$ be the sheaf of smooth $\k$-valued
    forms on $G_d^p$. Put 
    $\ul\Omega^\bullet_{G^p} = \varinjlim_d (q_d^p)^{-1} \ul\Omega^\bullet_{G_d^p}$.
    It is a sheaf of cdga's on $G^p$ quasi-isomorphic to $\ul\k_{G^p}$. 
     Thus $\Omega^\bullet_\reg(\GG^p)$ maps into the global sections of $\ul\Omega^\bullet_{G^p}$.
    This allows us to define a $\dgVect$-valued sheaf $\Sen(V^\bullet)$
    on the simplicial space $N_\bullet G$ by ``localizing $\Cob(V^\bullet)$". More precisely,
    we define  
    $\Sen(V^\bullet)_p =\ul \Omega^\bullet_{G^p}\otimes V^\bullet$ and the 
    compatibility maps $\alpha _s$, $s\in\Mor(\Delta)$
    are  induced by the corresponding maps of   $\Cob(V^\bullet)$. 
    This gives a symmetric monoidal functor
    \be\label{eq:sheafify}
    \Sen: \GG^*\dash\dgVect \lra D_G(\pt) = \Sh_{N_\bullet(G)}(\dgVect), \quad V^\bullet \mapsto  \Sen(V^\bullet)
    \ee
    which we call   the {\em sheafification functor}.

    \vskip .2cm
    
    We now recall    a version of the standard comparison between    the ``Weil model'' and the ``classifying space model''
    for equivariant cohomology, cf. \cite{guillemin-sternberg}. 
    An element $v\in V^\bullet$ is called {\em basic}, if $v$ is $\GG$-invariant
    and annihilated by the operators $i(\xi), \xi\in\gen$. Basic elements form a subcomplex $V^\bullet_\basic$ .
    Denote $I(\gen) = S^\bullet(\gen^*)^\gen$. 
    Let also 
    \be\label{eq:weil}
    \We(\gen) \, =\,\bigl( S^\bullet(\gen^*)\otimes \Lambda^\bullet (\gen^*), d\bigr)
    \ee
    be the {\em Weil algebra of} $\gen$. Here the first $\gen^*$ has degree $2$, while the second $\gen^*$ has
    degree $1$. With this grading, $\We(\gen)$
    is a  cdga with $\GG^*$-action,  quasi-isomorphic to $\k$. 
    Moreover, $I(\gen)$ with trivial differential, is a dg-subalgebra in $\We(\gen)$. 
    
    \begin{prop} \label{prop:weil-model}
        Let $\GG$ be reductive.  
        Then:
        \begin{assertions}
            \item\label{ass:isoI} We have an isomorphism 
            $I(\gen) \simeq H^\bullet(B\GG(\CC),\k)$.
            \item\label{ass:qisoModules} For a dg-vector-space (resp. cdga) with a $\GG^*$-action, we have a natural quasi-isomorphism of dg-modules (resp. cdga's) over $I(\gen)=H^\bullet(B\GG(\CC),\k)$
            \[
        \digamma(\Sen(V^\bullet)) \,\simeq \, \bigl(\We(\gen)\otimes_\k V^\bullet\bigr)_\basic. 
            \]
        \end{assertions}
        Here $\digamma(\Sen(V^\bullet))=C^\bullet(B\GG(\CC), |\Sen(V^\bullet)|)$
        is the dg-module (resp. cdga) associated to $\Sen(V^\bullet)$ by Proposition 
        \ref{prop:BL-dg-mod}(b). 
    \end{prop}
    
    \begin{proof}
   \ref{ass:isoI} is well known (it also follows from \ref{ass:qisoModules} using \cite[Thm. 3.2.2]{guillemin-sternberg}). To prove \ref{ass:qisoModules}, note, first, that we have a natural identification 
        \[
        \digamma(\Sen(V^\bullet)) \simeq \Th^\bullet(\Cob(V^\bullet)),
        \]
        compatible with the symmetric monoidal structures. Indeed, this follows from 
        the definition in \eqref{eq:C(BG,-)}, from the identifications 
        \eqref{eq:DR=Thom} and from the fact that for any $p$, the inclusion map 
        $r_p: \Omega^\bullet_\reg(\GG^p) \to \Omega_{C^\oo}^\bullet(\GG(\CC)^p)$ is a quasi-isomorphism. 
        Indeed, note that the  target of $r_p$ calculates 
        $H^\bullet(\GG(\CC)^p, \k)$  
       by the usual de Rham theorem while the source calculates 
  the same by  Grothendieck's algebro-geometric version. 
        
        Note next that for any dg-vector space $E^\bullet$ with $\GG^*$-action
        \be\label{basic=cobar}
        E^\bullet_\basic \,=\, \on{Eq} \bigl\{ 
        \xymatrix@C=1em{
            E^\bullet  \ar@<.4ex>[rr]^{\hskip -1cm c }\ar@<-.4ex>[rr]_{\hskip -1cm 1\otimes -} &&\Omega^\bullet_\reg(\GG)\otimes E^\bullet 
        }
        \bigr\}
        \ee
        is the equalizer of the first two cofaces in $\Cob(E^\bullet)$.

        We apply this to $E^\bullet = \We(\gen)\otimes V^\bullet$ which is quasi-isomorphic to $V^\bullet$. 
        For a dg-vector space $F^\bullet$ let $F^\sharp$ denote the graded vector space obtained from $F^\bullet$
        by forgetting the differential. 
        So to prove part \ref{ass:qisoModules},
        it is enough to show the following acyclicity statement:
        the embedding of the constant cosimplicial graded vector space associated to $E^\sharp_\basic$, 
        into the 
        cosimplicial  graded vector space $\Cob(E^\sharp)$  
        is a weak equivalence.  This is equivalent to saying that the complex
        \[
        E^\sharp\lra \Omega^\bullet_\reg(\GG)\otimes E^\sharp \lra \Omega^\bullet_\reg(\GG^2)\otimes E^\sharp
        \lra\cdots
        \]
        with differential $\sum (-1)^i\delta_i$, is exact everywhere except the leftmost term. 
        But this  complex   calculates
        $\on{Cotor}^\bullet_{\Omega^\sharp(\GG)}
        (\k, E^\sharp)$, i.e., the derived functors of the functor $\beta: E \mapsto E_\basic$ on the category of
        graded $\Omega^\sharp(\GG)$-comodules.  So it is enough to show that for $E=(\We(\gen)\otimes V^\bullet)^\sharp$, 
        the higher derived functors $R^k\beta(E)$ vanish for $k > 0$. 
        
        Now, $\beta(E)$ is obtained by first, taking invariants with respect to the abelian Lie superalgebra
        $i (\gen)$ and then  taking $\GG$-invariants. Since $\GG$ is reductive, taking $\GG$-invariants is an exact
        functor. So  vanishing of $R^{>0}\beta(E)$ will be assured if $\Lambda^\bullet(\gen)$, the enveloping algebra of  $i(\gen)$, 
        acts on $E$ freely.
        This is the case for $E = (\We(\gen)\otimes V^\bullet)^\sharp$.
    \end{proof}
    

    \subsection{Equivariant cdga's and factorization homology}

    \paragraph{Factorization algebras from equivariant cdga's.}
   Let $G\subset GL(n,\RR)$ be   a closed subgroup.  
    We present a simplified instance of a result  due to Lurie \cite{lurie-ha}
        which relates $G$-equivariant $E_n$-algebras with locally constant factorization algebras on $G$-structured
        manifolds, see also \cite{AF-survey}, Prop. 
        3.14  and  \cite{ginot} \S 6.3.  It corresponds, via \autoref{prop:fact-cdga}, to
        factorization algebras enriched in cdga's, for which the $E_n$-structure reduces to
        a commutative one.  In this case, the (co)sheaf language  leads to   a simple direct construction.  
    
    \begin{Defi}\label{def:G-str}
        Let $M$ be a $C^\infty$-manifold of dimension $n$. By a {\em $G$-structure} on $M$ we will mean a reduction of structure
        group of the tangent bundle $T_M$ from $GL_n(\RR)$ to $G$ in the homotopy sense, i.e., a homotopy class of maps
        $\gamma$ making the diagram
        \[
        \xymatrix{
            & BG
            \ar[d]
            \\
            M \ar[ur]^\gamma
            \ar[r]_{\hskip-.5cm \gamma_{TM}}& BGL_n(\RR)
        }
        \]
        homotopy commutative. Here $\gamma_{TM}$ is the map classifying the tangent bundle $TM$. 
    \end{Defi}

    \begin{prop}\label{prop:assoc-alg}
        Let $A$ be a cdga with a BL-action of $G$. 
        One can associate with $A$ the following data:
        \begin{itemize}
       \item[(1)]    For any $n$-dimensional manifold $M$ with
        $G$-structure,  a  locally constant cosheaf of cdga's $\ul A_M$ on $M$.
        
        \item[(2)] For any local diffeomorphism $\phi: M\to N$ of $n$-dimensional
        manifolds with $G$-structures, an isomorphism $\ul A_\phi: \phi^*\ul A_N\to \ul A_M$,
   the $\ul A_\phi$ being compatible with composition of local diffeomorphisms.
   
   \item[(3)] For $M=\RR^n$ with trivial $G$-structure,     $\ul A_{\RR^n}$ 
   is the constant cosheaf corresponding to $A$. 
        
        \end{itemize}
  
    \end{prop}
    
    \begin{proof}
        We note first that by \autoref{prop:sh-cosh-lc}   it is enough to associate to $A$ a locally constant
        {\em sheaf} of cdga's $[A]_M$ on $M$, so that $\ul A_M$ will be defined as the inverse cosheaf $([A]_M)^{-1}$. 
        But a BL-action on $A$ is, by definition, a sheaf $\Bc$ on $N_\bullet(G)$ which gives a locally constant
        sheaf $|\Bc|$ on $BG$, and we define $[A]_M= \gamma^{-1} |\Bc|$, where $\gamma$
        is the $G$-structure on $M$. 
    \end{proof}
    
    We will refer to   $\ul A_M$ as the  {\em cosheaf of cdga's associated to}  $A$ on a $G$-structured manifold $M$.

    \begin{Defi}
        Let $A$ be a  cdga with a BL-action of $G$,  and $M$ be a $G$-structured manifold. The
        {\em factorization homology} of $M$ with coefficients in $A$ is defined as the space of global cosections of
        the cosheaf $\ul A_M$ and denoted by 
        \[
        \int_M(A) \,=\, \ul A_M(M). 
        \]
    \end{Defi}

    \paragraph{Non-abelian Poincare's duality II.} 
    
    Let $G\subset GL_n(\RR)$ be a closed subgroup, $M$ be a $n$-dimensional manifold with $G$-structure.
    Let $Y$ be a CW-complex with $G$-action. This $G$-action gives rise to the {\em associated fibration}
    \be\label{eq:YM}
    \ul Y_M \,=\, P_M \times_G Y \buildrel p \over\lra M
    \ee
    with fiber $Y$. 
    Here $P_M$ is the principal $GL_n(\RR)$-bundle of frames in the tangent bundle $TM$.

    At the same time, the $G$-action on $Y$ gives the 
    fibration 
    \be\label{eq:Y//G}
  (Y/\!/ G)_\bullet  =  G\bigl\backslash (Y\times E_\bullet G) \buildrel\pi\over\lra N_\bullet G
    \ee
    with fiber $Y$. Here $E_\bullet G$ is as in \eqref{eq:EG}. 
     We obtain therefore a
    BL-action of  $G$ on the cdga $\Th^\bullet(Y,\k)$. This action is given by the sheaf
    $R\pi_*(\ul \k)$ of cdga's on $N_\bullet G$.

    \begin{thm}[Non-abelian Poincar\'e duality]\label{thm:NAPD}
        Suppose $Y$ is $n$-connected. 
        Then, the cosheaf of cdga's $Rp_*^\otimes (\ul\k_{\ul Y_M})$
         is identified with
        $\ul{\Th^\bullet(Y,\k)}_M$, the  cosheaf associated to the 
        $G$-equivariant cdga $\Th^\bullet(Y,\k)$. 
        Therefore
        \[
        \Th^\bullet (\Sect(Y_M/M), \k ) \,\simeq \, \int_M (\Th^\bullet(Y, \k)). 
        \]
    \end{thm}
    
    \begin{proof}
        Direct consequence of \autoref {prop:NAPD-1}.
        Indeed, the locally constant  sheaf $[\Th^\bullet(Y, \k)]_M =
         (\ul{\Th^\bullet(Y,\k)}_M)^{-1}$ is,
        by construction,   $R p_*(\ul \k_{Y_M})$.
    \end{proof}
    
    \begin{rem}
        As with \autoref{prop:assoc-alg}, \autoref{thm:NAPD} is an adaptation of a
        result of Lurie   about $G$-equivariant $E_n$-algebras  
        to the much simpler case of $E_n$-algebras being commutative.  It is often formulated
        in a ``dual''  version involving  the non-commutative $G$-equivariant  $E_n$-algebra $C_\bullet(\Omega^n(Y,y))$,
        the singular chain complex of the $n$-fold loop space of $Y$ at a point $y$ (assumed $G$-invariant).
        In this case  $y$ gives rise to a distinguished section
        $\ul y$ of $\ul Y_M$ and we have
        \[
        \int_M \bigl( C_\bullet(\Omega^n(Y, y))\bigr) \,\simeq \, C_\bullet( \Sect_c(\ul Y_M/M)),
        \]
        where $\Sect_c$ stands for sections with compact support (those which coincide with $\ul y$
        outside of a compact subset of $M$).
        The relation between this formulation and \autoref{thm:NAPD} comes from a Koszul duality quasi-isomorphism over the (Koszul self-dual, up to a shift)
        operad $\eb_n$ (singular chains on $E_n$):
        \[
        \Th^\bullet(Y)^!_{\eb_n} \,\simeq \, C_\bullet(\Omega^n(Y, y_0)),
        \]
        see \cite{francis-poincare} \cite{ayala-francis}. Here the notation $ \Th^\bullet(Y)^!_{\eb_n} $
        means the Koszul dual to $\Th^\bullet(Y)$ as a $\eb_n$-algebra.

    \end{rem}


    \subsection {Classical Gelfand-Fuchs theory: \texorpdfstring{$\CE^\bullet (W_n)$ and $C^\bullet(Y_n)$ as  $GL_n(\CC)$}{CE*(Wn) and C*(Yn) as GLn(?)}-cdga's}\label{subsec:clas-GF}
    
    \paragraph{The Gelfand-Fuchs skeleton.} 
    
    Let $\k$ be either $\RR$ or $\CC$. We denote by $W_n(\k)=\Der \, \k[\![z_1,\dots, z_n]\!]$ be the $\k$-Lie algebra of formal vector fields on $\k^n$, equipped with its natural adic topology. By $\CE^\bullet(W_n(\k))$ we denote the $\k$-linear Chevalley-Eilenberg cochain complex of $W_n(\k)$ formed by cochains $\Lambda^p(W_n(\k))\to\k$ which are continuous with respect to the adic topology. 
    By $H^\bullet_\Lie(W_n(\k))$ we denote the cohomology of $\CE^\bullet(W_n(\k))$, i.e., the continuous Lie algebra cohomology of $W_n(\k)$.
    We recall the classical calculation of $H^\bullet_\Lie(W_n(\k))$, see \cite{gelfand-fuks} \cite{fuks}. 
    
    \vskip .2cm
    
    Consider the infinite Grassmannian
    \[
    G(n, \CC^\infty) \,=\, \varinjlim_{N\geq n} G(n, \CC^N) \,\simeq\, BGL_n(\CC)
    \]
    as a CW-complex (union of projective algebraic varieties over $\CC$). For any $N \geq n$ let 
    \[
    E(n, \CC^N) \,=\,\bigl\{ (e_1, \dots, e_n) \in (\CC^n)^N \, \bigl| e_1, \dots, e_n \text{ are linearly independent} \bigr\}
    \]
    be the Stiefel variety formed by partial ($n$-element) frames in $\CC^N$. There is a natural projection
    (principal $GL_n(\CC)$-bundle) 
    \[
    \rho: E(n, \CC^N)\lra G(n, \CC^N), \quad (e_1, \dots, e_n)\, \mapsto\,  \CC e_1 + \cdots + \CC e_n,
    \]
    which associates to each frame the $n$-dimensional subspace spanned by it. Then the union
    \[
    E(n, \CC^\infty) \,=\, \varinjlim_{N\geq n} E(n, \CC^N) \,\simeq\, EGL_n(\CC)
    \]
    is the universal bundle over the classifying space of $GL_n(\CC)$. 
    
    The {\em Gelfand-Fuchs skeleton} $Y_n$ is defined as the fiber product
    \[
    \xymatrix{
        Y_n 
        \ar[d]_\rho
        \ar[r] & EGL_n(\CC)
        \ar[d]^\rho
        \\
        \sk_{2n} BGL_n(\CC) \ar[r]& BGL_n(\CC),
    }
    \]
    where $\sk_{2n}BGL_n(\CC) \subset G(n, \CC^{2n})$ is\footnote{Note that the $2n$-skeleton in $G(n, \CC^{2n})$ agrees with the $2n$-skeleton in $G(n, \CC^{N})$ for any $N \geq 2n$.}  the $2n$-dimensional skeleton with respect to
    the standard Schubert cell decomposition. Thus $Y_n$ is a quasi-projective algebraic variety over $\CC$
    with a free $GL_n(\CC)$-action which makes it a principal $GL_n(\CC)$-bundle over $\sk_{2n}BGL_n(\CC)$.
    More explicitly, $Y_n$ is a closed subvariety in $E(n, \CC^{2n})$ which is, in its turn, a Zariski open subset
    in the affine space of matrices $\on{Mat}(n, 2n)(\CC)$. The following is the classical result of Gelfand-Fuchs
    ( \cite{fuks} Th. 2.2.4). 
    
    \begin{thm}\label{thm:GF-Wn}
     Recall that $\k$ is either $\RR$ or $\CC$.
        \begin{assertions}
            \item We have an isomorphism $H^\bullet_\Lie(W_n,\k) \simeq H^\bullet(Y_n,\k)$ with the cohomology of the topological space $Y_n$ with coefficients in 
            (the constant sheaf) $\k$.
            Further, the cup-product on both sides is equal to zero, as well as all the higher Massey operations. 
            \item The space $Y_n$ is $2n$-connected: its first non-trivial cohomology space is $H^{2n+1}(Y_n,\k)$ (besides $H^0$). 
        \end{assertions}
    \end{thm}
    
    \paragraph{The result for smooth vector fields.} 
    
    Let now $M$ be an $n$-dimensional $C^\infty$-manifold. The group $GL_n(\CC)$ acts on $Y_n$.
    In particular, the action of $GL_n(\RR)\subset GL_n(\CC)$ gives rise to 
    the {\em Gelfand-Fuchs fibration}
    \[
    p:  \ul Y_M \lra M
    \]
    with fiber $Y_n$,  associated to the tangent bundle of $M$, as in
    \eqref{eq:YM}.
    
    \vskip .2cm
    
    We denote by $\Vect_\k(M)$ the Lie algebra of $\k$-valued smooth vector fields on $M$ equipped with the standard
    Fr\'echet topology (as for the space of $C^\infty$-sections of any smooth vector bundle).
    As before, we denote by $\CE^\bullet(\Vect_\k(M))$ and $H^\bullet_\Lie(\Vect_\k(M))$ the $\k$-linear Chevalley-Eilenberg complex of continuous cochains of $\Vect_\k(M)$ and its cohomology. The following is also classical
    (\cite{fuks}  Lemma 1 p. 152). 
    
    \begin{thm}\label{thm:approx}
        Let $M=D\subset \RR^n$ be the standard unit ball. Then the homomorphism $\Vect_\k(M)\to W_n(\k)$
        given by the Taylor series expansion at $0$ induces an isomorphism $H^\bullet_\Lie(W_n(\k)) \to H^\bullet_\Lie(\Vect_\k(D))$. 
    \end{thm}

    The  following theorem  was
    conjectured by Gelfand-Fuchs and proved by Haefliger \cite{haefliger-lnm} \cite{haefliger-ens} and Bott-Segal 
    \cite{bott-segal}.
    
    \begin{thm}\label{thm:GF-Y}
        For any $M$ we have an isomorphism $H^\bullet_\Lie(\Vect_\k(M))\simeq H^\bullet(\Sect(\ul Y_M/M), \k)$. 
    \end{thm}

    \paragraph{Approach via factorization homology.} 
    From the modern point of view, \autoref{thm:GF-Y} can be seen as  a textbook application of the techniques of  factorization homology. 
    In the remainder of this section we give its proof using these techniques, as a precursor to the study of the algebro-geometric case.
    We first formulate a version, in the formalism of \cite{ginot}, of  \cite[Th.6.6.1]{CG}:
    
    \begin{thm}\label{thm:CE-fact}
        Let $\Lc$ be a $C^\infty$ local Lie algebra on $M$, i.e., a smooth $\k$-vector bundle with a Lie bracket on sections given
        by a bi-differential operator.  For an open $U\subset M$ let $\Lc(U)$ be the space of smooth sections of $\Lc$ over $U$,
        considered as a Lie algebra with its Fr\'echet topology, and $\CE^\bullet(\Lc(U))$ be its Chevalley-Eilenberg complex
        of continuous sections. 
        Then 
        \[
        \CE^\bullet(\Lc): U \,\mapsto \, \CE^\bullet(\Lc(U))
        \]
        is a factorization algebra on $M$. 
          \end{thm}
        
        \begin{proof} This is an adaptation of the  proof of   \cite[Th.6.6.1]{CG}. 
         Given disjoint $U_1, \dots, U_n\subset U$, we have the restriction map
        \[
        \Lc(U) \lra \Lc(U_1\amalg\cdots\amalg U_n) \,=\, \Lc(U_1)\oplus\cdots\oplus\Lc(U_n).
        \]
      It is a morphism of topological Lie algebras and so induces a morphism of complexes
      \[
     \mu_{U_1,\dots, U_n}^U:  \CE^\bullet(\Lc(U_1))\otimes\cdots \otimes\CE^\bullet(\Lc(U_n)) \lra \CE^\bullet(\Lc(U)).
      \]  
      This defines a pre-factorization algebra structure on 
        $\CE^\bullet(\Lc)$. It is further clear that this structure satisfies the
        additional property (4) of \S \ref{subsec:fac-top}\ref{par:fact-alg}
        In fact $\mu_{U_1,\dots, U_n}^U$ is an isomorphism
        (not just a quasi-isomorphism), whenever $U=U_1\amalg\cdots\amalg U_n$. 
        
        It remains to prove that $\CE^\bullet(\Lc)$ satisfies multiplicative
        codescent in  factorizing coverings. Given a covering $\Uen$ of an open $U$
        and a set $\alpha=\{U_1,\dots, U_n\} \in P\Uen$, we denote by
        $U_\alpha = U_1\amalg\cdots\amalg U_n$ the corresponding (typically disconnected)
        open in $U$. We also write $U_{\alpha\beta}=U_\alpha\cap U_\beta$ etc. 
        With this notation and taking into account the property (4) above, the
        $p^\mathrm{th}$ term   $\Cen_p(\Uen, \CE^\bullet(\Lc))$ of the multiplicative \v Cech complex
        is written simply as 
        \[
        \bigoplus_{\alpha_0,\dots,\alpha_p\in P\Uen} \CE^\bullet(\Lc(U_{\alpha_1,\dots, \alpha_p}))
        \,=\, \Nc_p(\Ven, \CE^\bullet(\Lc))
        \]
i.e., as the $p^\mathrm{th}$ term  of the ordinary \v Cech complex of $\CE^\bullet(\Lc)$ with respect to the 
covering $\Ven$ formed by the $U_\alpha$, $\alpha\in P\Uen$.       
So we need to prove that if $\Uen$ is a factorizing covering, then  $\CE^\bullet(\Lc)$,
considered  
as a pre-cosheaf of dg-vector spaces, satisfies codescent over $\Ven$. 
For this it suffices to prove that for any $q\geq 0$ the cosheaf $\CE^p(\Lc)$ of ordinary vector
spaces satisfies codescent over $\Ven$. For this, let us interpret each $\CE^q(\Lc(U))$.

Let $\Lc^{\boxtimes q}$ be the $q^\mathrm{th}$ external tensor power of $\Lc$, a local Lie algebra on $M^q$. 
The manifold $M^q$ is acted upon by the symmetric group $S_q$, and $\Lc^{\boxtimes q}$
is $S_q$-equivariant. 
Now, $\CE^q(\Lc(U))$, being the space
of continuous antisymmetric $q$-linear functionals on $\Lc(U)$, is the topological dual
of $\Lc^{\boxtimes q} (U^n)_{-S_q}$, where the subscript ``$-S_q$'' means the space
of anti-coinvariants, i.e., of coinvariants  under the standard action tensored with the
sign character.

So it suffices to show that the pre-sheaf  of topological 
vector spaces $U\mapsto \Lc^{\boxtimes q} (U^q)$ satisfies descent with respect to
the covering $\Ven$ associated to a factorizing covering $\Uen$. 
But if $\Uen$ is factorizing, then the open sets $U_\alpha^q$, $\alpha\in P\Uen$,
cover $U^q$, and our statement follows from the fact that $\Lc^{\boxtimes q}$ is a sheaf
on $M^q$.
\end{proof}
    
    Since $\CE^\bullet(\Lc)$ consists of cdga's, \autoref{prop:fact-cdga} implies:
    
    \begin{cor}
        In the situation of \autoref{thm:CE-fact}, $\CE^\bullet(\Lc)$ is a cosheaf of cdga's on $M$. \qed
    \end{cor}

    We apply this to $\Lc= T_M^\k$ being the tangent bundle of $M$ for $\k = \RR$ or its complexification for $\k = \CC$.  
    
    \vskip .2cm
    
    Consider the algebraic group $\GL_{n,\k}$ over $\k$. It acts on $\k[\![z_1, \dots, z_n]\!]$ and thus on $W_n(\k)$
    in a natural way. Moreover, 
    the cdga $\CE^\bullet(W_n(\k))$ has a natural structure of
    $\GL_{n,\k}^*$-algebra (see \autoref{def:G*}).  Indeed, any $\xi\in\gl_n(\k)$ gives a linear vector field on $\k^n$, also denoted $\xi$,
    which we can consider as an element of $W_n(\k)$. The derivation $i(\xi)$ is given by contraction of the cochains
    with $\xi$. Now, the sheafification functor \eqref{eq:sheafify}
    gives a BL-action of $GL_n(\k)$
    on
    $\CE^*(W_n(\k))$.
    Therefore,
    we have the cosheaf of cdga's $\ul{\CE^\bullet(W_n(\k))}_M$ on $M$ associated with the  BL-action of
    $GL_n(\RR)\subset GL_n(\k)$ on the cdga $\CE^\bullet(W_n(\k))$. 
    
    \begin{prop}\label{prop:TM-Wn}
        The cosheaves of cdga's 
        $\CE^\bullet(T_M^\k)$ and $\ul{\CE^\bullet(W_n(\k))}_M$ are weakly equivalent. 
    \end{prop}
    
 \paragraph{Proof of Proposition \ref{prop:TM-Wn}.}\label{par:proof-TM-Wn}
  For $x\in M$ let $W_x(\k)$ be the Lie algebra of formal $\k$-valued
    vector fields on $M$ near $x$ and $W_x(\k)^{\geq d}\subset W_x(\k)$ be the subalgebra
    of fields whose Taylor expansion begins with terms of degree $\geq d$.
   Thus $W_x(\k)$ is isomorphic to $W_n(\k)$ but not canonically. 
   
   If $U\ni x$ is a disk,  Theorem \ref{thm:approx} implies that the Taylor expansion at $x$ gives a quasi-isomorphism
   $\CE^\bullet(W_x(\k))\to \CE^\bullet(\Vect_\k(U))$. This further implies that  the  cosheaf of cdga's $\CE^\bullet(T_M^\k)$ is
    locally constant  and its costalk at any $x\in M$ is identified with $\CE^\bullet(W_x(\k))$. To identify it
    with $\ul{\CE^\bullet(W_n(\k))}_M$ we compare it with two intermediate objects constructed via Gelfand-Kazhdan-Fuchs
    formal geometry \cite{gelfand-kazhdan, gelfand-kazhdan-fuks}. 
    
    First, each cdga $\CE^\bullet(W_x(\k))$ is, as a vector space, the union of finite-dimensional subspaces
    $\Lambda^\bullet(W_n(\k)/W_x^{\geq d} (\k))^*$. For $d$ fixed, these subspaces form the fibers of
    a natural vector bundle on $M$; for $d$ varying, these bundles form an inductive system.
    We obtain therefore an infinite-dimensional bundle (the inductive limit of the finite-dimensional ones above)
    which we denote $\ul\CEW^\bullet_M$ and whose fibers are the cdga's $\CE^\bullet(W_x(\k))$. 
    
    Alternatively, consider the proalgebraic group $\JJ_n$ over $\RR$ from Example \ref{ex:group-J}.
     We have a principal $\JJ_n(\RR)$-bundle $SM\to M$ formed by formal coordinate systems on
    $M$, see  \cite{gelfand-kazhdan, gelfand-kazhdan-fuks}. 
    It is the projective limit of finite-dimensional principal $\JJ_{n,d}(\RR)$-bundles $S_d M$
    of $d$-jets of coordinate systems, so we will consider the  dg-sheaf 
    $\ul\Omega^\bullet_{SM}$ of forms on $SM$ as the inductive limit 
    of pullbacks of the $\ul\Omega^\bullet_{S_dM}$. 
     The group $\JJ_n(\RR)$ acts on $W_n(\k)$
    and therefore on the cdga $\CE^\bullet(W_n(\k))$. The 
    bundle  associated to $SM$ via this action  
    is $\ul\CEW^\bullet_M$. 
    
    The bundle  $\ul\CEW^\bullet_M$
 carries a natural formally flat connection induced by the standard connection in the
 $\oo$-jet bundle of $TM$ (whose fiber at $x\in M$ is $W_x$). 
  Therefore we have the $C^\oo$ de Rham complex 
    $\ul\Omega^\bullet_M \otimes \ul\CEW^\bullet_M$, a sheaf of cdga's on $M$.
    Since each
    $\CE^\bullet(W_x(\k))$ has finite-dimensional cohomology   (Theorem \ref{thm:GF-Wn}),
    the spectral sequence argument shows that 
    $\ul\Omega^\bullet_M \otimes \ul\CEW^\bullet_M$
     is  locally constant, with
      stalk at any $x\in X$  identified with $\CE^\bullet(W_x(\k))$. 
    
    \begin{prop}\label{prop:CEW-CE}
    The locally constant sheaf of cdga's $\ul\Omega^\bullet_M \otimes \ul\CEW^\bullet_M$ and
    cosheaf of cdga's $\CE^\bullet(T_M^\k)$ are inverse (\autoref{Defi:inverse}) to each other. 
    \end{prop}
    
    \begin{proof}
    Let $U'\subset U$ be two disks. We have the commutative diagram of cdga's
    \[
    \xymatrix{
    (\ul\Omega^\bullet_M\otimes\ul\CEW^\bullet_M)(U)
    \ar[dd]_r \ar[r]^{a_U} & \CE^\bullet(\Vect_\k(U), \Omega^\bullet(U))
    \ar[d]^e &
    \ar[l]_{\hskip 0.7cm b_U} \CE^\bullet(\Vect_\k(U), \k)
    \\
    & \CE^\bullet(\Vect_\k(U), \Omega^\bullet(U'))& 
    \\
     (\ul\Omega^\bullet_M\otimes\ul\CEW^\bullet_M)(U') 
      \ar[r]^{a_{U'}} & \CE^\bullet(\Vect_\k(U'), \Omega^\bullet(U')) \ar[u]^f&
    \ar[l]_{\hskip .7cm b_{U'}} \CE^\bullet(\Vect_\k(U'), \k)
    \ar[uu]_c
    }
    \]
    with the quasi-isomorphisms $r, c$ being the restriction for the sheaf and the corestriction of the cosheaf. 
    The arrows $a_U, a_{U'}$ are given by pointwise Taylor expansion and are quasi-isomorphisms by 
    Theorem \ref{thm:approx}. The arrows $b_U, b_{U'}$ are quasi-isomorphisms by the de Rham theorem for $U, U'$.  
    Similarly for $e,f$. These diagrams of quasi-isomorphisms, considered for all pairs $U'\subset U$ of disks,
    show that the sheaf and the cosheaf are inverse to each other. 
     \end{proof}
     
     Next,  the $\GL_{n,\k}^*$-action on $\CE^\bullet(W_n(\k))$ used to define 
     $\ul{\CE^\bullet(W_n(\k))}_M$ is in fact the restriction of a natural
     $\JJ_n^*$-action, where $\JJ_n$ is considered as a proalgebraic group over $\k$. 
     Indeed, $\JJ_n$ itself acts by formal coordinate changes, and  for any 
     $\xi\in \Lie(\JJ_n) = W_n^{\geq 1}(\k)\subset W_n(\k)$ we have the contraction $i(\xi)$
     with $\xi$ in $\CE^\bullet(W_n(\k))$. So  the sheafification functor \eqref{eq:sheafify}
    gives a BL-action of $J_n = \JJ_n(\k)$ on $\CE^\bullet(W_n(\k))$  given by a locally constant
    sheaf $\Sigma = \Sen(\CE^\bullet(W_n(\k)))$ on $N_\bullet J_n$ and
    so a locally constant sheaf $|\Sigma|$ on $BJ_n$. The principal $\JJ_n(\RR)$-bundle
    $SM\to M$ gives a classifying map 
    \[
    \gamma_{SM}: M\lra B\JJ_n(\RR)\subset B \JJ_n(\k)=BJ_n
    \]
    and therefore the locally constant sheaf of cdga's $\gamma_{SM}^{-1}|\Sigma|$ on $M$.
    
    \begin{prop}\label{prop:gamma_SM=CEW}
    The sheaves of cdga's $\Ac_1:= \gamma_{SM}^{-1}|\Sigma|$  and 
     $\Ac_2=\ul\Omega^\bullet_M \otimes \ul\CEW^\bullet_M$ are weakly equivalent. 
    \end{prop}
    
    \begin{proof}
    Denote  $A=\CE^\bullet(W_n(\k))$, $G=\JJ_n(\RR)$
    and $p: S:=SM\to M$ the projection of the principal $G$-bundle. As $S/G=M$
    (free action), the homotopy quotient, i.e., the realization of the simplicial
    space
    \[
    S_\bullet \,: =\, S/\!/ G \,:=\,\bigl\{ 
    \xymatrix@C=1em{
        S && \ar@<.4ex>[ll]^-{\pr}\ar@<-.4ex>[ll]_-{\text{action}}  G\times S &&  \ar@<.7ex>[ll] \ar@<-.7ex>[ll]\ar[ll] 
     G\times G\times S
         &&  \ar@<.9ex>[ll] \ar@<.3ex>[ll] \ar@<-.3ex>[ll] \ar@<-.9ex>[ll] \cdots
    }
    \bigr\},
    \]
   is weakly equivalent to $M$. The equivalence $q$ consists of projections
   $G^r\times S \to S\buildrel p\over\to M$.  The morphism $\gamma: S_\bullet\to N_\bullet G$ consisting of the projections $\pr_{G^r}: G^r\times S\to G^r$, represents
   $\gamma_{SM}: M\to BG$ via the equivalence $q$. By definition, the sheaf
   $\gamma^{-1}\Sigma$ on $S_\bullet $ consists of sheaves 
   $\pr_{G^r}^{-1}(\ul\Omega^\bullet_{G^r}\otimes A)$ on $G^r\times S$. 
   We can replace these pullback sheaves as relative de Rham complexes along the
   fibers of the $\pr_{G^r}$, getting $\ul\Omega^\bullet_{G^r}\otimes\Omega^\bullet_S\otimes A$. Taking pushforward along the weak equivalence
   $q$, we represent $\Ac_1= \gamma_{SM}^{-1}|\Sigma|$ as
   (the Thom-Sullivan construction of) the cosimplicial sheaf of cdga's on $M$
    \[
    \xymatrix@C=1em{
     p_*(\ul\Omega_S^\bullet\otimes A)  \ar@<.4ex>[rr]^{\hskip -1cm c }\ar@<-.4ex>[rr]_{\hskip -1cm 1\otimes -} &&\Omega^\bullet(G)\otimes  p_*(\ul\Omega_S^\bullet \otimes A)   \ar@<.7ex>[rr] \ar@<-.7ex>[rr]\ar[rr] && 
        \Omega^\bullet({G^2})
        \otimes  p_*(\ul\Omega_S^\bullet \otimes A)  \ar@<.9ex>[r] \ar@<.3ex>[r] \ar@<-.3ex>[r] \ar@<-.9ex>[r] & \cdots
    }
    \]
   which is the cobar-construction of the sheaf $p_*(\ul\Omega_S^\bullet \otimes A)$
   of $\Omega^\bullet(G)$-comodules. Because $G$ acts on $S$ freely, 
   $p_*\ul\Omega^\bullet_S$ is locally cofree over $\Omega^\bullet(G)$.
   Further, $H^\bullet(A)$ is finite-dimensional, with $\Omega^\bullet(G)$ coacting trivially.
   Therefore $p_*(\ul\Omega_S^\bullet \otimes A)$ has a filtration (cohomological truncation of $A$)
    with locally
  cofree quotients. 
  This implies   that the Thom-Sullivan construction, i.e., $\Ac_1$, is identified with the equalizer
  of the first two cofaces, i.e., $\Ac_1 = (p_*(\ul \Omega^\bullet_S\otimes A))_\basic$,
  see \eqref{basic=cobar}. That is, $\Ac_1$ so it is defined inside 
  $p_*(\ul \Omega^\bullet_S\otimes A)$ by the conditions of, first, $G$-invariance
  and second, annihilation by the contractions $i_\xi, \xi\in W_n^{\geq 1}$.
  These contractions are defined by the Leibniz rule, i.e., $i_\xi= i_\xi^\Omega + i_\xi^A$
 with the $i_\xi^\Omega$ acting on $\ul\Omega^\bullet_S$ and $i_\xi^A$ on $A$. 
 The differential in $p_*(\ul \Omega^\bullet_S\otimes A)$ is $d_1=d_\DR + d_A$.

  On the other hand,  the connection $\nabla$ in 
  $\ul\CEW^\bullet_M= (p_*(\ul\Omega^0_S\otimes A))^G$ can be described
  as follows \cite[\S 17.1]{FBZ}. The manifold $S$ is acted upon by all of   $W_n(\RR)$,
  not just $W_n^{\geq 1}(\RR)$. This action is simply transitive, i.e., for each $s\in S$
  the action map $a: W_n(\RR)\to T_sS$ is an isomorphism. For $v\in T_sS$ let $\xi(v)\in W_n(\RR)$
  be the corresponding vector. This defines a Maurer-Cartan form 
  $\omega\in\ul\Omega^1_S\otimes W_n(\RR)$, $v\mapsto \xi(v)$. Since $A$ is also a $W_n(\RR)$-module, we have
  a connection form $D\in\ul\Omega^1_S\otimes\End(A)$ induced by evaluating $\omega$
  so that $d_2= d_1+D$ is a flat connection on the trivial bundle $S\times A$.
  This connection is $G$-invariant and so descends to a connection in the associated
  bundle $(p_*(\ul\Omega^0_S\otimes A))^G$ which is $\nabla$. This means that
  the differential $d_2$ in $p_*(\ul\Omega^\bullet_S\otimes A)^G$  
  preserves the subsheaf $\ul \Omega^\bullet_M\otimes  (p_*(\ul\Omega^0_S\otimes A))^G
  =\Ac_2$. This subsheaf is given by vanishing of the contractions $i_\xi^\Omega$,
  $\xi\in W_n^{\geq 1}(\RR)$ (action in one factor only). 
  
  To compare $\Ac_1$ and $\Ac_2$, note that $D=[d_A, h]$, where $h\in\ul\Omega^1_S
  \otimes\End(A)$ takes $v\in T_sS$ to the contraction $i_{\xi(v)}$. Further, $[d_\DR, h]
  \in\ul\Omega_S^2\otimes \End(A)$ is found by evaluation of $d_\DR\omega = -[\omega, \omega]/2 \in \ul\Omega_S^2\otimes W_n(\RR)$ via $W_n(\RR)\to\End(A)$, $\xi\mapsto i_\xi$. 
 This form takes $v,w\in T_sS$ into $-i_{[\xi(v), \xi(w)]}$ which is equal to
 $-[L_{\xi(v)}, i_{\xi(w)}]$ and so $[d_\DR,h] = -[D,h]$. 
 In other words, $d_1=d_2 -[d_2,h]$. This means that $d_1$ is obtained from $d_2$ by
 conjugation with $e^{-h}$, i.e., $d_1(a) = e^h d_2(e^{-h} a)$. Now $h^2=0$ since different
 $i_\xi$ anticommute so $e^{-h}=1-h$. 
 
 The element $h$ is invariant under $G$-action so conjugation by $1-h$ takes
 $(p_*(\ul\Omega^\bullet_S\otimes A), d_2)^G$ to  $(p_*(\ul\Omega^\bullet_S\otimes A), d_1)^G$. Further, if  a section $x$ of $\ul\Omega^\bullet_S\otimes A$ is such that 
 $i_\xi^\Omega (x)=0$, then $y= (1-h)x=x-hx$ satisfies $(i_\xi^\Omega+ i_\xi^A)(y)=0$,
 as follows from the Leibniz rule and the fact that $h$ anticommutes with $i_\xi^A$
 (since it consists of contractions which anticommute with each other). 
  This means that conjugation by $1-h$ identifies $\Ac_2$ with $\Ac_1$. 
 \end{proof}

    Now  the embedding (homotopy equivalence)
    $GL_n(\RR) \hra \JJ_n(\RR)$ gives a homotopy commutative diagram
    \[
    \xymatrix{
    & B\JJ_n(\RR) \ar@{^{(}->}[r]& BJ_n
    \\
    M \ar[ur]^{\gamma_{SM}} \ar[r]_{\hskip -0.5cm \gamma_{TM}}& BGL_n(\RR) \ar@{^{(}->}[u] &
    }
    \]
      so
     we have the first of
     the claimed chain of equivalences below
    \begin{multline*}
    \ul{\CE^\bullet(W_n(\k))}_M\,  \= \, \bigl(\gamma_{TM}^{-1}(|\Sigma|_{BGL_n(\RR)}|)\bigr)^{-1} 
   \buildrel \eqref{cor:hom-inv-D_G(pt)}\over  \= \bigl( \gamma_{SM}^{-1} (|\Sigma|_{B\JJ_n(\RR)}|)\bigr)^{-1} 
    \\
  \buildrel \eqref{prop:gamma_SM=CEW}\over\=  (\ul\Omega^\bullet_M\otimes \ul\CEW^\bullet_M)^{-1}\buildrel\eqref{prop:CEW-CE}\over \= \CE^\bullet(T_M^\k). 
    \end{multline*}
    The second equivalence follows from homotopy invariance of BL-actions
    (Corollary  \ref{cor:hom-inv-D_G(pt)}) while the other two express
    Propositions \ref{prop:gamma_SM=CEW} and \ref{prop:CEW-CE}. 
    Proposition \ref{prop:TM-Wn} is proved. 
    

\paragraph{Comparison of local models.} 
 Proposition \ref{prop:TM-Wn} gives the first of the two quasi-isomorphisms  below,
 while the second is obtained by specializing Theorem \ref{thm:NAPD}  to  $Y=Y_n$:
\be\label{eq:CE=,TH=}
\begin{gathered}
\CE^\bullet(\Vect_\k(M)) \= \int_M \bigl(\CE^\bullet(W_n(\k)\bigr), \quad
\Th^\bullet(\Sect(Y_M/M),\k) \= \int_M \bigl(\Th^\bullet(Y_n, \k)\bigr).
\end{gathered}
\ee
Here the BL-actions of $GL_n(\RR)$ on the $\k$-cdga's 
$CE^\bullet(W_n(\k))$ and $\Th^\bullet(Y_n, \k)$
were described before Proposition  \ref{prop:TM-Wn} and Theorem
 \ref{thm:NAPD} respectively. So Theorem  \ref{thm:GF-Y} 
would follow from an identification (an isomorphism in the
homotopy category) of these cdga's with BL-action of $GL_n(\RR)$.

If $\k=\CC$, then in both cases the BL-action of $GL_n(\RR)$ appears as
the restriction of a natural BL-action of $GL_n(\CC)$. So we
start with 
 the following statement. 

\begin{thm}\label{thm:W-Y}
    Let $\k=\CC$. 
    There is an identification (isomorphism in the homotopy category)
    \[
    \CE^\bullet (W_n(\CC)) \,\simeq \,  \Th^\bullet(Y_n, \CC)
    \]
    as cdga's with a BL-action of   $GL_n(\CC)$. 
\end{thm}

\begin{proof}
Denote for short $G=GL_n(\CC)$.
   Invoking Proposition
    \ref{prop:BL-dg-mod}(b), we see that is suffices to construct an identification
    (isomorphism in the homotopy category)
  \[
  \digamma( \CE^\bullet (W_n(\CC))) \,\simeq \,
  \digamma(\Th^\bullet(Y_n, \CC))
    \]
   of  cdga's over
    \[
    H^\bullet(BG, \CC)\,=\, \CC[e_1, \dots, e_n], \quad \deg(e_i)=2i.
    \] 
   Here and later in the proof, the  value of $\digamma$ on a cdga with a BL-action of $G$ stands, for shortness,  for the value of $\digamma$ on the sheaf on $N_\bullet G$ representing the
   BL-action.  Now, 
  \[
  \digamma( \Th^\bullet(Y_n, \CC)) \= \Th^\bullet (|(Y_n/\!/G)_\bullet|, \CC)
  \]  
    is the cochain algebra
    of the fibration over $BG$ (homotopy quotient)
    corresponding to the $G$-space $Y_n$,
    see \eqref{eq:Y//G}. Because  $G$ acts freely on $Y_n$
    with quotient  $\sk_{2n}(BG)$, the space 
    $|(Y_n/\!/G)_\bullet|$
  is homotopy equivalent to  $\sk_{2n}(BG)$, so
  \[
   \Th^\bullet (|(Y_n/\!/G)_\bullet|, \k) \= 
   \Th^\bullet(\sk_{2n}(BG), \CC)
  \]
  is identified with the Thom cochain complex  of $\sk_{2n}(BG)$
  as a cdga over \hfill\break $\Th^\bullet(BG, \CC)\= H^\bullet(BG,\CC)$. 
    Since the Schubert cells in the infinite Grassmannian are all of even dimensions,
     the Thom cochain complex cohomology of the $2n$-skeleton is 
    quasi-isomorphic to the truncation of the cohomology of $BG$, i.e., we have
    a quasi-isomorphism
   \[
     \Th^\bullet(\sk_{2n}(BG), \CC) \=   H^\bullet (\sk_{2n}(BG), \CC) \,=\,  \CC[e_1, \dots, e_n]/(\deg >2n),
     \]
      where the RHS has trivial differential and is given the standard structure of a
      $\CC[e_1,\cdots, e_n]$-algebra  by its presentation as a quotient. 
      
      \vskip .2cm
      
    We now identify the cdga $\digamma(\CE^\bullet(W_n(\CC)))$. 
    For this we first recall the standard material on relative Lie algebra cohomology \cite{fuks} Ch.1. 
    
    \vskip.2cm
    
    Let $\gen$ be a Lie subalgebra of a Lie algebra $\wen$. The {\em relative Chevalley-Eilenberg complex}
    of $\wen$ modulo $\gen$ (with trivial coefficients) is defined as
    \[
    \CE^\bullet(\wen,\gen) \,=\, \bigl( \Lambda^\bullet((\wen/\gen)^*)\bigr)^\gen. 
    \]
    We apply this to $\wen = W_n(\CC)$ and $\gen = \gl_n(\CC)$, 
    so $\gen$ is the Lie algebra of the algebraic group 
    $\GG=\GL_{n, \CC}$. 
     Let $I(\gen)$ and $\We(\gen)$ be as in \autoref{prop:weil-model}. In particular $I(\gen) \simeq  H^\bullet(BG,\CC)$.
    Further,  as a $\GG$-representation, the topological dual
    $W_n(\CC)^*$ splits as the direct sum 
    \[
   \bigoplus_{k\geq -1}  V_{n, k}, \quad V_{n,k}  = \Sym^{k+1}(\CC^n) \otimes (\CC^n)^*. 
    \]
    This gives
    a $\GG$-equivariant projection $q: W_n(\CC) \to \gl_n(\CC) = V_{n,0}^*$. The projection $q$ induces a canonical
    $\GG$-equivariant morphism of dg-algebras 
    (a so-called connection)  $\nabla: \We(\gl_n(\CC)) \to \CE^\bullet(W_n(\CC))$,
    see \cite[Thm. 3.3.1]{guillemin-sternberg}. In particular, $\CE^\bullet(W_n(\CC), \gl_n(\CC))$ becomes an algebra over $I(\gl_n(\CC)) \simeq H^\bullet(BG,\CC)$.

    \begin{prop}\label{prop:koszul2}
        We have a canonical quasi-isomorphism of $I(\gen)$-cdga's
        \[
        ( \We(\gen)\otimes \CE^\bullet(\wen))_\basic \,  \simeq \,  \CE^\bullet(\wen)_\basic \, =  \, \CE^\bullet(\wen, \gen). 
        \]
    \end{prop}
    
    \begin{proof}
        The second equality is by definition. The first quasi-isomorphism follows from the existence of $\nabla$ by \cite[Thm. 4.3.1]{guillemin-sternberg}.
    \end{proof}
    
      \vskip .2cm
    
     Combining Propositions  \ref{prop:koszul2}  and  \ref{prop:weil-model}, we get an identification
     \[
     \digamma(\CE^\bullet(W_n(\CC))) \= \CE^\bullet(W_n(\CC), \gl_n(\CC))
     \]
   as a cdga over $\CC[e_1,\cdots, e_n]= I (\gl_n(\CC))$.  
    
    \vskip .2cm
    
    \autoref{thm:W-Y}  now follows from the above identifications and 
      from the  quasi-isomorphism of cdga's over 
    $\CC[e_1,\cdots, e_n]$
    \[
    \CE^\bullet(W_n(\CC),\gl_n(\CC)) \,\simeq \, \CC[e_1, \dots, e_n]/(\deg > 2n). 
    \]
    which is  the original  computation of Gelfand-Fuchs, see \cite[proof of Thm. 2.2.4]{fuks}.
\end{proof}

 \paragraph{Proof of  Theorem \ref{thm:GF-Y}. }
 The case $\k=\CC$ follows directly
from \eqref{eq:CE=,TH=} and   \autoref{thm:W-Y}. Indeed, an identification
of cdga's with a BL-action of $GL_n(\CC)$ given in  \autoref{thm:W-Y}
implies their identification
as cdga's with a BL-action of $GL_n(\RR)$ and so an identification
of the corresponding factorization homology. 

The case $\k=\RR$ is treated by Galois descent.  
Let  $\sigma: \CC\to\CC$ be  the complex conjugation
  An $\RR$-structure on  a cdga  $A$  is
 thus a  $\sigma$-antilinear involution $\sigma_A$ on it, and the
 corresponding subalgebra   of $\RR$-points is the fixed point subalgebra of
 $\sigma_A$. 
The cdga $\CE^\bullet(W_n(\CC))$ has a natural $\RR$-structure
$\sigma_\CE$
  coming from the action of $\sigma$ on $\CC^n$.
 The corresponding cdga of $\RR$-points 
  is $\CE^\bullet(W_n(\RR))$. 
The cdga $\Th^\bullet(Y_n, \CC)$ has an $\RR$-structure $\sigma_\Th$
induced by the action
of $\sigma$ on the coefficients $\CC$, and the corresponding cdga of  
$\RR$-points
is $\Th^\bullet(Y_n,\RR)$. 
Now,  the BL-action of $GL_n(\CC)$ on $\CE^\bullet(W_n(\CC))$
comes from the $\GL_{n,\CC}^*$-action and so it does not respect
 $\sigma_\CE$ but does so after twist with the $\sigma$-action on $GL_n(\CC)$
 itself (whose fixed point locus is $GL_n(\RR)$). 
 The BL-action of $GL_n(\CC)$ on  $\Th^\bullet(Y_n, \CC)$
 does respect $\sigma_\Th$. Analysing the  arguments of 
  \autoref{thm:W-Y}, we see that 
     after restricting  
  from $GL_n(\CC)$ to $GL_n(\RR)$  the identification
 will be compatible  both with  the  BL-actions and with 
  the $\RR$-structures. 
Theorem \ref{thm:GF-Y} is proved.

\section{\texorpdfstring{$\Dc$}{D}-modules and extended functoriality}\label{sec:Dmod}

\subsection {\texorpdfstring{$\Dc$}{D}-modules}\label{subsec:Dmod-dif}

\paragraph{Generalities on $\Dc$-modules.} For general background see \cite{borel} \cite{hotta}.

Let $Z$  a smooth algebraic variety over $\k$. Put $n=\dim(Z)$. 
We denote by 
$\Coh_Z = \Coh_{\Oc_Z}$  and $\QCoh_Z = \QCoh_{\Oc_Z}$ the categories of coherent and quasi-coherent sheaves of $\Oc_Z$-modules.

\vskip .2cm

We denote by $\Dc_Z$ the sheaf of rings of differential operators from $\Oc_Z$ to $\Oc_Z$. 
By  $\Coh_{\Dc_Z} \subset \QCoh_{\Dc_Z}$  we denote the categories of coherent and quasi-coherent sheaves of right $\Dc_Z$-modules.
By $_{\Dc_Z} \Coh \subset {}_{\Dc_Z} \QCoh$  we denote similar categories for  left $\Dc_Z$-modules. 
We will be mostly interested in right $\Dc$-modules.  

\vskip .2cm

Let $D(\Coh_{\Dc_Z})$,    resp. $D^b(\Coh_{\Dc_Z})$  denote  the full (unbounded), resp. bounded  derived $\infty$-categories of 
coherent right $\Dc_Z$-modules.
We consider them as dg-categories and then as stable $\infty$-categories in the standard way. Similarly for $\QCoh_{\Dc_Z}$ etc. 
Since $\Dc_Z$ has finite homological dimension (equal to $2\dim(Z)$), we have the  identification
\[
D^b\Coh_{\Dc_Z} \,\simeq \, \Perf_{\Dc_Z},
\]
where on the right-hand side, we have the category of complexes that are perfect over $\Dc_Z$. 

\vskip .2cm

By $\omega_Z$ we denote the sheaf of volume forms on $Z$, a right $\Dc_Z$-module. 
We have the standard  equivalence ({\em volume twist})
\[
_{\Dc_Z} \QCoh \lra \QCoh_{\Dc_Z}, \quad \Nc \mapsto \Nc\otimes_{\Oc_Z} \omega_Z. 
\]
We call the {\em Verdier duality} the anti-equivalence
\[
D^b(\Coh_{\Dc_Z})^\op \to D^b(\Coh_{\Dc_Z}), \quad \Mc \mapsto  \Mc^\vee \,=\,   \ul\RHom_{\Dc_Z}(\Mc, \Dc_Z) \otimes_{\Oc_Z}\omega_Z[n]. 
\]
For a right $\Dc_Z$-module $\Mc\in \QCoh_{\Dc_Z}$ we have the {\em de Rham complex}
\[
\DR(\Mc) \,=\, \Mc \otimes^L_{\Dc_Z} \Oc_Z. 
\]
For a coherent sheaf $\Fc\in\Coh_Z$ we have the {\em induced right $\Dc_Z$-module}
 $\Fc\otimes_{\Oc_Z} \Dc_Z$. 
It comes with canonical identification
\be\label{eq:DR-Induced}
\DR(\Fc\otimes_{\Oc_Z} \Dc_Z) \,\simeq \, \Fc. 
\ee


For a vector bundle $E$ on $Z$ we denote by $E^\vee = \omega_Z \otimes E^*$ the {\em Serre dual vector bundle}.
Thus the Verdier dual of the induced $\Dc$-module $E\otimes_\Oc\Dc$ is given by
\[
\bigl(E\otimes_{\Oc_Z} \Dc_Z\bigr)^\vee \,\,=\,\, (E^\vee)\otimes_{\Oc_Z}\Dc_Z[n]. 
\]

\paragraph{Analytic de Rham complex.} 
Let $\k=\CC$. We denote $Z_\an$ the space $Z(\CC)$ with its standard analytic
topology and sheaf $\Oc_{Z_\an}$ of analytic functions. 

For any $\Mc\in\Coh_{\Dc_Z}$, we define the analytic de Rham complex as
\[
\DR(\Mc)_\an := \Mc \otimes^L_{\Dc_Z} \Oc_{Z_{\an}}. 
\]
This extends straightforwardly to objects of $D(\Coh_{\Dc_Z})$.
The following are standard features of the Riemann-Hilbert correspondence between holonomic $\Dc$-modules and
constructible complexes, see \cite{hotta} for instance. 

\begin{prop}\label{prop:deRham-duality}
    Let $\Mc^\bullet\in D(\Coh_{\Dc_Z})$ be a complex with holonomic regular cohomology modules. Then
    $\DR(\Mc^\bullet)_\an$ is a constructible complex on $Z_\an$ and:
    
    \begin{assertions}
        \item We have 
        \[
        R\Gamma(Z_\Zar, \DR(\Mc^\bullet)) \,\,\simeq \,\, R\Gamma (Z_\an, \DR(\Mc^\bullet)_\an). 
        \]
        \item We also have
        \[
        \DR(M^\vee)_\an \,\,\simeq \,\, \DD_Z(\DR(\Mc^\bullet)_\an),
        \]
        where $\DD_Z$ is the Verdier duality on the derived category of constructible complexes on $Z_\an$, see \cite {KS-sheaves}.
    \end{assertions}
\end{prop}

%
%


\subsection{The standard functorialities}

We now review the standard functorialities on quasi-coherent $\Dc$-modules. Our eventual interest
is always in {\em right} modules. 

\paragraph{Inverse image $f^!$.} 

Let $f: Z\to W$ be a morphism of smooth algebraic varieties. We then have the {\em transfer bimodule}
\[
\Dc_{Z\to W} \,\,=\,\, \Oc_Z \otimes_{f^{-1}\Oc_W}  f^{-1} \Dc_W.
\]
It can be viewed as consisting of differential operators from $f^{-1}\Oc_W$ to $\Oc_Z$. 
It is a left $\Dc_Z$-module (quasi-coherent but not, in general, coherent) and a right $f^{-1}\Dc_W$-module.

The inverse image is most easily defined on left $\Dc$-modules in which case it is given by $f^*$, the usual
(derived) inverse image for underlying $\Oc$-modules. That is, we have the functor
\begin{gather*}
f^*: D({}_{\Dc_W}\QCoh) \lra D({}_{\Dc_Z}\QCoh),\\
f^*\Nc \,= \, \Oc_Z\otimes^L_{f^{-1}\Oc_W} f^{-1}\Nc
\,=\, \Dc_{Z\to W}\otimes^L_{f^{-1}\Dc_W} f^{-1}\Nc. 
\end{gather*}
The corresponding functor on right $\Dc$-modules is denoted
\[
f^!: D(\QCoh_{\Dc_W}) \lra D(\QCoh_{\Dc_Z}), \quad f^!\Mc \,=\, \omega_Z \otimes_{\Oc_Z}\bigl( f^*(\omega^{-1}_W \otimes_{\Oc_W}\Mc)\bigr). 
\]

\paragraph{Compatibility of $f^!$  with $\DR$ on complex topology.} 
Let $\k=\CC$ .

\begin{prop}[see {\cite[VIII.\@ Thm 14.4]{borel}}]\label{prop:f^!-an}
    Let $\Mc^\bullet$ be a bounded complex of quasi-coherent right $\Dc_W$-modules with cohomology modules
    $\ul H^j(\Mc^\bullet)$ being holonomic regular. Then 
    \[
    \DR(f^!\Mc^\bullet)_\an \,\,\simeq \,\, f^!(\DR(\Mc^\bullet)_\an),
    \]
    where in the right hand side we have the usual topological functor $f^!$ on constructible complexes. \qed
\end{prop}

Because of the above compatibilities, we use the same notation $f^!$ for the functor on right $\Dc$-modules
as well as for the corresponding functors on the de Rham complexes.

\paragraph{Direct image $f_*$.} 
The direct image of right $\Dc$-modules is the functor
\[
f_*: D(\QCoh_{\Dc_Z}) \lra D(\QCoh_{\Dc_W}), \quad f_*\Mc \,=\, Rf_\bullet\bigl( \Mc\otimes_{\Dc_Z} \Dc_{Z\to W}\bigr),
\]
where $Rf_\bullet$ is the usual topological derived direct image functor on sheaves on the Zariski topology.
In the particular case when $f=p: Z\to\pt$ is the projection to the point, $\Dc_{Z\to\pt}=\Oc_Z$, and we will use the
following notation:
\[
R\Gamma_\DR(Z,\Mc) \,=\,  p_*\Mc  \,=\, R\Gamma(Z, \DR(\Mc)) \,\,\in \,\, D(\Vect_\k). 
\]

Here are the standard properties of $f_*$ (see \cite{borel}):
\begin{prop}\label{prop:propertiespushforward}~
    \begin{statements}
        \item\label{ass:f*a} If $f$ is \'etale, then $f_*$ is right adjoint to $f^!$.
        \item\label{ass:f*b} If $f$ is proper, then $f_*$ takes $D(\Coh_{\Dc_Z})$ to $D(\Coh_{\Dc_W})$ as well as $D^b(\Coh_{\Dc_Z})$ to $D^b(\Coh_{\Dc_W})$ etc. 
        In this case $f_*$ is left adjoint to $f^!$. 

        \item\label{ass:f*d} Let $\k=\CC$ and $\Mc^\bullet$ be a bounded complex of quasicoherent right $\Dc_Z$-modules with holonomic regular
        cohomology modules. Then
        \[
        \DR(f_*\Mc^\bullet)_\an \,\,\simeq \,\, Rf_\bullet(\DR(\Mc^\bullet)_\an),
        \]
        where in the right hand side $Rf_\bullet$ is the topological direct image of sheaves on the complex topology.
    \end{statements}
\end{prop}

\begin{prop}
    Let  $j: Z\hookrightarrow W$ be an open embedding, with $i: K\hookrightarrow W$ be the closed embedding of the complement.
    Let  $\Mc\in\QCoh_{\Dc_W}$.
    We have canonical quasi-isomorphisms
    \[
    i_* i^!\Mc \,\simeq \,  \ul{R\Gamma}_K(\Mc), \quad j_* j^!\Mc \,\simeq Rj_\bullet j^{-1}\Mc
    \]  
    where on the right hand sides we have purely sheaf-theoretical operations for sheaves on the Zariski topology. We further have 
    the canonical triangle in $D(\QCoh_{\Dc_W})$
    \[
    i_* i^!\Mc\lra \Mc \lra j_* j^! \Mc\to i_* i^!\Mc[1]. 
    \]
\end{prop}

\begin{proof}
    For the first identification, see \cite{saito}. The second one is obvious since $j$ is an open embedding.
    After this, the triangle in question is just the standard sheaf-theoretic triangle
    \[
    \ul{R\Gamma}_K(\Mc) \lra \Mc \lra Rj_\bullet j^{-1}\Mc \lra \ul{R\Gamma}_K(\Mc)[1]. 
    \]
    See also \cite[Prop. 1.7.1]{hotta}.
\end{proof}

\paragraph{$\Dc$-modules on singular varieties.} 
The above formalism is extended, in a standard way, to right $\Dc$-modules on possibly singular varieties. Let us briefly recall
this procedure, following the treatment of \cite{saito} for the case of analytic varieties.

Let $i: Z\to \wt Z$ be a closed embedding of a (possibly singular) variety $Z$ into a smooth variety $\wt Z$. We define the categories
\[
\QCoh_{Z, \Dc_{\wt Z}} \subset \QCoh_{\Dc_{\wt Z}},  \quad \Coh_{Z, \Dc_{\wt Z}}\subset \Coh_{\Dc_{\wt Z}}
\]
to be the full subcategories of quasi-coherent and  coherent $\Dc_{\wt Z}$-modules which are, sheaf-theoretically, supported on $Z$. 
If $Z$ is smooth, then, as well known (Kashiwara's lemma, see \cite[Thm. 1.6.1]{hotta}), the functors
\[
i_*: \Coh_{\Dc_Z} \longleftrightarrow \Coh_{Z, \Dc_{\wt Z}}: i^!
\]
are mutually quasi-inverse equivalences, and similarly for $\QCoh$. This implies that the categories $\QCoh_{Z, \Dc_{\wt Z}}$ and $\Coh_{Z, \Dc_{\wt Z}}$ are canonically (up to equivalence which is unique up to a unique isomorphism)  independent on the choice of an embedding $i$ and we denote them
\[
\QCoh_{\Dc, Z} \,: = \, \QCoh_{Z, \Dc_{\wt Z}}, \quad \Coh_{\Dc, Z} \, := \, \Coh_{Z, \Dc_{\wt Z}}, \quad \forall \,\, Z\buildrel i\over \to\wt Z.
\]
We note that
\[
D^b(\Coh_{Z, \Dc_{\wt Z}}) \,\simeq \, \Perf_{Z, \Dc_{\wt Z}}
\]
is identified with the category of perfect complexes of right $\Dc_{\wt Z}$-modules which are exact outside of $Z$.
We thus define
\[
\Perf_{Z,\Dc} := D^b(\Coh_{Z, \Dc_{\wt Z}}) \simeq \Perf_{Z, \Dc_{\wt Z}}.
\]
In particular, we have $D(\QCoh_{Z,\Dc}) \simeq \Ind(\Perf_{Z,\Dc})$.

\vskip .2cm

Given a morphism $f: Z\to W$ of possibly singular varieties, we can extend it to a commutative diagram
\[
\xymatrix{
    Z
    \ar[d]_{ i}
    \ar[r]^f&W
    \ar[d]^{i'}
    \\
    \wt Z \ar[r]_{\wt f} &\wt W
}
\]
with $i, i'$ being closed embeddings into smooth varieties. After this, the functor $f_*: D(\QCoh_{\Dc, Z})\lra  D(\QCoh_{\Dc, W})$
is defined as the $\infty$-functor
\[
\wt f_*: D(\QCoh_{Z, \Dc_{\wt Z}})\lra  D(\QCoh_{W, \Dc_{\wt W}}). 
\]
Further, the functor $f^!:  D(\QCoh_{\Dc, W})\lra  D(\QCoh_{\Dc, Z})$ is defined as
\[
\ul{R\Gamma}_Z\circ \wt f^!:  D(\QCoh_{W, \Dc_{\wt W}})\lra   D(\QCoh_{Z, \Dc_{\wt Z}}). 
\]
Finally, the  Verdier duality
\[
D(\QCoh_{\Dc, Z})^\op \lra D(\QCoh_{\Dc, Z}), \quad \Mc\mapsto \Mc^\vee
\]
is given by the $\infty$-functor (using Verdier duality on $\wt Z$)
\[
D(\QCoh_{Z, \Dc_{\wt Z}})^\op \to D(\QCoh_{Z, \Dc_{\wt Z}}), \quad \Mc \mapsto \omega_{\wt Z}\otimes _{\Oc_{\wt Z}}
\ul{\RHom}_{\Dc_{\wt Z}}(\Mc, \Dc_{\wt Z}).
\]

\begin{prop}\label{prop:base1}~
    \begin{statements}
        \item These definitions are canonically independent on the choices.
        \item If $f$ is proper, then the $\infty$-functor $f_*$ is left adjoint to $f^!$, and there is a canonical equivalence $f_*(\Mc^\vee) \simeq (f_*\Mc)^\vee$.
        \item The $\infty$-functor $f_*$ is right adjoint to $f^!$ when $f$ is \'etale.
        \item\label{ass:base1b} For any Cartesian square of varieties
        \[
        \xymatrix{
            Z_{12}
            \ar[d]_{g_1}
            \ar[r]^{g_2} & Z_1
            \ar[d]^{f_1}
            \\
            Z_2 \ar[r]_{f_2}& Z
        }
        \]
        we have a canonical base change identification
        \[
        f_2^! \circ (f_{1})_* \,\simeq \, (g_1)_* \circ g_2^!.
        \]
    \end{statements}
\end{prop}
\begin{proof}
 This follows from \cite[Vol.\@ II Chap.\@ 4]{GR}. More precisely, the case of crystals is carried out in \S 2.1.3 and Corollary 2.2.7 and a comparison between crystals and D-modules is given in \S 4 therein.
\end{proof}

\vskip .2cm

For $\Mc^\bullet\in D(\QCoh_{\Dc, Z})$ we have the canonically defined de Rham complex $\DR(\Mc)$ of sheaves
on $Z$. If $\Mc$ is represented by a complex $\wt\Mc^\bullet\in D(\QCoh_{Z, \Dc_{\wt Z}})$, then
$\DR(\Mc^\bullet)$ is represented by the de Rham complex $\DR(\wt\Mc^\bullet)$ which is canonically
independent on the choices. 

\vskip .2cm

Further, we have a well defined concept of holonomic (resp. holonomic regular) objects of $\Coh_{\Dc, Z}$.
We denote by $\Hol_{\Dc,Z}$ the category of holonomic objects of $\Coh_{\Dc, Z}$. The corresponding derived $\infty$-category
will be denoted by 
\[
D^b_\hol\QCoh_{\Dc, Z} \,\,\simeq \,\, D^b \Hol_{\Dc, Z}.
\]
Here the LHS means the $\infty$-category of complexes with holonomic cohomology, the RHS the $\infty$-category of complexes consisting of holonomic modules and the equivalence between two derived categories thus defined follows from the result of Beilinson \cite[Th. 1.3]{beilinson} (which is formulated
 for classical triangulated categories but adapts immediately to the $\oo$-categorical setting). 
More precisely, Beilinson proves this type of result in several contexts including, besides
holonomic modules,  also perverse sheaves
and holonomic regular modules, each time without  requiring $Z$ to be smooth.


\subsection{The nonstandard functorialities}\label{subsec:nonstand}

\paragraph{Reminder on ind- and pro-$\Dc$-modules.} 
Let $Z$ be a (possibly singular) variety over $\k$. 
As in \cite{deligne}, we have identifications
\[
\QCoh_Z \,\,\simeq \,\, \Ind (\Coh_Z), \quad \QCoh_{\Dc, Z} \,\,\simeq \,\, \Ind (\Coh_{\Dc, Z}).
\]
For the derived $\infty$-categories we have identification (see \cite{neeman}):
\[
D(\QCoh_Z) \,\,\simeq \,\, \Ind(\Perf_Z), \quad D(\QCoh_{\Dc, Z}) \,\,\simeq \,\, \Ind(\Perf_{\Dc, Z}).
\]
where $\Perf_Z$ is the category of perfect complexes of $\Oc_Z$-modules.

Therefore the Verdier duality gives an anti-equivalence
\[
D(\QCoh_{\Dc, Z}) \overset\sim\lra \Pro(\Perf_{\Dc, Z}), \quad \Mc = \ilim \Mc_\nu \,\mapsto \, \Mc^\vee  = \plim   \Mc_\nu^\vee. 
\]

\paragraph{Formal inverse image $f^{[\![*]\!]}$.}  We keep the notation of the previous section. 

Define the functor of {\em formal inverse image} 
\[
f^{[\![*]\!]}: \Pro (\Perf_{\Dc, W}) \lra  \Pro(\Perf_{\Dc, Z})
\]
by putting, for  $\Mc^\bullet\in \Perf_{\Dc, W}$
\[
f^{[\![*]\!]}\Mc \,=\, \bigl(f^!(\Mc^\vee)\bigr)^\vee 
\]
and then extend to pro-objects in a standard way. If $Z$ and $W$ are smooth, then 

\be\label{eq;F^*-Hom}
f^{[\![*]\!]}\Mc \,=\,\, \ul\RHom_{f^{-1}\Dc_W} (\Dc_{Z\to W}, f^{-1}\Mc),
\ee
where  we notice that $\Dc_{Z\to W}$ is a quasi-coherent, i.e., ind-coherent right $\Dc_Z$-module,
so taking $\ul\RHom$ from it produces a pro-object. 

\begin{ex}\label{ex:f^*=completion}
Let $f: \{w\}\hra W$ be the embedding of a $\k$-point. Then $\Perf_{\Dc,\{w\}}$
is the derived ($\oo$-)category of finite complexes of finite-dimensional
$\k$-vector spaces. The category $\Pro(\Perf_{\Dc,\{w\}})$ is
the derived ($\oo$-)category of complexes of pro-finite-dimensional
(or, equivalently, linearly compact linearly topological) $\k$-vector spaces. 
Given a coherent sheaf $E$ on $W$, 
\[
f^{[\![*]\!]}(E\otimes_{\Oc_W} \Dc_W) \,=\, \wh E_w \,=\, \varprojlim\nolimits_n E\otimes (\Oc_W/I_w^n)
\]
is the formal completion of $E$ at $w$. This follows  from
\eqref{eq;F^*-Hom} once we notice that $\Dc_{\{w\}\to W} = \delta_w\cdot\Dc_W$,
the $\Dc$-module of distributions supported at $w$,
 is the topological dual of $\wh\Oc_{W,w}$. For a more general situation
 see Example \ref{exas:diff-bundles}(c). 
\end{ex}

\begin{prop}~
    \begin{statements}
        \item\label{ass:f[[*]]leftadjoint} Suppose $f$ is proper, so that $f_*$ takes $\Perf_{\Dc, Z}$ to $\Perf_{\Dc, W}$ and therefore extends to a functor
        \[
        f_*: \Pro(\Perf_{\Dc, Z})\lra \Pro(\Perf_{\Dc, W})
        \]
        denoted by the same symbol. Then
        the functor $f^{[\![*]\!]}$ is left adjoint to $f_*$ thus defined.
        \item\label{ass:f[[*]]holonomic} The functor $f^{[\![*]\!]}$ takes $D^b_\hol(\Coh_{\Dc,W})$ to $D^b_\hol(\Coh_{\Dc, Z})$
        (no pro-objects needed). 
    \end{statements}
\end{prop}

We note that defining the $*$-inverse image on holonomic $\Dc$-modules by conjugating $f^!$ with the  Verdier duality
is a standard procedure. The corresponding functor is usually denoted by $f^*$, see \cite{borel}. We use the notation
$f^{[\![*]\!]}$ to emphasize the pro-object structure in the general (non-holonomic) case. 

\begin{proof}
    \ref{ass:f[[*]]leftadjoint} In the case when $Z$ and $W$ are smooth, the statement follows from 
    \eqref{eq;F^*-Hom} and from the adjunction between $\Hom$ and $\otimes$. The general case is dual to \autoref{prop:propertiespushforward} \ref{ass:f*b}.
    
    \vskip .2cm
    
    Part \ref{ass:f[[*]]holonomic} is standard, see \cite{borel}.
\end{proof}


\paragraph{Compatibility of $f^{[\![*]\!]}$ with $\DR$ on complex topology.} 
Let $\k=\CC$. 
We extend the definition of $\DR$ to $\Pro(\Perf_{\Dc, W})$ using Verdier duality. By \autoref{prop:deRham-duality} (b), this functor coincides with the previously defined one on perfect $\Dc$-modules we regular holonomic cohomology.

\begin{prop}
    Let $\Mc^\bullet \in \Perf_{\Dc, W}$ be a complex with holonomic regular cohomology. 
    Then
    \[
    \DR(f^{[\![*]\!]}\Mc)_\an \,\,\simeq \,\, f^{-1} \bigl( \DR(\Mc^\bullet)_\an\bigr),
    \]
    where on the right we have the usual inverse image of constructible complexes on the complex topology. 
\end{prop}

\begin{proof}
    Follows from \autoref{prop:f^!-an} by Verdier duality.
\end{proof}

Because of these compatibilities we use the same notation $f^{[\![*]\!]}$ for the formal inverse image functor
on $\Dc$-modules and differential sheaves. 

\paragraph{The formal compactly supported direct image $f_{[\![!]\!]}$.}
We define the functor  of {\em formal compactly supported direct image}
\[
f_{[\![!]\!]}: \Pro(\Perf_{\Dc, Z}) \lra  \Pro(\Perf_{\Dc, W}) 
\]
by putting, for $\Mc^\bullet\in \Perf_{\Dc, Z}$
\[
f_{[\![!]\!]}\Mc \,=\, (f_* \Mc^\vee)^ \vee \,\,\in \,\,  \Pro(\Perf_{\Dc, W}) 
\]
and then extending to pro-objects in the standard way. 
In the particular case when $f=p: Z\to \pt$ is the projection to the point, we will
use the notation
\[
R\Gamma_\DR^{[\![\c]\!]}(Z, \Mc) \,\,:=\,\, p_{[\![!]\!]}\Mc \,\,\in \,\, \Pro(\Perf_\k). 
\]

\begin{prop}~
    \begin{statements}
        \item If $f$ is proper then $f_{[\![!]\!]} \simeq f_*$ and thus $f_{[\![!]\!]}$ is right adjoint to $f^{[\![*]\!]}$.
        If $f$ is \'etale then $f_{[\![!]\!]}$ is left adjoint to $f^{[\![*]\!]}$.
        \item For any Cartesian square of varieties as in \autoref{prop:base1}\ref{ass:base1b}, we have a canonical identification (base change)
        \[
        f_2^{[\![*]\!]} \circ (f_1)_{[\![!]\!]} \,\,\simeq \,\, (g_1)_{[\![!]\!]} \circ g_2^{[\![*]\!]}. 
        \]
        \item Let $j: Z\hookrightarrow W$ be an open embedding, with $i: K\hookrightarrow W$ be the closed embedding of the complement.
        Then
        \begin{assertions}[label=$\thestatementsi_\arabic*$]
         \item For any $\Mc \in \Perf_{\Dc, W}$, we have $j^{[\![*]\!]}\Mc \simeq j^!\Mc \in \Perf_{\Dc, Z}$;
         \item For any $\Mc' \in \Perf_{\Dc, K}$, we have $i_{[\![!]\!]}\Mc' \simeq i_*\Mc' \in \Perf_{\Dc, W}$;
         \item For any $\Nc\in \Pro(\Perf_{\Dc, W})$ we have the canonical triangle
        \[
        j_{[\![!]\!]} j^{[\![*]\!]} \Nc \lra \Nc \lra  i_{[\![!]\!]} i^{[\![*]\!]}\Nc \lra j_{[\![!]\!]} j^{[\![*]\!]}\Nc[1].
        \]
        \end{assertions}
        \item The functor $f_{[\![!]\!]}$ takes $D^b_\hol(\Coh_{\Dc,Z})$ to $D^b_\hol(\Coh_{\Dc, W})$
        (no pro-objects are needed).
    \end{statements}
\end{prop}

As with $f^{[\![*]\!]}$, we note that defining the $!$-direct image on holonomic $\Dc$-modules by conjugating $f_*$ with the Verdier duality
is a standard procedure. The corresponding functor is usually denoted by $f_!$, see \cite{borel}. We use the notation
$f_{[\![!]\!]}$ to emphasize the pro-object structure in the general (non-holonomic) case. 

\begin{proof}
    Follows from the corresponding properties for $f^!$ and $f_*$ of \autoref{prop:base1} by applying Verdier duality.
\end{proof}

\paragraph{Compatibility of $f_{[\![!]\!]}$ with $\DR$ on complex topology.}

\begin{prop}\label{prop:Grothendieck[[!]]}
    Let $\k=\CC$ and let $\Mc^\bullet \in \Perf_{\Dc, Z}$ be a complex of right $\Dc_Z$-modules
    whose cohomology modules are holonomic regular. Then we have a quasi-isomorphism
    \[
    \DR(f_{[\![!]\!]} \Mc^\bullet)_\an\,\,\simeq \,\, f_! (\DR(\Mc^\bullet)_\an),
    \]
    where $f_!$ is the usual functor of direct image with proper support for sheaves on the complex topology. 
\end{prop}

\begin{proof}
    Follows from the similar statement about the functor $f_*$ by applying Verdier duality.
\end{proof}

\paragraph{Compatibility with $\DR$ on the Zariski topology.} 
Let $Z$ be a smooth variety over $\k$. The category of induced $\Dc_Z$-modules
is equivalent to the category formed by coherent $\Oc_Z$-modules and
differential operators  between them. More precisely, we an
 identification of Zariski sheaves
\be\label{eq:diff(E,F)}
\ul\Hom_{\Dc_Z}(E\otimes \Dc_Z, F\otimes \Dc_Z) \, \= \,  \Diff(E,F), \quad
E,F\in\Coh_{\Oc_Z},
\ee
compatible with \eqref{eq:DR-Induced}. 

More generally, call a $\Dc_Z$-module   {\em locally induced} if it is locally on $Z_\Zar$
isomorphic to an induced $\Dc_Z$-module. By the above, the category of locally induced
$\Dc_Z$-modules is identified (via the functor $\DR$) with the category of objects which we call
{\em differential sheaves}. They are
 Zariski sheaves glued out of coherent $\Oc_Z$-modules with
transition functions being invertible differential operators.
They form a category whose morphisms 
  are {\em differential operators between differential sheaves},
a concept that makes sense in terms of gluing.

 Replacing here ``coherent
$\Oc_Z$-modules'' with ``vector bundles'', we obtain the concept of a
{\em differential bundle}. The category of differential bundles  and differential operators
between them is identified
with that of $\Dc_Z$-modules which are locally free (of finite rank).

Let us illustrate the effect
of our functorialities on induced $\Dc_Z$-modules is some cases. Doing so amounts to
describing their compatibility with the functor $\DR$
in the induced situation. 
This  will not be formally used in the sequel, so we omit the (elementary) proofs. 
We assume that $f: Z\to W$ is a morphism of smooth varieties.

\begin{exas}\label{exas:diff-bundles}
(a) The functor $f_*$   on induced $\Dc$-modules corresponds to the usual  Zariski sheaf-theoretic derived direct image of coherent $\Oc$-modules:
        \[
         f_*(E \otimes_{\Oc_Z}\Dc_Z) \,\,\simeq \,\, Rf_\bullet(E) \otimes_{\Oc_W}\Dc_W
        \quad E\in \Coh_{\Oc_Z}.
        \]
        Here $Rf_\bullet$ is the topological direct image of sheaves on the Zariski topology,
        so that $Rf_\bullet(E)$ is a complex of quasi-coherent 
        (i.e., ind-coherent) $\Oc_W$-modules.

\vskip .2cm

(b) The functor $f^!$ 
 for a closed embedding  of  codimension $d$ corresponds
to the usual 
cohomology with support $\ul H^\bullet_Z(-)$ on Zariski topology applied to coherent
$\Oc_W$-modules. Assume for simplicity that $E$ is 
a vector
bundle on $W$. Then the sheaf $\ul H^j_Z(E)$ is zero for $j\neq d$, and
$\ul H^d_Z(E)$ has a natural  filtration $F_n$ by ``order of poles''
with $F_n/F_{n-1}$ canonically identified with the vector bundle
$E|_Z \otimes \det N_{Z/W} \otimes S^n N_{Z/W}$. The $F_n$ themselves,
while not $\Oc_Z$-modules,  are naturally
differential bundles on $Z$, so $\ul H^d_Z(E)$  is an ind-differential bundle. 
The quasi-coherent $\Dc_Z$-module
$f^! (E\otimes_{\Oc_W}\Dc_W)$ is the inductive limit of the locally free $\Dc_Z$-modules
corresponding to the $F_n$ (and shifted by $d$ in the derived category).

\vskip .2cm

(c) The functor $f^{[\![*]\!]}$ for a closed embedding corresponds to
formal completion of coherent $\Oc_W$-modules. Given (for simplicity) a vector bundle
$E$ on $W$, the formal completion $\wh E_Z$ is the projective limit of the sheaves
$E\otimes(\Oc_W/I_Z^n)$ which, while not $\Oc_Z$-modules, 
 are naturally differential bundles
on $Z$, so  $\wh E_Z$ is a pro-differential bundle. The pro-object  $f^{[\![*]\!]}(E\otimes_{\Oc_W} \Dc_W)$ is the formal projective limit of the locally free
$\Dc_Z$-modules corresponding to the $E\otimes(\Oc_W/I_Z^n)$. 
This generalizes Example \ref{ex:f^*=completion}.

\vskip .2cm

(d) The functor $f_{[\![!]\!]}$ corresponds to the algebraic
direct image with proper supports as defined by Deligne 
\cite{deligne, hartshorne-annalen}, 
 for coherent $\Oc_Z$-modules.   That is, we compactify $f$ by factoring it as
 $Z \buildrel j\over \hra \ol Z \buildrel \ol f\over\to W$ with $j$ open,  $\ol f$
 proper, $\ol Z$ smooth and and  $D=\ol Z\- Z$ a divisor. Then for a coherent $\Oc_Z$-module  $E$ 
Deligne defines 
\[
  f_!^{\on{alg}}(E)\,=\, \, \text{``}\varprojlim_n\!\text{''} \, R\ol f_* (E(-nD)), 
\]
a pro-object in the  classical derived category of coherent $\Oc_W$-modules. It is
 easily upgraded to a pro-object in the $\oo$-categorical enhancement. 
 The pro-object $  f_{[\![!]\!]}(E\otimes_{\Oc_Z} \Dc_Z)$ is identified with
  the formal projective
 limit of the complexes of induced $\Dc_W$-modules corresponding to the
 $ R\ol f_* (E(-nD))$. 
 
\end{exas}


\subsection{Correspondences and base change}\label{subsec:corr}

\begin{Defi}
    We denote by $\Var_\k^\corr$ the following $2$-category, called the category of correspondences between $\k$-varieties.
    Its objects are objects of $\Var_\k$. A morphism in $\Var_\k^\corr$ from $X$ to $Y$ is a correspondence: a third variety $Z$ with two morphisms of varieties $X \leftarrow Z \to Y$.
    The composition of two correspondences $X \leftarrow Z \to Y$ and $Y \leftarrow Z' \to Y'$ is the correspondence $X \leftarrow Z \times_Y Z' \to Y'$.
    Finally, a transformation (i.e.,  a $2$-morphism) between two correspondences $X \leftarrow Z \to Y$ and $X \leftarrow Z' \to Y$ is the datum of a \emph{proper} morphism $Z \to Z'$ commuting with the maps to $X$ and $Y$. The vertical composition is the obvious one.
\end{Defi}

It follows from \cite[Vol.\@ II Chap.\@ 4 Thm.\@ 2.1.2]{GR} that base change between lower-$*$ and upper-! functors can be encoded as an $(\infty,2)$-functor
\[
\Dcorr \colon \Var_\k^{\on{corr}} \lra \Cat_\infty
\]
mapping a variety $X$ to $D(\QCoh_{\Dc,X})$ and a correspondence $X 
\overset{f}\leftarrow Z \overset{g}\to Y$ to the functor $g_*f^!$. It finally 
maps a (proper) transformation $f \colon Z \to Z'$ between two correspondences 
$X \overset{g}\leftarrow Z \overset{h}\to Y$ and $X \overset u\leftarrow Z' 
\overset v \to Y$ to the natural transformation
\[
h_* g^! = v_* f_* f^! u^! \Rightarrow v_* u^!
\]
induced by the adjunction counit.

Similarly, we have an $(\infty,2)$-functor
\[
\Dc^{[\![\corr]\!]} \colon \Var_\k^\corr \lra \Cat_\infty
\]
that maps a variety $X$ to $\Pro(\Perf_{\Dc,X})$, a correspondence  $X 
\overset{f}\leftarrow Z \overset{g}\to Y$ to the functor $g_{[\![!]\!]}f^{[\![*]\!]}$. It finally 
maps a (proper) transformation $f \colon Z \to Z'$ between two correspondences 
$X \overset{g}\leftarrow Z \overset{h}\to Y$ and $X \overset u\leftarrow Z' 
\overset v \to Y$ to the natural transformation
\[
h_{[\![!]\!]} g^{[\![*]\!]} = v_{[\![!]\!]} f_{[\![!]\!]} f^{[\![*]\!]} u^{[\![*]\!]} \Rightarrow v_{[\![!]\!]} u^{[\![*]\!]}
\]
induced by the adjunction counit.


\section{\texorpdfstring{$\Dc$}{D}-modules on the Ran space}\label{sec:Ran}

\subsection{The Ran space in algebraic geometry}\label{subsec:Ran-AG}

Throughout this section, we will denote by $\Var_\k$ the category of varieties over $\k$. Let $X \in \Var_\k$.

\paragraph{Ran diagram and Ran space.}

\begin{Defi}~
    \begin{statements}
        \item Let $\finsurj$ denote the category of non-empty finite sets with surjective maps between them.
        We define the {\em Ran diagram} of $X$ as the contravariant functor $X^\finsurj \colon \finsurj \to \Var_\k$ defined by:
        \[
        X^\finsurj \colon I \mapsto X^I, \quad (g \colon I \twoheadrightarrow J) \,\mapsto \,  (\delta_g  \colon X^J \to X^I),
        \]
        where $\delta_g$ is the diagonal embedding  associated to $g$. 
        
        \item By an {\em  (algebro-geometric) space}  over $\k$ we will mean a sheaf of sets on the big \'etale site (affine $\k$-schemes with \'etale
        coverings). The category of such will be denoted $\AGS$. We have the standard embedding $\Var_\k\hookrightarrow  \AGS$ (representable functors).
        
        \item The {\em Ran space} of $X$ is defined as the colimit in $\AGS$:
        \[
        \Ran(X) \,\,=\,\,\varinjlim X^\finsurj \,\,=\,\, \varinjlim _{I\in\finsurj} X^I.
        \]
    \end{statements}
\end{Defi}

\noindent The category $\finsurj$ is not filtering so $\Ran(X)$ is not an ind-variety in the standard sense.

\paragraph{Diagonal skeleta.} 

\begin{Defi}\label{def:skel}~
    \begin{statements}
        \item For any   $I\in \finsurj $ and  $q>0$, we call the {\em $q$-fold diagonal} and  denote by
        \[
        X^I_{q} := \bigcup_{\substack{f \colon I \twoheadrightarrow Q\\ |Q| = q}} \delta_f(X^Q)
        \]
        the closed subvariety of $X^I$ whose closed points are families of at most $q$ different points of $X$.   
        \item 
        We denote by $X^\finsurj_{q}$ the functor $\finsurj \to \Var_\k$ given by $I \mapsto X^I_{ q}$.
        We also denote 
        \[
        \Ran_{q}(X)\,\, =\,\, \varinjlim \, X^\finsurj_{q}\,\,=\,\, \varinjlim_{I\in\finsurj}\,  X^I_q
        \] the space corresponding to $X^\finsurj_{ q}$. 
    \end{statements}
\end{Defi}

Let  $S_q$ be  the symmetric group on $q$ symbols . We denote by
\[
\Sym^q(X) \,\,=\,\, X^q/S_q, \quad \Sym^q_{\neq}(X) = (X^q-X^q_{q-1})/S_q
\]
the $q^\mathrm{th}$ symmetric power of $X$ (as a singular variety) and its open part given by complement of all the diagonals. 

\begin{prop}~
    \begin{statements}
        \item We have
        \[
        X^\finsurj = \varinjlim_q \, X^\finsurj_q, \quad \Ran(X) = \varinjlim_q \, \Ran_q(X)
        \]
        (colimit in the category of functors resp. in $\AGS$). 
        
        \item We have an identification
        \[
        \Ran_q(X) - \Ran_{q-1}(X) \,\,=\,\,\Sym^q_\neq (X). 
        \]
        In particular, $\Ran_1(X)=X$.
    \end{statements}
\end{prop}

\paragraph{Varieties $\Delta(I,J)$.} We will use the following construction 
from \cite{GL} (9.4.12). 

\begin{Defi}
    Let $I,J$ be two nonempty finite sets. We denote by $\Delta(I,J)\subset X^I\times X^J$
    the closed algebraic subvariety whose $\ol \k$-points are pairs of tuples $\bigl( (x_i)_{i\in I}, (y_j)_{j\in J}\bigr)$,
    $x_i, y_j\in X(\ol\k)$ such that the corresponding unordered subsets
    \[
    \{x_i\}_{i\in I} \,=\, \bigcup_{i\in I}\,  \{x_i\}, \quad \{y_j\}_{j\in J} \,=\, \bigcup_{j\in J}\, \{y_j\}
    \]
    coincide. We denote 
    \be\label{eq:p-q-IJ}
    X^I \buildrel p_{IJ}\over\lla \Delta(I,J)\buildrel q_{IJ}\over\lra X^J
    \ee
    the natural projections. They are finite morphisms. 
\end{Defi}

The following is obvious, 

\begin{prop}\label{prop:Delta-colim}
    We have 
    \[
    \Delta(I,J) \,=\,\varinjlim{}^\Var_{\{ I\buildrel u \over \twoheadrightarrow Q \buildrel v\over\twoheadleftarrow J\}} 
    \Im \bigl\{ (\delta_u, \delta_v): X^Q \lra X^I\times X^J\bigr\},
    \]
    where the colimit is taken over the category whose objects are pairs of surjections $I\buildrel u\over\dra Q\buildrel v\over\dla J$
    and morphisms are surjections $Q \twoheadrightarrow Q'$ commuting with the arrows. The colimit reduces to the union inside $X^I\times X^J$. 
    \qed
\end{prop}

Any surjection $g: I\twoheadrightarrow J$ induces, for each finite nonempty $T$,  a natural morphism
(closed embedding)  $\Delta(g,T): \Delta(J,T)\to\Delta(I,T)$. The following is clear by definition. 

\begin{prop}\label{prop:Delta-cart}
    For any $f: I\twoheadrightarrow J$ and $T$ as above we have a commutative diagram with the square being Cartesian
    \[
    \xymatrix{
        X^I 
        & \ar[l]_{p_{IT}} \Delta(I,T) \ar[r]^{q_{IT}}& X^T 
        \\
        X^J
        \ar[u]^{\delta_g}
        & \ar[l]^{p_{JT}} \Delta(J,T). 
        \ar[u]^{\Delta_{f,T}}
        \ar[ur]_{q_{JT}}& 
    }
    \]
    \qed
\end{prop}

\begin{cor}\label{cor:Delta-natural-transformation}
 Consider the functors $F_1, F_2 \colon \finsurj \times \finsurj^\op \to \Var_\k^{\on{corr}}$:
 \begin{align*}
  F_2 &\colon (S,I) \mapsto X^I ,\hspace{1em} (g,I \overset{f}{\twoheadrightarrow} J) \mapsto (X^J \overset{=}{\leftarrow} X^J \overset{\delta_f}{\to} X^I)
  \\
  F_1 &\colon  (S,I)  \mapsto X^S ,\hspace{1em} (S \overset{g}{\twoheadrightarrow} T,f) \mapsto (X^S \overset{\delta_g}{\leftarrow} X^T \overset{=}{\to} X^T).
 \end{align*}
 The data of the varieties $\Delta(I,S)$ assembles into a natural transformation
 \[
  \underline \Delta \colon F_2 \Rightarrow F_1.
 \]
\end{cor}

\begin{proof}
 The natural transformation $\underline \Delta$ maps a pair $(S,I) \in \finsurj \times \finsurj^\op$ to the correspondence $X^I \leftarrow \Delta(I,S) \to X^S$. For $f \colon I \twoheadrightarrow J$ and $g \colon S \twoheadrightarrow T$, we have to show the corresponding diagram in $\Var_\k^{\on{corr}}$ commutes. This follows from the following commutative diagram, where both squares (a) and (b) are pullback squares:
 \[
  \begin{tikzcd}
   X^T & X^T \ar{l}[swap]{=} \ar{r}{\delta_g} & X^S \\
   \Delta(I,T) \ar{u}{p_{IT}} \ar{d}[swap]{q_{IT}} & \Delta(J,T) \ar{u}{p_{JT}} \ar{d}{q_{JT}} \ar{l}[swap]{\Delta_{f,T}} \ar{r}[swap]{\Delta_{J,g}} \ar[phantom, start anchor=center, end anchor=center]{ur}{\mathrm{(a)}} \ar[phantom, start anchor=center, end anchor=center]{dl}{\mathrm{(b)}} & \Delta(J,S) \ar{u}[swap]{p_{JS}} \ar{d}{q_{JS}} \\
   X^I & X^J \ar{l}{\delta_f} \ar{r}[swap]{=} & X^J.
  \end{tikzcd}
 \]
\end{proof}

An alternative way to arrive at \autoref{prop:Delta-cart} is via the 
next statement. 

\begin {prop}
The natural morphism in $\AGS$
\[
(p_{IJ}, q_{IJ}): \Delta(I,J) \lra X^I \times_{\Ran(X)} X^J
\]
is an isomorphism. 
\end{prop}
\begin{proof}
    To say that $(p_{IJ}, q_{IJ})$ is an isomorphism of sheaves on the big \'etale site,
    means that it induces a bijection on $S$-points for any scheme $S$ which is the spectrum of a strictly Henselian
    local ring. Let $S$ be given and $(p,q): S\to  X^I \times_{\Ran(X)} X^J$ be a morphism, i.e., $p: S\to X^I$
    and $q: S\to X^J$ are morphisms of schemes which become equal after map to $\Ran(X)$. We need to show
    that $(p,q)$ considered as a morphism $S\to X^I\times X^J$ factors through $\Delta(I,J)$.
    
    By definition of $\Ran(X)$ and our assumptions on $S$, 
    \[
    \Hom(S, \Ran(X)) \,=\,\varinjlim{}^\Set_{I\in\finsurj} \, \Hom(S, X^I)\,\,=\,\,\bigsqcup_{I\in\finsurj} \Hom(S, X^I) \biggl/ \equiv
    \]
    where $\equiv$ is the equivalence relation generated by the following relation $\equiv_0$.
    We say that $p: S\to X^I$ and $q: S\to X^J$ satisfy $p\equiv_0 q$, 
    if there is a diagram of surjections $I\buildrel a\over\dla L \buildrel b\over\dra J$
    such that $\delta_a p = \delta_b q$  in $\Hom(S,  X^L)$.
    
    We treat the case $p\equiv_0 q$, the general case follows easily.  Let $Q$ be the coproduct, fitting in the coCartesian square  of sets on the left:
    \[
    \xymatrix{
        L\ar@{->>}[d]_b
        \ar@{->>}[r]^a &I \ar@{->>}[d]^c
        \\
        J\ar@{->>}[r]_d&Q
    }
    \quad\quad\quad\quad
    \xymatrix{
        X^L& \ar[l]_{\delta_a} X^I
        \\
        X^J \ar[u]^{\delta_b}&\ar[l]^{\delta_d} X^Q
        \ar[u]_{\delta_c}
    }
    \]
    This square induces a Cartesian square of varieties (above right). So $p,q$ comes from a morphism $S\to X^Q$
    and our statement follows from \autoref {prop:Delta-colim}.
\end{proof}


\subsection{\texorpdfstring{$[\![\Dc]\!]$}{[[D]]}-modules and \texorpdfstring{$\Dc^!$}{D!}-modules}\label{sec:laxmodules}

\paragraph{Reminder on lax limits.} \label{par:laxlim}
Let $\scrA$ be a small $\infty$-category and $\Pc\colon \scrA \to\Cat_\infty$ an $\infty$-functor. In particular, for each object  $a\in\scrA$ we have an $\infty$-category $\Pc_a$ and for any morphism $g\colon a\to a'$ in $\scrA$ we have  an $\infty$-functor $p_g\colon \Pc_a\to\Pc_{a'}$. In this setting we have the $\infty$-category known as the {\em Cartesian Grothendieck construction}
(or {\em relative nerve} or \emph{unstraightening}) $\Groth(\Pc)$. See \cite{lurie-htt} Def. 3.2.5.2. Thus, in particular, 
\begin{itemize}
    \item[($\widebar\GR0$)] Objects ($0$-simplices) of $\Groth(\Pc)$  are pairs
    $(a,x)$ where $a\in \Ob(\scrA)$ and $x$ is an object ($0$-simplex) in $\Pc_a$. 
    \item[($\widebar\GR1$)] Morphisms ($1$-simplices) in $\Groth(\Pc)$ 
    from $(a,x)$ to $(a', x')$ are pairs $(g,\alpha)$ where $g\colon a'\to a$ is a morphism in $\scrA$ and
    $\alpha \colon x \to p_g(x')$ is a morphism in $\Pc_{a}$,  
\end{itemize}   
and so on, see {\em loc. cit.} for details.   Note that there is a ``dual" version, called the coCartesian Grothendieck construction $\Grothco(\Pc)$ with the same objects but with (higher) morphisms defined in a partially dualized way, for example,
\begin{itemize}
    \item[(${\GR1}$)] Morphisms ($1$-simplices) in $\Grothco(\Pc)$ 
    from $(a,x)$ to $(a', x')$ are  pairs $(g,\alpha)$ where $g\colon a\to a'$ is a morphism in $\scrA$ and
    $\alpha\colon p_g(x) \to x'$ is a morphism in $\Pc_{a'}$, 
\end{itemize}
and so on. In other words,
\[
\Grothco(\Pc) \,=\,\Groth(\Pc^\circ)^\op, \quad \Pc^\circ = (\Pc_a^\op, \, 
p_g^\op \colon \Pc_a^\op\to \Pc_{a'}^\op). 
\]
We have the natural projections
\[
\ol q \colon \Groth(\Pc) \lra \scrA^\op,
\quad
 q \colon \Grothco(\Pc) \lra \scrA.
\]

\begin{Defi}\label{def:lax-oplax-lim}
    The {\em lax limit}  and  {\em op-lax limit} of $\Pc$ are the $\infty$-categories
    \[
    \laxlim (\Pc) \,\,=\,\,\Sect(\Grothco(\Pc)/\scrA), \quad \oplaxlim(\Pc) \,\,=\,\,\Sect(\Groth(\Pc)/\scrA)
    \]
    formed by sections of $\ol q$ and $ q$, i.e., by $\infty$-functors 
    (morphisms of simplicial sets)
    $s\colon \scrA \to \Grothco(\Pc)$ (resp. $\ol s \colon \scrA^\op \to \Groth(\Pc)$)
    such that $q s=\Id$ (resp. $\ol q \ol s=\Id$). Thus  
    $\laxlim(\Pc) \,=\, \oplaxlim(\Pc^\circ)^\op$. 
\end{Defi}

\begin{rem}
 For a more intrinsic approach to lax and oplax limits in the context of $\infty$-categories, see \cite{ghn-lax-colim}.
\end{rem}

\begin{exas}~
    \begin{statements}
        \item Thus, an object of $\laxlim(\Pc)$ is a following set of data: 
        \begin{itemize}
            \item[(0)] For each $a\in\Ob(\scrA)$, an object $x_a\in \Pc_a$.
            
            \item[(1)] For each morphism $g: a_0\to a_1$ in $\Ac$, a morphism (not necessarily an equivalence)
            $\gamma_g\colon p_g(x_{a_0})\to x_{a_1}$.
            
            \item[(2)] For each composable pair $a_0\buildrel g_0\over\to a_1\buildrel g_1\over\to a_2$ in $\scrA$,
            a homotopy  (necessarily invertible, as we work in an $(\infty, 1)$-category)
            $\gamma_{g_1} \circ p_{g_1}(\gamma_{g_0})\Rightarrow \gamma_{g_1g_0}$.
            
            \item[(p)] Similar homotopies for composable chains in $\scrA$ of length $p$ for any $p$. 
        \end{itemize}
        
        \item Similarly, an  object of $\oplaxlim(\Pc)$ is a set of data with part (0) identical to the above, part (1)  replaced by morphisms
        $\beta_g\colon x_{a_1} \to p_g(x_{a_0})$ and so on. 
    \end{statements}
\end{exas}

\begin{Defi}
    A morphism $(a_0, x_{a_0}) \to (a_1,x_{a_1})$ in $\Groth(\Pc)$ is called Cartesian if the corresponding morphism $x_{a_0} \to p_g(x_{a_1})$ is an equivalence.
    Dually, a morphism $(a_0, x_{a_0}) \to (a_1,x_{a_1})$ in $\Grothco(\Pc)$ is called coCartesian if the corresponding morphism $p_g(x_{a_0}) \to x_{a_1}$ is an equivalence.
\end{Defi}

In particular, the full sub-$\infty$-category of $\oplaxlim \Pc$ spanned by sections $s$ such that for any map $g$ in $\scrA$, $s(g)$ is Cartesian is equivalent to the limit $\holim \Pc$ (see \cite[Cor. 3.3.3.2]{lurie-htt}).
Dually, the full sub-$\infty$-category of $\laxlim \Pc$ spanned by sections mapping every arrow to a coCartesian one is equivalent to $\holim \Pc$.

Let now $\Qc \colon \scrB \to \scrA$ be another $\infty$-functor. There are pullback diagrams
\[
\xymatrix{
    \Groth(\Pc \circ\Qc) \ar[r] \ar[d] & \Groth(\Pc) \ar[d]
    \\
    \scrB^\op \ar[r]_{\Qc} & \scrA^\op
}
\hspace{1cm}
\xymatrix{
    \Grothco(\Pc \circ\Qc) \ar[r] \ar[d] & \Grothco(\Pc) \ar[d]
    \\
    \scrB \ar[r]_{\Qc} & \scrA.
}
\]
In particular, pulling back sections defines projections
\[
\oplaxlim \Pc \lra \oplaxlim \Pc \circ \Qc \hspace{1cm} \text{and} \hspace{1cm} \laxlim \Pc \lra \laxlim \Pc \circ \Qc,
\]
both compatible with the projection $\holim \Pc \to \holim \Pc \circ \Qc$.

\paragraph{Lax and strict $[\![\Dc]\!]$- and $\Dc^!$-modules on diagrams.}\label{par:modulesondiagrams}
By a {\em diagram} of varieties of finite type over $\k$ we mean a datum of a small category $\scrA$ and
a functor $\Yc\colon \scrA \to \Var_\k$. That is, for each object $a \in \scrA$ we have a (possibly singular) variety $Y_a$ and for each
morphism $g\colon a \to b$ in $\scrA$ we have a morphism of varieties $\xi_g: Y_a\to Y_b$. 

Given a diagram $\Yc$, we have two $\infty$-functors $\scrA^\op\to \Cat_\infty^\st$:
\[
\begin{gathered}
    \Dc^{[\![*]\!]}_\Yc\colon a\mapsto \Pro(\Perf_{\Dc, Y_a}), \quad (g\colon a\to a')\mapsto \xi_g^{[\![*]\!]} 
    \\
    \Dc^!_\Yc\colon a\mapsto D(\QCoh_{\Dc, Y_a}), \quad (g\colon a\to a')\mapsto \xi_g^{!}. 
\end{gathered}
\]

\begin{Defi}\label{def:D-mod-diagram}
We define the $\infty$-categories
    \[
    \Mod_{[\![\Dc]\!]}(\Yc) \,\,=\,\,\laxlim (\Dc^{[\![*]\!]}_\Yc), \quad\quad \Mod_\Dc^!(\Yc) \,\,=\,\,\oplaxlim(\Dc_\Yc^!)
    \]
    whose objects will be called {\em lax $[\![\Dc]\!]$-modules} and {\em lax $\Dc^!$-modules} on $\Yc$. 
\end{Defi}

\begin{rems} \label{rems:*!-D-modules}  ~
    \begin{statements}
        \item Thus, a lax $[\![\Dc]\!]$-module $F$  on $\Yc$ can be viewed as  a family $(F_a)$ where $F_a$ is a 
        pro-coherent complex of left $\Dc$-modules over $Y_a$, together with {\em transition (compatibility)
            maps} 
        \[
        \gamma_g\colon \xi_g^{[\![*]\!]} F_{a'} \lra F_a, 
        \]
        given  for any $g \colon a\to a'$  in $\scrA$ and further compatible
        under compositions of the $g$'s.  Because of the adjunction between $\xi_g^{[\![*]\!]}$ and $(\xi_g)_*$,
        we can write transition maps of a lax $[\![\Dc]\!]$-module in the {\em dual form},
        as morphisms
        \[
        \gamma_g^\dagger\colon F_{a'}  \lra (\xi_g)_* F_a. 
        \]
        Since  $(\xi_g)_*$ preserves coherent $\Dc$-modules, this allows us to deal with some  lax 
        $[\![\Dc]\!]$-modules without using pro-objects. 
        
        \item Similarly, we will view a   lax $\Dc^!$-module  on $\Yc$ as a family $(E^{(a)})$ where $E^{(a)}$ is a quasicoherent 
        (i.e., ind-coherent) complex of 
        right $\Dc$-modules on $Y_a$, together with transition maps  
        \[
        \beta_g\colon E^{(a)}  \lra \xi^!_g  E^{(a')}
        \]
        given  for any   $g \colon a\to a'$ in $\scrA$
        and further compatible under composition of the $g$'s. As before, we can define the structure maps of a lax $\Dc^!-$module in
        {\em the dual form}, as morphisms
        \[
        \beta_g^\dagger\colon (\xi_g)_* E^{(a')} \lra E^{(a)},
        \]
        using the adjunction between  $(\xi_g)_*$ and  $\xi_g^!$. 
    \end{statements}
\end{rems}

\begin{Defi}~
    \begin{statements}
        \item A lax $[\![\Dc]\!]$-module $F$ is called {\em strict}, if all transition maps $\gamma_g$ are equivalences. We denote
        by $\strict{\Mod}_{[\![\Dc]\!]}(\Yc)$ the full $\infty$-category of strict $[\![\Dc]\!]$-modules on $\Yc$.  
        It 
        embeds fully faithfully into  $\LDM(\Yc)$.
        
        \item A lax  $\Dc^!$-module $E$ on $\Yc$ is called {\em strict}, if 
        the transition maps  $\beta_g$ are equivalences. Let 
        $\DsM(\Yc)$ be the $\infty$-category of  $\Dc^!$-modules on $\Yc$. It 
        embeds fully faithfully into $\LDsM(\Yc)$.
    \end{statements}
\end{Defi}

In other words
\[
\strict{\Mod}_{[\![\Dc]\!]}(\Yc) \,=\,\varprojlim \, (\Dc^{[\![*]\!]}_\Yc), \quad\quad \DsM(\Yc) \,=\,
\varprojlim (\Dc_\Yc^!)
\]
are the strict ($\infty$-categorical) limits of the same $\infty$-functors as above, cf. \cite{FG} \S 2.1. 
The strict limit consists of the Cartesian sections of the Grothendieck construction inside
all sections.

\paragraph{$[\![\Dc]\!]$- and $\Dc^!$-modules on the Ran diagram.} 
We now specialize to the case when $\Ac=\finsurj^\op$ and $\Yc = X^\finsurj$ is the Ran diagram. 
Thus, a lax $[\![\Dc]\!]$-module $F$ on $X^\finsurj$ consists of pro-coherent complexes of  $\Dc$-modules $F_I$ on $X^I$, 
morphisms $\gamma_g\colon \delta_g^{[\![*]\!]} F_J \to F_I$ plus coherent higher compatibilities for the $\gamma_g$. 
A lax $\Dc^!$-module $E$ on $X^\finsurj$ consists of quasi-coherent complexes of $\Dc$-modules $E^{(I)}$ on $X^I$,
morphisms $\beta_g\colon E^{(J)} \to \delta_g^! E^{(I)}$ plus  coherent higher compatibilities for the $\beta_g$.

Recall also, the Ran space $\Ran(X) = \varinjlim X^\finsurj$. 
It is clear that strict  modules 
can be defined invariantly in terms of 
the space
$\Ran(X)$, while the concept of  a lax module is tied to the specific diagram $X^\finsurj$
representing $\Ran(X)$. Nevertheless, most of our constructions can and will be performed
directly on $X^\finsurj$.

\paragraph{Strictification of lax $\Dc^!$-modules.} 
We have the full embedding
\[
\DsM(X^\finsurj) \hookrightarrow \LDsM(X^\finsurj).
\]
The left adjoint $\infty$-functor to this embedding will be called the {\em strictification} and will be denoted $E \mapsto \strict{E}$.
Its existence can be guaranteed on general grounds,
see \cite{GL} (5.2). Here we give an explicit formula for it.

\begin{Defi}\label{def:!-strictif}
    Let $E = \bigl( E^{(I)}, \beta_g\colon E^{(J)}\to\delta_g^! E^{(I)}\bigr)$ be  a lax $\Dc^!$-module $E$  on $X^\finsurj$. 
    Its {\em strictification}  is the strict $\Dc^!$-module $\strict{E}$ on $X^\finsurj$ defined by
    \[
    \strict E^{(I)} \,\,=\,\,\hocolim_{K\in\finsurj} (p_{IK})_*\, q_{IK}^! E^{(K)}, 
    \]
    where $p_{IK}$ and $q_{IK}$ are the canonical projections of the variety $\Delta(I,K)$, see
    \eqref{eq:p-q-IJ}. The structure map
    \[
    \strict{\beta}{}_g: \strict{E}^{(J)}\buildrel\simeq \over \lra \delta_g^! \strict{E}^{(I)}, \quad g: I\dra J
    \]
    comes from identification of the target with
    \begin{multline*}
        \hocolim_K \, \delta_g^! (p_{IK})_* q_{IK}^! \, E^{(K)} \,\,\simeq \,\, \hocolim_K \, (p_{JK})_*\,  \Delta(g,K)^! \, q_{IK}^! \, E^{(K)}
        \\
        \simeq \,\, \hocolim_K \,  (p_{JK})_*\, q_{JK}^! \, E^{(K)} \,\,=\,\, \strict{E}^{(J)},
    \end{multline*}
    where we used the base change theorem for the Cartesian square in   \autoref{prop:Delta-cart} as well as the commutativity of the
    triangle there. 
\end{Defi}

\paragraph{Strictification of lax $[\![\Dc]\!]$-modules.} 
The theory here is parallel to the $\Dc^!$-module case.  

\begin{Defi}
    Let $F= \bigl( F_I, \gamma_g: \delta_g^{[\![*]\!]} F_I\to F_J\bigr)$ be a lax $[\![\Dc]\!]$-module on $X^\finsurj$. 
    Its {\em strictification} is the strict $[\![\Dc]\!]$-module $\strict F$ on $X^\finsurj$ defined by
    \[
    \strict {F}{} _I \,=\,\holim_{K\in\finsurj} (p_{IK})_{[\![!]\!]} \,\,q_{IK}^{[\![*]\!]} \,\,  F_K,
    \]
    with the structure map
    \[
    \strict{\gamma}{}_g :  \delta_g^{[\![*]\!]} \strict{F}{}_I \buildrel\simeq \over \lra \strict{F}{}_J, \quad g: I\dra J
    \]
    induced by the base change in the square of \autoref{prop:Delta-cart}
\end{Defi}

The $\infty$-functor $F\mapsto\strict F$ is right adjoint to the embedding 
$\DM(X^\finsurj) \hookrightarrow \LDM(X^\finsurj)$.

\paragraph{Factorization homology.} 
\begin{Defi}
    We call the {\em  factorization homology}  
    (resp.~{\em compactly supported factorization homology}) of a (lax)  $\Dc^!$-module $E$ on $X^\finsurj$
    the complex of $\k$-vector spaces
    \begin{align*}
        &\int_X E\,\, :=\,\, \hocolim_{I\in\finsurj} \,\,  R\Gamma_{\DR}(X^I, E^{(I)})\,\,\in \,\, \Ind(\Perf_\k) = C(\k), 
        \\
        &\int_X^{[\![\c]\!]}  E\,\, :=\,\, \hocolim_{I\in\finsurj} \,\,  R\Gamma_\DR^{[\![\c]\!]}(X^I,  E^{(I)})\,\,\in \,\,\Ind(\Pro(\Perf_\k)). 
    \end{align*}
\end{Defi}

Thus $\int_X E$ is just a (possibly infinite-dimensional) complex of $\k$-vector spaces,
as $\Ind(\Perf_\k) = C(\k)$ is the $\infty$-category of all chain complexes over $\k$. On the other hand,  
$\int_X^{[\![\c]\!]} E$ is an ind-pro-finite-dimensional complex. 

\begin{Defi}
    We call the {\em compactly supported factorization cohomology} (resp.~{\em factorization homology})  of a (lax) $[\![\Dc]\!]$-module $F$ 
    on $X^\finsurj$ the complexes
    \begin{align*}
        &\oint_X^{[\![\c]\!]} F\,\, :=\,\, \holim_{I\in \finsurj} \,\,  R\Gamma_\DR^{[\![\c]\!]} (X^I, F_I) \,\,\in \,\, \Pro(\Perf_\k),
        \\
        &\oint_X F\,\, :=\,\, \holim_{I\in \finsurj} \,\,  R\Gamma_\DR(X^I, F_I) 
        \,\,\in \,\, \Pro(\Ind(\Perf_\k)). 
    \end{align*}
\end{Defi}

Those constructions define exact $\infty$-functors between stable $\infty$-categories:
\begin{align*}
    &\int_X \colon  \LDsM(X^\finsurj)  \lra C(\k),
    \\ &\int_X^{[\![\c]\!]} \colon  \LDsM(X^\finsurj)  \lra \Ind(\Pro(\Perf_\k)),
    \\
    &\oint_X^{[\![\c]\!]} \colon \LDM(X^\finsurj) \lra \Pro(\Perf_\k), 
    \\
    &\oint_X \colon \LDM(X^\finsurj) \lra \Pro(\Ind(\Perf_\k)). 
\end{align*}

\begin{prop}
    All four types of factorization homology are unchanged under strictification, i.e., we have 
    \[
    \begin{gathered}
        \int_X E \,\,\simeq \,\,\int_X  \strict{E}, \quad \quad \int^{[\![\c]\!]}_X E \,\,\simeq \,\,\int^{[\![\c]\!]}_X  \strict{E}, \quad \quad 
        E\in \LDsM(X^\finsurj),
        \\
        \oint_X^{[\![\c]\!]} F \,\,\simeq \,\,\oint^{[\![\c]\!]}_X  \strict{F}, \quad \quad \oint_X F \,\,\simeq \,\,\oint_X  \strict{F}, 
        \quad\quad
        F\in\LDM(X^\finsurj). 
    \end{gathered} 
    \]
    
\end{prop}

\begin{prop}\label{prop:Verdierduality}
    Verdier duality (on each $X^I$)  induces equivalences
    \[
    (-)^\vee \colon \LDsM(X^\finsurj)^\op \simeq \LDM(X^\finsurj) \quad\text{and}\quad 
    \strict{\Mod}^!_\Dc(X^{\finsurj})^\op\simeq \strict{\Mod}_{[\![\Dc]\!]}(X^{\finsurj}), 
    \]
    compatible with the inclusion $\infty$-functors. Moreover, there are  natural equivalences
    \[
    \oint_X^{[\![\c]\!]} (-)^\vee \,\,\simeq\,\,  \left(\int_X -\right)^*, \quad \quad 
    \int_X^{[\![\c]\!]} (-)^\vee \,\,\simeq\,\,  \left(\oint_X -\right)^*.
    \]
\end{prop}


\subsection{Covariant Verdier duality and the diagonal filtration}\label{subsec:coVerdieranddiag}

The content of this section is inspired from \cite{GL}.

\paragraph{For (lax) $\Dc^!$-modules.}
We denote by $\Coh^!_\Dc (X^\finsurj)$   the full sub-$\infty$-category of  $\LDsM(X^\finsurj)$  
spanned by lax $\Dc^!$-modules  $E=(E^{(I)} )$ such that each $E^{(I)} $ is coherent as a complex of $\Dc$-modules over $X^I$ (i.e. belongs to $\Perf_{\Dc, X^I}$).

Since for any $f \colon I \twoheadrightarrow J$, the diagonal embedding $\delta_f \colon X^J \to X^I$ is closed (hence proper), the associated pullback functor $\delta_f^!$ admits a left adjoint $\delta_{f*}$ by \autoref{prop:base1}. We get an $\infty$-functor $\Dc_*^{X^\finsurj} \colon \finsurj^\op \to \Cat_\oo$ such that $\Dc_*^{X^\finsurj}(f) = \delta_{f*}$. We get
\[
 \LDsM(X^\finsurj) := \oplaxlim \Dc^!_{X^\finsurj} \simeq \laxlim \Dc_*^{X^\finsurj}.
\]
Since each $\delta_{f*}$ preserves coherent $\Dc$-modules, $\Dc_*^{X^\finsurj}$ admits a sub-functor $\Dc_*^{\Coh} \colon $ and we find
\[
 \Coh^!_\Dc (X^\finsurj) \simeq \laxlim \Dc_*^{\Coh} \subset \laxlim \Dc_*^{X^\finsurj}.
\]
Similarly, consider the $\infty$-functor $\Dc_{[\![!]\!]}^{X^\finsurj} \colon \finsurj^\op \to \Cat_\oo$ that maps $I$ to the $\infty$-category $\Pro(\Perf_{\Dc, X^I})$ and $f \colon I \twoheadrightarrow J$ to $\delta_{f[\![!]\!]}$. Since each $\delta_f$ is proper, the functor $\delta_{f[\![!]\!]}$ preserves perfect $\Dc$-modules and coincides with $\delta_{f*}$ on such $\Dc$-modules. In particular, $\Dc_*^{\Coh}$ identifies as a sub-functor of $\Dc_{[\![!]\!]}^{X^\finsurj}$. We get
\[
 \Coh^!_\Dc (X^\finsurj) \simeq \laxlim \Dc_*^{\Coh} \subset \laxlim \Dc_{[\![!]\!]}^{X^\finsurj}.
\]

 Using the natural transformation $\underline \Delta$ of \autoref{cor:Delta-natural-transformation}, we find a diagram
\[
\begin{tikzcd}[row sep=tiny]
& \finsurj^\op \ar{dr}
  \ar[bend left=15]{rrrd}{\Dc_{[\![!]\!]}^{X^\finsurj}} \\
 \finsurj \times \finsurj^\op \ar{ur}{\pi_2} \ar{dr}[swap]{\pi_1} &
 \big\Downarrow {\ul\Delta}
 & \Var_\k^{\on{corr}} \ar{rr}
 {\Dc^{[\![\corr]\!]}} && \Cat_\oo. \\
& \finsurj \ar{ur} \ar[bend right=15]{urrr}[swap]{\Dc^{[\![*]\!]}_{X^\finsurj}}
\end{tikzcd}
\]
where $\pi_1$ and $\pi_2$ are the projections. We find a sequence of $\infty$-functors
\begin{equation}\label{eq:building-coVerdier}
 \Coh^!_\Dc (X^\finsurj) \hookrightarrow \laxlim \Dc_{[\![!]\!]}^{X^\finsurj} \to 
\laxlim \left(\Dc_{[\![!]\!]}^{X^\finsurj} \circ \pi_2 \right)  \to \laxlim \left(\Dc^{[\![*]\!]}_{X^\finsurj} \circ \pi_1 \right).
\end{equation}
Since each $\delta_f^{[\![*]\!]}$ is a left adjoint and thus preserves colimits, we have a relative left Kan extension functor (see \cite[\S 10]{shah}) along the projection $\pi_1$:
\[
 \laxlim \left(\Dc^{[\![*]\!]}_{X^\finsurj} \circ \pi_1 \right) \lra \laxlim \Dc^{[\![*]\!]}_{X^\finsurj}.
\]
We now define an $\oo$-functor $\phi$ as the composite
\[
 \Coh^!_\Dc (X^\finsurj) \lra \laxlim \left(\Dc^{[\![*]\!]}_{X^\finsurj} \circ \pi_1 \right) 
 \buildrel  \eqref{eq:building-coVerdier}\over\lra
  \laxlim \Dc^{[\![*]\!]}_{X^\finsurj},
\]
where the second arrow   is the composition of the functors
in Eq.  \eqref{eq:building-coVerdier}.
A more explicit form of the values of $\phi$ on objects is given in the proof of the next lemma.

\begin{lem}
 For any $E \in \Coh^!_\Dc (X^\finsurj)$, we have $\phi(E) \in \strict{\Mod}_{[\![\Dc]\!]}(X^\finsurj)$.
\end{lem}

\begin{proof}
 We use the notations of \autoref{prop:Delta-cart}. Unravelling the above construction, for $E \in \Coh^!_\Dc (X^\finsurj)$ and $I \in \finsurj$, we have
 \[
  \phi(E)_I = \hocolim_{T\in\finsurj} \,  (p_{IT})_{*}\,  q_{IT}^{[\![*]\!]}\, \,  E^{(T)}.
 \]
The transition morphism associated to $g \colon I \twoheadrightarrow J$ is then the composite
 \begin{multline*}
  \delta_f^{[\![*]\!]} (\phi(E)_I) = \delta_f^{[\![*]\!]} \left( \hocolim_{T \in \finsurj}\, (p_{IT})_* \, q_{IT}^{[\![*]\!]}~E^{(T)} \right) \simeq \hocolim_{T \in \finsurj} \delta_f^{[\![*]\!]}\, (p_{IT})_* \, q_{IT}^{[\![*]\!]}~E^{(T)} 
  \\
  \underset{\mathrm{BC}}{\overset\sim\longrightarrow} \hocolim_{T \in \finsurj} (p_{JT})_*\, \Delta_{f,T}^{[\![*]\!]}\, q_{IT}^{[\![*]\!]}(E^{(T)}) \simeq (\phi(E))_J
 \end{multline*}
 where BC is the base change equivalence. The result follows.
\end{proof}

\begin{Defi}
    We call the {\em covariant Verdier duality} the $\infty$-functor defined above
    \[
    \phi:  \Coh^!_\Dc (X^\finsurj) \lra \strict{\Mod}_{[\![\Dc]\!]}(X^\finsurj), \quad 
    (\phi(E))_I = \hocolim_{K\in\finsurj} \,  (p_{IK})_*\,\,  q_{IK}^{[\![*]\!]}\, \,  E^{(K)},
    \]
    where $p_{IK}$ and $q_{IK}$ are the projections of the subvariety $\Delta(I,K)\subset X^I\times X^K$. 
\end{Defi}

\begin{rems}~
    \begin{statements}
        \item We restrict to coherent !-sheaves so that $q_{IK}^{[\![*]\!]} E^{(K)}$ is defined in pro-coherent $\Dc$-modules. Without the coherence assumption, we would have to work with ind-pro- or pro-ind-coherent $\Dc$-modules.
        \item Note the similarity with \autoref {def:!-strictif} of  $\strict{E}$, the strictification of  $E$. 
    \end{statements}
\end{rems}

\paragraph{Topological approximation and diagonal filtration.}
Recall the following general fact.
Let $\Ic, \Kc$ be two small categories and $(A_{IK})_{I\in\Ic,K\in\Kc}$ be a bi-diagram in an $\infty$-category $\scrC$,
i.e. an $\infty$-functor $\Ic \times\Kc\to\scrC$. Then there is a canonical morphism
\be\label{eq:can-twolimits}
\can: \hocolim_{K\in\Kc} \,\holim_{I\in\Ic} \, A_{IK} \,\lra \,  \holim_{I\in\Ic } \, \hocolim_{K\in\Kc}  \, A_{IK}. 
\ee
We apply this to the case when $\Ic = \Kc = \finsurj$ and $\scrC = \Pro(C(k))$.

\vskip .2cm
Let $E\in \Coh^!_\Dc(X^\finsurj)$ be a coherent lax $\Dc^!$-module. We then have  natural maps
\be
\tau: \int_X E \,\, \lra \,\, \, \oint_X \phi(E), \quad\quad  \tau_c: \int_X^{[\![\c]\!]} E \,\, \lra \,\, \, \oint^{[\![\c]\!]}_X \phi(E),
\ee
which we call the {\em topological approximation} maps and which are defined as follows.
Using the standard map
\be\label{eq:map-sigma}
\sigma: R\Gamma_\DR(X^K, E^{(K)}) \lra \,\, \holim_I \, R\Gamma_\DR(\Delta(I,K), q_{IK}^{[\![*]\!]} E^{(K)}), 
\ee
we first map
\[
\int_X E \,\,=\,\,\hocolim_K R\Gamma_\DR(X^K, E^{(K)}) \to \hocolim_K\,  \holim_I\,  R\Gamma_\DR(\Delta(I,K), q_{IK}^{[\![*]\!]} E^{(K)} )
\]
and then map the target by the canonical map \eqref{eq:can-twolimits} to
\begin{align*}
    \holim_I\, &\hocolim_K\, R\Gamma_\DR(\Delta(I,K), q_{IK}^{[\![*]\!]} E^{(K)})
    \\
    & \simeq \,\,  \holim_I\, \hocolim_K\, R\Gamma_\DR(X^I, (p_{IK})_*\,\,  q_{IK}^{[\![*]\!]} E^{(K)})
    \\
    &  \simeq \,\,  \holim_I\, R\Gamma_\DR(X^I, \, \phi(E)_I) \,\,=\,\, \oint_X \phi(E).
\end{align*}
This gives $\tau$. The map $\tau_c$ is defined similarly by noticing that $q_{IK}$ being finite, we have a map
$\sigma_c$  analogous to $\sigma$ and featuring compactly supported de Rham cohomology. 

\vskip .2cm

For any positive integer $d$, we denote by $i_d$ the pointwise closed embedding $X_{d}^\finsurj \to X^\finsurj$,
see \autoref{def:skel}.

\begin{Defi}
    We define
    \begin{align*}
        &\int_X^{\leq d} \, E ~:=~ \holim_I R\Gamma_\DR(X^I_d, i_d^{[\![*]\!]} E^{(I)}) ~ \in ~ \Pro (C(\k)),
        \\ 
        &\int_X^{[\![\c]\!], \leq d} \, E ~:=~ \holim_I R\Gamma_\DR^{[\![\c]\!]}(X^I_d, i_d^{[\![*]\!]} E^{(I)}) ~ \in ~ \Pro (\Perf_\k),
    \end{align*}
    and call them  the {\em factorization homology} (resp. {\em compactly supported factorization homology})
    {\em  of arity at most} $d$ of $E$.  
\end{Defi}

Remark that $\int_X^{\leq d} \, E$ (and similarly for its compactly supported analog) can be seen as the factorization homology of a lax pro-$\Dc^!$-module $(i_d)_* i_d^{[\![*]\!]} E$.

As $d$ varies, they fit into sequences
\[
\xymatrix{
    \displaystyle \int_X E  & \cdots \ar[r] & \displaystyle \int_X^{\leq d} E \ar[r] & \cdots \ar[r] & \displaystyle \int_X^{\leq 1} E,
    \\
    \displaystyle \int_X^{[\![\c]\!]} E & \cdots \ar[r] & \displaystyle \int_X^{[\![\c]\!], \leq d} E \ar[r] & \cdots \ar[r] & \displaystyle \int_X^{[\![\c]\!], \leq 1} E. 
}
\]

\begin{lem}
    For   $E\in\Coh^!_\Dc(X^\finsurj)$, there are  canonical equivalences
    \[
    \oint_X \, \phi(E)\,\, \simeq\,\,  \holim_d  \int^{\leq d}_X E,\quad \quad 
    \oint^{[\![\c]\!]}_X \, \phi(E)\,\, \simeq\,\,  \holim_d  \int^{[\![\c]\!], \leq d}_X E. 
    \]
\end{lem}

\begin{proof}
    This is a straightforward formal computation on limits. 
\end{proof}


\paragraph{For (lax)  $[\![\Dc]\!]$-modules.}
The above constructions admit a dual version, for $[\![\Dc]\!]$-modules. We sketch it briefly, formulating
statements but omitting details. 

\vskip .2cm

We denote by $\Coh_{[\![\Dc]\!]}(X^\finsurj)$ the $\infty$-category of lax $[\![\Dc]\!]$-modules 
\[
F= \left(F_I, (\gamma_g: \delta_g^{[\![*]\!]} F_I\lra F_J)_{ g: 
    I\twoheadrightarrow J}\right)_I
\]
on $X^\finsurj$ such that each $F_I$ is coherent on $X^I$. 

\begin{Defi}\label{def:coVerdier-psi}
    We call the {\em covariant Verdier duality} for $[\![\Dc]\!]$-modules the $\infty$-functor
    \[
    \psi: \Coh_{[\![\Dc]\!]}(X^\finsurj) \lra \DsM(X^\finsurj)
    \]
    defined by 
    \[
    \psi(F)^{(I)} = \holim_K (p_{IK})_*\,\, q_{IK} ^!\,  F_K.
    \]
\end{Defi}
We have canonical {\em topological approximation}  maps
\be
\tau: \int_X \psi(F) \lra \oint_X F,\quad\quad  \tau_c: \int^{[\![\c]\!]}_X \psi(F) \lra \oint^{[\![\c]\!]}_X F.
\ee

For any lax coherent $[\![\Dc]\!]$-module $F$ on $X^\finsurj$, we denote by $\underline{R\Gamma}_{X^\finsurj_d}(F)$ the lax ind-$[\![\Dc]\!]$-module
$(i_d)_* i_d^! F$ on $X^\finsurj$. 

\begin{Defi}
    We define the {\em (compactly supported) factorization cohomology of $F$ with $d$-fold support}  as the 
    (compactly supported) factorization cohomology
    of $\underline{R\Gamma}_ {X_d^\finsurj}(F)$. 
    We denote these cohomologies by
    \begin{align*}
        &R\Gamma_{X_d^\finsurj\!,\,\DR}( X^\finsurj, F)\,\, := \,\, \oint_X \underline{R\Gamma}_{ X_d^\finsurj}(F) \in C(\k),
        \\
        &R\Gamma^{[\![\c]\!]}_{X_d^\finsurj\!,\,\DR}( X^\finsurj, F)\,\, := \,\, \oint^{[\![\c]\!]}_X \underline{R\Gamma}_{ X_d^\finsurj}(F) \in \Ind (\Pro(\Perf_\k)). 
    \end{align*}
\end{Defi}

We call the sequences
\[
\begin{gathered}
    R\Gamma_{X_1^\finsurj\!,\,\DR}( X^\finsurj, F) \to \cdots \to R\Gamma_{X_d^\finsurj\!,\,\DR}( X^\finsurj, F) \to \cdots, 
    \\
    R\Gamma^{[\![\c]\!]}_{X_1^\finsurj\!,\,\DR}( X^\finsurj, F) \to \cdots \to R\Gamma^{[\![\c]\!]}_{X_d^\finsurj\!,\,\DR}( X^\finsurj, F) \to \cdots, 
\end{gathered}
\]
the {\em diagonal filtration}  on the (compactly supported) factorization cohomology of $F$.

\begin{lem}
    There are canonical equivalences
    \begin{align*}
        &\hocolim_d R\Gamma_{X_d^\finsurj\!,\,\DR}( X^\finsurj, F) \simeq \int_X \psi(F),
        \\
        &\hocolim_d R\Gamma^{[\![\c]\!]}_{X_d^\finsurj\!,\,\DR}( X^\finsurj, F) \simeq \int_X^{[\![\c]\!]} \psi(F). 
    \end{align*}
\end{lem}

\paragraph {Compatibility with  the  usual (contravariant) Verdier duality.}

We start with a lax coherent  $\Dc^!$-module
$E$ on $X^\finsurj$. As Verdier duality exchanges $i^{[\![*]\!]}$ and $i^!$ and commutes with $i_*$ for a closed immersion $i$, we get

\begin{prop}
    There are canonical equivalences
    \[
    R\Gamma^{[\![\c]\!]} _{X_d^\finsurj\!,\,\DR}(X^\finsurj, E^\vee) \simeq \left(\int_X^{\leq d} E\right)^*, 
    \quad
    R\Gamma _{X_d^\finsurj\!,\,\DR}(X^\finsurj, E^\vee) \simeq \left(\int_X^{[\![\c]\!], \leq d} E\right)^* .
    \]
    Moreover, it is compatible with the transition maps (when $d$ varies) and with the map to
    \[
    \oint_X^{[\![\c]\!]} E^\vee \simeq \left( \int_X E \right)^*, \quad\quad 
    \oint_X E^\vee \simeq \left( \int^{[\![\c]\!]}_X E \right)^*.
    \]
\end{prop}

 \vfill\eject
\section{Factorization algebras}\label{sec:FA}

The concept of algebro-geometric factorization algebras is due to
  \cite{BD}  for curves  and to \cite{FG} in general. 
  The goal of this chapter is to extend to the factorization context  the covariant Verdier duality introduced in \S \ref {subsec:coVerdieranddiag}. 
For this we need several equivalent definitions of factorization algebras
since different features
become clear in different approaches. Similarly to \cite{BD, FG}, the definitions
are given in terms of commutative algebra objects in certain symmetric
monoidal $\oo$-categories.   The $\oo$-categories
 (formed by lax  $\Dc^!$- or  $[\![\Dc]\!]$-modules
on the Ran space) are certain lax limits of categories associated to finite-dimensional
varieties, and the monoidal structures are given by versions of  the Day convolution,
cf.\cite{glasman} \cite[\S2.2.6]{lurie-ha}. So we start with explaining these
concepts in the form we need.

\subsection{Symmetric monoidal \texorpdfstring{$\infty$}{?}-categories}\label{subsec:sym-mon-inf}

\paragraph{Non-unital symmetric monoidal $\oo$-categories and $\oo$-operads.} 
In this paper we will use symmetric monoidal ($\oo$-) categories
which are not necessarily unital, i.e.,  informally, are not required to 
possess a unit object. 
Although the unital case is  more commonly documented in the literature, cf.  
 \cite{glasman} the general theory of  \cite[Ch.2] {lurie-ha} covers
 the non-unital case as well, as we will explain below, see Remark 
 \ref{rems:unit-nonunit}(b).

\vskip .2cm

Let $\finsurj^{*}$ be the category of pointed finite sets with morphisms being pointed surjections. The base point is denoted $*$.

For a pointed set $J$ we denote $J^o = J\-\{*\}$ the complement of the base
point. A morphism $f: J\to K$ in $\finsurj^+$ is called {\em inert}, if 
it induces a bijection between $f^{-1}(K^o)$ and $K^o$, and {\em active},
if $f^{-1}(*)=*$. 
 
In the opposite direction, let $I$ be any non-empty finite set.  
we denote by $I^{*}$ the set $I \amalg \{*\}$ pointed at $*$. For any $i \in I$, we have the inert map  $p_i: I^* \to \{i\}^*$ given by
\[
p_i(j) = \begin{cases} i &\text{if } j = i \\ * & \text{else.} \end{cases}
\]

A {\em non-unital symmetric monoidal structure}  on an $\infty$-category $\Cc$ is the datum of an $\infty$-functor $\underline \Cc \colon \finsurj^* \to \Cat_\infty$ such that $\underline \Cc(\{1\}^*) \simeq \Cc$ and such that for any $I$, the $\infty$-functor
\[
\underline \Cc\left(\prod p_i\right) \colon \underline \Cc(I^*) \to \prod_{i\in I}\underline \Cc(\{i\}^*)
\]
is an equivalence. 

\vskip .2cm

 In this situation,  the tensor product on $\Cc$ is 
 the $\infty$-functor
\[
\otimes: \Cc \times \Cc \simeq \underline \Cc ( \{1,2\}^* ) \lra \underline \Cc( \{1\}^* ) \simeq \Cc
\] 
induced by the map $\{1,2\} \to \{1\}$ mapping both elements to $1$. 
More generally, for any $n\geq 1$ we have the $n$-fold tensor product functor
$\otimes^n: \Cc^n\to\Cc$ obtained by the value of $\ul\Cc$ on the morphism
$\{1, \dots, n\}^*\to \{1\}^*$ induced by $\{1, \dots, n\}\to \{1\}$. 
 The value of $\ul\Cc$ on
an arbitrary pointed surjection is  then  interpreted in a standard way as ``tensor multiplication over the fibers''.

As is done in \cite{lurie-ha}, we encode symmetric monoidal $\infty$-categories by their Grothendieck constructions:
\begin{align*}
    &\Cc^\otimes := \Grothco(\underline \Cc) \to \finsurj^*,  \\
    &\Cc_\otimes := \Groth(\underline \Cc) \to (\finsurj^*)^\op,
\end{align*}
  a coCartesian fibration over  $\finsurj^*$ and a Cartesian fibration over
$(\finsurj^*)^\op$. 

\vskip .2cm

More generally, see  \cite[Ch.2] {lurie-ha}, an {\em (nonunital) $\oo$-operad} 
  is a 
  {\em partially coCartesian }fibration over $ \finsurj^*$,
i.e., an $\oo$-functor $\pi: \Qc^\otimes \to  \finsurj^*$ which
has coCartesian morphisms covering all the inert morphisms (rather than
all morphisms) of  $ \finsurj^*$. In terms of 
the category $\Qc=\pi^{-1}(\{1\}^*)$ this means that for any finite set of
objects $(q_i)_{i\in I}$   the object $\otimes_{i\in I} q_i$ may not exist
 but we know what should be the functor $\Hom(\otimes_{i\in I} q_i,-)$
 corepresented by this object. 
 
 Dually, an {\em (non-unital) $\oo$-cooperad} is a {\em partially Cartesian}
  (in the same sense as above)
 fibration $\rho: \Rc\to  (\finsurj^*)^\op$. In terms of the category
 $\Rc=\rho^{-1}(\{1\}^*)$ this means that for any finite set of objects
 $(r_i)_{i\in I}$ we know what should be the functor
  $\Hom(-, \otimes_{i\in I} r_i)$ represented by the desired object
  $\otimes_{i\in I}r_i$ without that object necessarily existing. 

Thus a symmetric monoidal $\oo$-category is both an $\oo$-operad and
an $\oo$-cooperad. 

\vskip .2cm

 A {\em  (symmetric) monoidal $\infty$-functor} 
 between two  symmetric monoidal $\infty$-categories $\Cc$ and $\Dc$ is,
 by definition,  a natural transformation $\underline \Cc \to \underline \Dc$.
Equivalently, it corresponds to morphism of coCartesian fibrations
 $\Cc^\otimes \to \Dc^\otimes$, i.e.,  an $\infty$-functor over $\finsurj^*$ that preserves all coCartesian morphisms. 
 Alternatively, it corresponds to a morphism of Cartesian fibrations
 $\Cc_\otimes \to \Dc_\otimes$, i.e., 
  to an $\infty$-functor  over $(\finsurj^*)^\op$ that preserves all Cartesian morphisms.

\begin{Defi}\label{def:lax-mo-functor}
Let $\Cc, \Dc$ be (non-unital) symmetric monoidal categories.
\begin{assertions}
 \item A {\em lax monoidal functor} $\Cc \to \Dc$ is a morphism of $\oo$-operads
   $\Cc^\otimes \to \Dc^\otimes$, i.e.,   
     an $\infty$-functor over $\finsurj^*$ such that any coCartesian morphism in $\Cc^\otimes$ lying over an inert morphism,  is mapped to a coCartesian morphism in $\Dc^\otimes$.

 \item A {\em colax monoidal functor} $\Cc \to \Dc$ is a morphism of $\oo$-cooperads
 $\Cc_\otimes\to\Dc_\otimes$, 
i.e.,  an $\infty$-functor over $(\finsurj^*)^\op$ such that any Cartesian morphism in $\Cc_\otimes$ lying over an inert morphism,  is mapped to a Cartesian morphism in $\Dc_\otimes$.
\end{assertions}
\end{Defi}

It follows from the above definitions that a monoidal $\infty$-functor is in particular both a lax and colax monoidal $\infty$-functor.

 A  lax (resp. colax) monoidal  functor $F \colon \Cc \to \Dc$ is  equipped, in particular, with functorial morphisms
\be\label{eq:lax-mon-fun}
\begin{gathered}
 F_{c_1,\dots, c_n}:    F(c_1) \otimes\cdots\otimes  F(c_n) \lra F(c_1 \otimes 
 \cdots\otimes c_n) \\ 
    \text{(resp. }F^{c_1, \dots, c_n}:   F(c_1 \otimes\cdots\otimes c_n) \lra F(c_1) \otimes\cdots \otimes F(c_n)\text{ ). }
    \end{gathered}
\ee
It also contains coherence (associativity)
data for such morphisms for varying values of $n$. 
It is monoidal if those morphisms are equivalences.

\begin{rems}\label{rems:unit-nonunit}~
\begin{statements}
 \item The more familiar definition   of a
  {\em unital symmetric
monoidal $\oo$-category} uses
  instead of 
  category $\finsurj^{*}$, the category $\scrF^*$  formed by all finite pointed sets
  and all pointed maps (so that the functoriality along non-surjective maps
  encodes insertions of the unit object). One has similar concepts of
  {\em unital $\oo$-operads and cooperads}, see \cite[Ch.2] {lurie-ha}. 
  
 \item The non-unital case is covered by a more general concept from 
    \cite[Ch.2] {lurie-ha}:
  that of a {\em  $\Pc$-monoidal $\oo$-category},
   where $\Pc\to\scrF^*$ is any  unital $\oo$-operad. More
   precisely, the inclusion $\finsurj^*\hra \scrF^*$ makes $\finsurj^*$
   into a unital  $\oo$-operad. A non-unital symmetric monoidal
   $\oo$-category is the same as an $\finsurj^*$-monoidal $\oo$-category. 
 Therefore we can  rely on  \cite[Ch.2] {lurie-ha} to deduce various basic facts in the non-unital case. 
\end{statements}
  \end{rems}

\paragraph{Day convolution products.}\label{par:day-convo}
Fix two (non-unital) symmetric monoidal $\infty$-categories $\Cc$ and $\Dc$. 
In fact, we will need only the case when $\Cc$ is an ordinary category, so we assume this
for simplicity. 

Let $\Fun(\Cc, \Dc)$ be the $\oo$-category of $\oo$-functors $\Cc\to\Dc$.
Under certain conditions, one can define the
  {\em left and right Day convolution products} which are two  (non-unital) symmetric monoidal structures
$\lday$, $\rday$ on  $\Fun(\Cc, \Dc)$.
See \cite{glasman} for the unital case and \cite[\S 2.2.6]{lurie-ha}
for a general treatment that covers our non-unital case. 
Before  making precise
the situation when these monoidal structures exist, let us explain
the issues involved. 

\vskip .2cm

At the  lowest level, 
 on objects
 (i.e., on pairs of $\oo$-functors $F,G:\Cc\to\Dc$), these products are 
 defined  by 
\[
F\lday G\,=\, \otimes_! (F\boxtimes G), \quad F\rday G\,=\, \otimes_* (F\boxtimes G). 
\]
Here $\otimes_!, \otimes_*$ are the left and right Kan extensions with respect to
the functor $\otimes: \Cc\times\Cc \to\Cc$, and $F\boxtimes G: \Cc\times\Cc\to\Dc$
is the external tensor product of $F$ and $G$ 
 which on objects is given by
\[
(F\boxtimes G)(c_1, c_2) \,=\, F(c_1) \otimes G(c_2). 
\]
This gives (by the standard explicit formulas for Kan extensions)  the values of the Day convolutions
on objects of $\Cc$  as
\be\label{eq:Day-convs-lim}
\begin{gathered}
(F \lday G)(c) = \hocolim_{(c_1 \otimes c_2 \to c)\in \, \otimes/c} F(c_1) \otimes G(c_2), 
\\
(F \rday G)(c) = \holim_{(c\to c_1 \otimes c_2)\in \,  c/\otimes} 
F(c_1) \otimes G(c_2). 
\end{gathered} 
\ee
Here $\otimes/c$ is the overcategory of $\otimes: \Cc^2\to\Cc$ over $c$.
Its objects  are  triples 
  formed by objects $c_1, c_2\in\Cc$ and a morphism
$\alpha: c_1\otimes c_2\to c$.  A morphism from such a triple to another
triple $(c'_1, c'_2, \alpha': c'_1\otimes c'_2\to c)$ is a pair of morphisms
 $\gamma_1: c_1\to c'_1$, $\gamma_2: c_2\to c'_2$ such that
$\alpha' \circ (\gamma_1\otimes\gamma_2) =\alpha$.  

Dually,  $c/\otimes$ is the undercategory of $\otimes$ under $c$. 
Its objects   are triples $(c_1, c_2, \beta: c\to c_1\otimes c_2)$
with morphisms defined similarly to the above. 

\vskip .2cm

Existence of the full-fledged symmetric monoidal structures
$\lday$, $\rday$  requires therefore existence and good properties
of limits such as in \eqref{eq:Day-convs-lim} (and similar limits
for multiple tensor products). However, as follows from
\cite[\S 2.2.6]{lurie-ha},  $\lday$ always exists at the
level of $\oo$-operads, i.e., we have a non-unital $\oo$-operad
$\Fun(\Cc, \Dc)^{\lday}$.  Similarly,  $\rday$ always exists
at the level of $\oo$-cooperads, i.e., we have a non-unital $\oo$-cooperad
$\Fun(\Cc, \Dc)_{\rday}$. 
This corresponds to the fact that the functors (co)represented by the desired
(co)limits in  \eqref{eq:Day-convs-lim}  always exist without these (co)limits
necessarily existing in $\Dc$.  The main feature of these $\oo$-(co)operads
is as follows. More precisely, this is a particular case of  Th. 2.2.6.2 and Construction 2.2.6.7 of
\cite{lurie-ha}, see  \cite[Prop. 2.12]{glasman} and  \cite[Ex. 2.2.6.9]{lurie-ha}
for discussion of the unital case. 

\begin{thm}~
\begin{statements}
\item A commutative (i.e., $E_\oo$- )algebra in $\Fun(\Cc, \Dc)^{\lday}$
 is tantamount to a lax monoidal functor $\Cc\to\Dc$.
 
\item Dually, a cocommutative (i.e., $E_\oo$-) coalgebra in
 $\Fun(\Cc, \Dc)_{\rday}$
 is tantamount to a colax monoidal functor  $\Cc\to\Dc$.
 \qed
\end{statements}
\end{thm}

 Let us now discuss sufficient conditions for existence of the Day convolutions
 as symmetric monoidal structures, i.e., for the property that
 the $\oo$-operad $\Fun(\Cc, \Dc)^{\lday}$ or the $\oo$-cooperad
 $\Fun(\Cc, \Dc)_{\rday}$ are symmetric monoidal $\oo$-categories. 
 
 Recall that we assume that
 $\Cc$ is a (non-unital) symmetric monoidal ordinary category. 
 For an object $c\in \Cc$ and $n\geq 2$ we denote  by $\otimes^n/c$ and
 $c/\otimes^n$ the over- and undercategory of the functor
 $\otimes^n: \Cc^n\to\Cc$ at $c$. 

 \begin{Defi}~
\begin{statements}
\item A symmetric monoidal category $\Cc$ is called {\em of  finite codecompositions}
 if for each object $c\in\Cc$ and each $n\geq 2$, the category $c/\otimes^n$ 
 contains an $\oo$-coinitial full subcategory equivalent to a finite
 discrete set $\Codec_n(c)$.

 \item We say that $\Cc$ is {\em  of finite decompositions}, if, for each $c\in \Cc$ and
 $n\geq 2$ as above, 
  $\otimes^n/c$ contains an $\oo$-cofinal full subcategory equivalent to a finite
 discrete set $\Dec_n(c)$.
 \end{statements}
\end{Defi}

Note that such a set  $\Codec_n(c)$ or $\Dec_n(c)$ is defined uniquely up to unique isomorphism: it is identified with $\pi_0$ (the set of connected components
of the nerve) of $c/\otimes^n $ or $\otimes^n/c$.
Note also that  $\Cc$ being of finite decomposition is equivalent to  $\Cc^\op$ being of finite codecomposition: the category $c/ \otimes^n$ for $\Cc^\op$
  is opposite to the category $\otimes^n/c$ for $\Cc$. 

  \pagebreak[3]
\begin{exas}~
\begin{statements}[ref=Example \theexas(\alph*)]  
\item \label{ex:finsurj-codecomposition} The category $\finsurj$ of nonempty finite sets and surjections,
with $\amalg$ (disjoint union) as monoidal structure, is of finite codecompositions. 
Indeed, let $I\in \finsurj$ be an object. An object of $I/\amalg^n$
is a surjection $\alpha: I\to J_1\amalg \cdots \amalg J_n$. Let $\Codec_n(I)$ consists
of partitions $I=I_1\amalg\cdots \amalg I_n$ into the union of $n$ numbered disjoint
nonempty subsets.  
Then $\Codec_n(I)\subset I/\amalg^n$
is a discrete subcategory. Further, any $\alpha$ as above
  gives a partition
$I=I_1\amalg\cdots\amalg  I_n$ where $I_i=\alpha^{-1}(J_i)$ and a canonical morphism
in  $I/\amalg^n$
\[
(I_1, \dots, I_n,  \Id: I \to  I_1\amalg \cdots \amalg I_n) \lra  
(J_1, \dots, J_n, \alpha)
\]
given
by the restrictions of $\alpha$ to the $I_i$. This shows that $\Codec_n(c)$
is $\oo$-coinitial.

\item \label{ex:finsurj-decomposition} Further, $\finsurj$ is also of
finite decompositions. Indeed, 
an object of $\amalg^n/I$ is a surjection
 $\beta: J_1\amalg \cdots \amalg J_n\to I$.
Call such a surjection {\em essential}, if $\beta|_{J_i}: J_i\to I$ is
injective for $i=1,\dots, n$.  The full  category in $\amalg^n/I$ spanned by  essential surjections is discrete, equivalent to the set $\Dec_n(I)$
 of $n$-tuples  $(I_1, \dots, I_n)$ of nonempty subsets
of $c$ such that $c= \bigcup_i I_i$ (not necessarily disjoint union). 
Further, any surjection $\beta$ factors uniquely through an 
essential surjection
with $I_i=\beta(J_i)$, so the subcategory of essential surjections
is $\oo$-cofinal.
\end{statements}
\end{exas}

  For $\Cc$ of finite (co)decompositions the (co)limits 
  \eqref{eq:Day-convs-lim} reduce to finite  (co)products.
  This makes natural the following claim which will be sufficient
  for our purposes.

\begin{prop}\label{prop:finiteconvolution}
       
       Suppose  that   $\Dc$ is a stable  $\oo$-category equipped with a symmetric
       monoidal structure $\otimes$ which is exact in each argument.

\begin{assertions}
 \item Let $\Cc$ be a symmetric monoidal category  of finite decompositions.
       Then the left   Day convolution on the $\infty$-category $\Fun(\Cc, \Dc)$
     exists, i.e., the $\oo$-operad $\Fun(\Cc, \Dc)^{\lday}$
     comes from a symmetric monoidal $\oo$-category. 
     
 \item Dually, let  $\Cc$ is of finite codecomposition.    Then the right Day convolution on 
  $\Fun(\Cc, \Dc)$
     exists, i.e., the $\oo$-cooperad $\Fun(\Cc, \Dc)_{\rday}$
     comes from a symmetric monoidal $\oo$-category. 
\end{assertions}
\end{prop}

\begin{proof}
  Let us consider (a), since (b) is dual. 
     The property that  $\Fun(\Cc, \Dc)^{\lday}$
     comes from a symmetric monoidal $\oo$-category means that 
     $\Fun(\Cc, \Dc)^{\lday}
     \to\finsurj^*$ is a coCartesian fibration. A sufficient condition for this is given in 
  \cite[Prop. 2.2.6.16]{lurie-ha} in terms of existence of arbitrary
  ($\kappa$-small) colimits in $\Dc$ and  commutativity of $\otimes^n: \Cc^n\to\Cc$
  with such colimits in each variable. An examination of the argument shows
  that we need only the colimits appearing as the $n$-variable analogs of the 
  colimits in the first line of  \eqref{eq:Day-convs-lim}. These colimits
  are reduced to finite coproducts by our assumption on $\Cc$. Further, the restriction
  of $\otimes^n: \Dc^n\to\Dc$ to each variable (keeping all the other variables fixed)
  gives an $\oo$-functor $\Dc\to\Dc$ which preserves these coproducts (i.e., sums)
  by our assumption on $\Dc$ and $\otimes$. 
\end{proof}         
     

\subsection{Day convolution on (op-)lax limits.}

\paragraph{From functors to sections of the Grothendieck construction.} 
We will need a generalization of the   framework  of \S \ref{subsec:sym-mon-inf} obtained via
 the formalism
of \S \ref{sec:laxmodules}\ref{par:laxlim}
 Let $\Cc$ be a symmetric monoidal category as above and 
 $\Fc: \Cc \to (\Cat_\oo, \times)$ be a lax  monoidal $\oo$-functor
 (here $\Cat_\oo$ is considered with the Cartesian product 
 monoidal structure). We are interested in monoidal structures
 on the $\oo$-categories $\laxlim (\Fc)$, $\oplaxlim(\Fc)$. 
 
 \begin{ex}
 Let $\Dc$ be a symmetric monoidal $\oo$-category.  Then we have
 the {\em constant lax   monoidal $\oo$-functor} $\Fc_\Dc: \Cc\to\Cat_\oo$.
 On objects, $\Fc_\Dc(c)=\Dc $ independently on $c\in\Cc$, and each morphism in $\Cc$
 is sent by $\Fc_\Dc$ to the identity functor of $\Dc$. 
 The lax compatibility morphisms  \eqref{eq:lax-mon-fun}
 \[
 \Fc_\Dc(c_1)\times \Fc_\Dc(c_2) = \Dc\times\Dc\lra 
 \Dc=\Fc_\Dc(c_1\otimes c_2)
 \]
  are given by the monoidal structure $\otimes$ on $\Dc$, and similarly
  for the higher data comprising the concept of a lax monoidal $\oo$-functor. 
 In this case 
 \[
\laxlim(\Fc_\Dc) \,=\, \Fun(\Cc, \Dc), \quad  \oplaxlim(\Fc_\Dc) \,=\,  \Fun(\Cc^\op, \Dc). 
 \]
 \end{ex}
 We now explain how, under certain conditions,  one can generalize the Day convolution 
 $\lday$ from the case of $\Fun(\Cc, \Dc)$
 to the case of more general $\laxlim(\Fc)$. Similarly, one can generalize
 the convolution $\rday$ from the case of $\Fun(\Cc, \Dc)$ to the case of
 more general $\oplaxlim(\Fc)$, where $\Fc$ is a {\em contravariant} functor
 $\Cc \to \Cat_\oo$. 
 
 By Definition \ref{def:lax-oplax-lim},  $\laxlim(\Fc)$, resp. 
$\oplaxlim(\Fc)$ is the category of sections of the coCartesian  Grothendieck
construction $\Grothco (\Fc) \to \Cc$, resp. of the Cartesian Grothendieck
construction  $\Groth (\Fc) \to \Cc^\op$.
We further recall the identifications, see again Definition  \ref{def:lax-oplax-lim}:
  \be\label{eq:groth-cogroth-op}
   \Groth(\Fc)\,  \= \,  \left( \Grothco (\Fc^\circ)  \right) ^\op, \quad  \oplaxlim(\Fc) \, \=  \, \laxlim(\Fc^\circ)^\op. 
 \ee

 \paragraph{Grothendieck construction as a symmetric monoidal $\oo$-category.} 
 Before dealing with sections, we study the Grothendieck constructions themselves.

\begin{prop}\label{prop:grothmonod}
    Let $\Fc \colon \Cc \to \Cat_\infty$ be a non-unital lax monoidal $\infty$-functor. 
    
\begin{assertions}
 \item The coCartesian Grothendieck construction $\Grothco(\Fc)$ is endowed with a natural 
    non-unital symmetric monoidal structure $\otimes$, such that the projection
     $\Grothco (\Fc) 
    \to \Cc$ is a symmetric monoidal $\infty$-functor.

 \item Dually, 
    the Cartesian Grothendieck construction $\Groth (\Fc)$ is also endowed with a 
    natural non-unital symmetric monoidal structure $\otimes$  such that the projection $\Groth(\Fc) 
    \to \Cc^\op$ is a symmetric monoidal $\infty$-functor.
\end{assertions}
    \end{prop}
    
    At the lowest level, that  of objects, the symmetric monoidal structure in, say, (a) is given
    as follows. An object of $\Grothco(\Fc)$ is a pair $(c,d)$ where $c\in\Cc$ and
    $d\in\Fc(c)$ are objects. Then, in the notation of \eqref{eq:lax-mon-fun}, 
    \[
    (c_1,d_1) \otimes (c_2, d_2) = 
    (c_1\otimes c_2,  \Fc_{c_1, c_2}(d_1, d_2))
    \] 

 \begin{proof} 
   (a) 
    We consider the Cartesian symmetric monoidal structure on $\Cat_\infty$.
    In particular, the $\infty$-category of lax 
monoidal functors $\Cc \to \Cat_\infty$ is equivalent to a full sub-$\infty$-category of
 $\Fun(\Cc^\otimes, \Cat_\infty)$ 
 (the full subcategory of lax Cartesian structures, see
  \cite[Prop. 2.4.1.7]{lurie-ha}).
        
     Let $\wt \Fc$ denote the image of $\Fc$ under this inclusion.
     We can now apply the coCartesian Grothendieck construction to $\wt\Fc$ and obtain a 
    coCartesian fibration $\pi \colon \Grothco (\wt \Fc) \to \Cc^\otimes$. Composing it  with
    the coCartesian fibration $\Cc^\otimes \to \finsurj^{*}$, we obtain the coCartesian fibration 
   $\Grothco (\wt \Fc) \to \finsurj^{*}$. The fact that $\wt \Fc$ is a lax Cartesian structure implies that this coCartesian fibration defines a non-unital symmetric monoidal structure $\otimes$ on $\Grothco(\Fc)$  such that $\Grothco (\Fc)^{\otimes } = \Grothco (\wt \Fc)$.

   It remains to see that the projection $\Grothco (\Fc) \to \Cc$ is monoidal -- i.e. that the projection $\pi$ preserves morphisms that are coCartesian over $\finsurj^*$. This follows from the following
   \begin{lem}
    Let $p \colon \Cc_1 \to \Cc_2$ and $q \colon \Cc_2 \to \Cc_3$ be two coCartesian fibrations of $\infty$-categories. Let $f$ be a morphism of $\Cc_1$ that is coCartesian over $\Cc_3$. Then $q(f)$ is coCartesian over $\Cc_3$.
   \end{lem}
   \begin{proof}
Let $f \colon x \to y \in \Cc_1$ be coCartesian over $\Cc_3$. Since $q$ is a coCartesian fibration, $q \circ p(f)$ admits a coCartesian lift $f' \colon p(x) \to y' \in \Cc_2$. With $p$ also being a coCartesian fibration, there exists $f'' \colon x \to y'' \in \Cc_1$ lifting $f'$ and coCartesian over $\Cc_2$.
It follows from \cite[Prop. 2.4.1.3 (3)]{lurie-htt} that $f''$ is also coCartesian over $\Cc_3$. By uniqueness of coCartesian lifts, $f$ and $f''$ are canonically homotopic. In particular $p(f) = f'$ is coCartesian over $\Cc_3$. 
   \end{proof}
   This concludes the proof of \autoref{prop:grothmonod} (a). Part (b) follows via
   \eqref{eq:groth-cogroth-op}.  
    \end{proof}
    
     Proposition \ref{prop:grothmonod} leads to the following concepts
 which, for  constant $\Fc$, reduce to the concept of a  lax and colax monoidal functor
 $\Cc\to\Dc$.

 \begin{Defi}~
 \begin{statements}
   \item A  {\em lax monoidal object} of $\laxlim(\Fc)$ or a
      {\em lax monoidal section}
     of $\Grothco(\Fc)$, is  lax monoidal functor $\Cc\to \Grothco(\Fc)$
     which is a section of 
      $\Grothco(\Fc)\to\Cc$.  
  
  \item A {\em colax monoidal object} of $\oplaxlim(\Fc)$, or a {\em colax monoidal
      section} of $\Groth(\Fc)$, is a colax monoidal functor 
      $\Cc^\op\to \Groth(\Fc)$
     which is a section of 
      $\Groth(\Fc)\to\Cc^\op$.
 \end{statements}
 \end{Defi}
 
 Our  next goal is to interpret (co)lax monoidal objects in terms of   symmetric monoidal structures
 given by Day convolutions of sections.

    \paragraph{Day convolution of sections as relative Kan extension.} 
    We recall the notion of {\em relative Kan extensions} which generalize the usual left and right
    Kan extensions to the setting of relative (or parametrized) $\oo$-category theory.
    In that setting, the  main object is, instead of  a single $\oo$-category,
    a (co)Cartesian fibration over a base 
    $\oo$-category.

    More precisely, for 
     a coCartesian fibration $p: \Dc\to\Dc'$ and a commutative square of $\oo$-categories
    on the left of \eqref{eq:rel-kan-ext}, the relative  left Kan extension 
    is an  $\oo$-functor $\lambda = u_!^{\Dc/\Dc'} f$ from $\Ec'$ to $\Dc$, 
    denoted by the dashed arrow, which (when it exists)
    satisfies the relative analog of the standard universal properties and splits the
    square into two triangles: the lower one homotopy commutative and the upper one
    equipped with a natural transformation $f\Rightarrow \lambda u$.
    See \cite[Def. 4.3.2.2]{lurie-htt}
    for the case when $u$ is a full embedding and \cite[\S 10]{shah} for the general case. 
     \be\label{eq:rel-kan-ext}
     \begin{tikzcd}
      \Ec \ar[r, ""{name=U, below, near start}, "f"]  \ar[Rightarrow, from=U, to=U, end anchor={[yshift=-1em]south}] \ar{d}[swap]{u} & \Dc \ar{d}{p} \\ \Ec' \ar{r} \ar[dashed]{ur}[swap]{\lambda} & \Dc'
     \end{tikzcd}
     \hspace{2cm}
     \begin{tikzcd}
      \Cc\times\Cc \ar[r, ""{name=U, below, near start}, "s\boxtimes t"]  \ar[Rightarrow, from=U, to=U, end anchor={[yshift=-1.2em]south}] \ar{d}[swap]{\otimes} & \Grothco(\Fc) \ar{d}{\pi} \\ \Cc \ar{r}[swap]{\Id} \ar[dashed]{ur}[swap]{s\lday t} & \Cc
     \end{tikzcd}
    \ee
    We apply this to the situation on the right of \eqref{eq:rel-kan-ext}. Given two objects
    of $\laxlim(\Fc)$, i.e., two sections $s,t$ of $\Grothco(\Fc)$, we have their
    external tensor product $s\boxtimes t$ as in   \S \ref{subsec:sym-mon-inf}\ref{par:day-convo}
    The left Day convolution of $s$ and $t$ is then defined as  the relative Kan extensions
    (provided the extension exists):
           \[
    s\lday t \,=\, \otimes_!^{\Grothco(\Fc)/\Cc}(s\boxtimes t).
    \]
    As for usual Kan extensions, we have the standard formulas for the values of relative Kan extensions on objects as certain relative
    (co)limits, see \cite[Th.10.3]{shah} . In our case they give
    \begin{equation}\label{eq:left-day-sections}
  (s\lday t)(c) \,=\,\hocolim_{(a: c_1\otimes c_2 \to c)\in \otimes/c}{}^{\hskip -0.7 cm \Fc(c)}\,\,\, 
   \Fc(a)\bigl( \Fc_{c_1, c_2} (s(c_1), t(c_2))\bigr). 
    \end{equation}
    Here $\Fc_{c_1, c_2}$ is the lax monoidal structure data for $\Fc$, while $\Fc(a)$ is the value
    of $\Fc$ on the morphism $a$, and the colimit is taken in the category $\Fc(c)$. 
    
    \vskip .2cm
    
    Dually, let $\Fc: \Cc^\op \to\Cat_\oo$ be an $\oo$-functor. Then
    $\Groth(\Fc)\to \Cc$,  the
    contravariant Grothendieck construction of $\Fc$, is a Cartesian fibration over $\Cc$
    (and not $\Cc^\op$). In this case,
    for two sections $s,t$ of $\Groth(\Fc)$ their right Day convolution and  its
    values on object  are defined as the right relative Kan extension
    \begin{equation}\label{eq:right-day-sections}
    \begin{gathered}
     (s\rday t) \,=\, \otimes_*^{\Groth(\Fc)/\Cc} (s\boxtimes t), 
   \\
    (s\rday t) (c) \,=\, \holim_{(b: c\to c_1\otimes c_2)\in c/\otimes }{}^{\hskip -0.7cm \Fc(c)}
    \,\,\, \Fc(b)\bigl( \Fc_{c_1, c_2} (s(c_1), t(c_2))\bigr).
    \end{gathered}
    \end{equation}
    Here $\Fc(b): \Fc(c_1\otimes c_2)\to\Fc(c)$ is the value of $\Fc$ on the morphism
    $b$. 
    
  \begin{rem}  Note that the usual Kan extension $\otimes_!(s\boxtimes t)$, 
    (i.e., the Day convolution of $s$ and $t$ as functors $\Cc \to\Grothco(\Fc)$)
    need not be a section, and similarly for the usual
   $ \otimes_*(s\boxtimes t)$. 
   Indeed,  the value of $\otimes_!(s\boxtimes t)$ at $c$ would be   the colimit      of the diagram of the $\Fc_{c_1, c_2} (s(c_1), t(c_2))$ which does not lie in a single fiber of $\pi$. 
    \end{rem}
    
    \paragraph{Left Day convolution of sections at the level of $\oo$-operads.}
     Next, we  define  the left Day convolution $\lday$ on $\laxlim(\Fc)$ at the level of $\oo$-operads.      
 That is, we define the $\oo$-operad $\laxlim(\Fc)^{\lday}$
    by applying the {\em norm construction} of \cite[Th. 2.2.6.2]{lurie-ha} to the composable pair
    of coCartesian fibrations of $\oo$-operads
    \[
     \Grothco ( \Fc)^{\otimes} = 
  \Grothco (\wt \Fc) \lra \Cc^\otimes \lra \finsurj^*. 
         \]
    It produces a fibration of $\oo$-operads (i.e., in our case,  simply a non-unital $\oo$-operad)
 
    \be\label{eq:laxlim=norm}
     \laxlim(\Fc)^{\lday} \,\buildrel{\on{def}.}\over =  \, \on{Nm}_{\Cc/\finsurj^*} (  \Grothco  (\Fc)^{\otimes} )
  \buildrel \pi\over  \lra \finsurj^*. 
  \ee
     
    \begin{prop}
    The $\oo$-category of non-unital commutative algebras in $\laxlim(\Fc)^{\lday}$
    (i.e., of operadic sections of $\pi$) is equivalent to that of lax monoidal objects in $\laxlim(\Fc)$.
    \end{prop}
    
\begin{proof}
 By the universal property of the norm  \cite[Def. 2.2.6.1]{lurie-ha}, the category
 of operadic sections of $\pi$ is equivalent to that of operadic sections of   $p: \Grothco(\Fc)\to \Cc$
 which are precisely lax monoidal sections, i.e., lax monoidal objects of $\laxlim(\Fc)$.
\end{proof}

 \paragraph{Day convolutions as symmetric monoidal structures.}
 We now give a criterion, generalizing Proposition \ref{prop:finiteconvolution}, for the $\oo$-operad
  $\laxlim(\Fc)^{\lday}$ to be a symmetric monoidal $\oo$-category.

\begin{prop}\label{prop:convonlaxlim}
    Let $\Fc \colon \Cc \to \Cat_\infty$ be a non-unital lax monoidal $\infty$-functor.
    Assume that:
    \begin{itemize}
    \item[(1)]  $\Cc$ is of finite decompositions.
    
    \item[(2)]  For any $c \in \Cc$, the $\infty$-category $\Fc(c)$ admits finite 
    coproducts.
    
    \item[(3)] For any morphism $a: c_1\otimes\cdots\otimes c_n\to c$ in $\Cc$, the 
    composite functor
    \[
    \Fc(c_1)\times\cdots\times \Fc(c_n)\buildrel \Fc_{c_1,\dots, c_n}\over\lra
      \Fc(c_1\otimes\cdots\otimes c_n) \buildrel  \Fc(a)\over\lra \Fc(c)
    \]
    preserves coproducts in each variable. 
     \end{itemize}
    
  \noindent Then the  $\oo$-operad $\laxlim(\Fc)^{\lday}$ is a non-unital symmetric monoidal $\oo$-category. In other words, the 
   left Day convolution defines a non-unital symmetric monoidal structure $\lday$ on $\laxlim \Fc$. 
 
\end{prop}

\begin{proof}
The argument is analogous to those of \cite[Prop.2.12]{glasman} 
and \cite[Prop.2.2.6.16]{lurie-ha}
so we indicate only the main steps. We need to prove that $\pi: \laxlim(\Fc)^\lday\to\finsurj^*$
is a coCartesian fibration, i.e., that for each morphism $\alpha: I\to J$ of $\finsurj^*$ and
any object $x$ lying over $\alpha$ there is a coCartesian lift. We already know this
when $\alpha$ is inert (as $\laxlim(\Fc)^\lday$ is an $\oo$-operad). To extend
this to all $\alpha$, recall the concept of a {\em locally coCartesian edge} in $\laxlim(\Fc)^\lday$.
This is an edge $f$ projecting to some $\alpha$ which becomes coCartesian
in the pullback fibration over $\Delta^1$ induced via the functor $\Delta^1\to\finsurj^*$
given by $\alpha$, 
see \cite[Def. 2.4.2.6]{lurie-htt}.  
So it suffices
to establish the following two claims:

\vskip .2cm

\noindent {\bf (A)}  
  For any morphism $\alpha: I\to J$ in $\finsurj^*$ and any $x\in \pi^{-1}(I)$ there is a locally
  coCartesian edge lifting $\alpha$.
  
  \vskip .2cm
  
  \noindent {\bf (B)} The composition of two locally coCartesian edges from (A)
 is locally coCartesian. (This will imply our statement by \cite[Prop.2.4.2.8]{lurie-htt}). 
  
  \vskip .2cm
  
  \noindent{\sl Proof of (A):}  we already know the statement for inert morphisms $\alpha$.
  Let us consider only the following family of active morphisms
   $\alpha_n: I:= \{*, 1,\dots, n\}\to J:= \{*,1\}$. By definition, $\alpha_n$
 projecting each $1\leq i\leq n$ to $1$.
  Consideration of more general morphisms is similar: for instance, active
  morphisms amount to external
  coproducts (in the pointed sense)
  of  several of the $\alpha_n$. So let us  show the existence of locally
  coCartesian lifts of $\alpha_n$. 
  
  The pullback of  $\pi$ via $\alpha_n$ (a fibration over $\Delta^1$) is a
   correspondence  between
  $\pi^{-1}(I)$ and $\pi^{-1}(J)$.  For each object of $\pi^{-1}(I)$, i.e., 
  is each $n$-tuple of sections $\ul s=(s_1,\dots, s_n)$ this correspondence gives
  a functor  from $\laxlim(\Fc) = \pi^{-1}(J)$ to spaces
which we  denote  $\Mult(s_1,\dots, s_n; -)$ and call the  {\em functor of multi-operations}. To construct a locally coCartesian lift of $\alpha_n$
  starting from $\ul s$ is the same as to exhibit a representing object for 
   $\Mult(s_1,\dots, s_n; -)$.  We claim the  following natural identifications:
  \be\label{eq:mult}
\begin{gathered}
 \Mult(s_1, \dots, s_n ; t) \,\simeq \,  \Map_{\Fun_{\Cc}(\Cc^n, \Grothco(\Fc))}(\otimes^n \circ \underline s, t \circ \otimes^n ).
 \end{gathered}
\ee
 Given \eqref{eq:mult},  the functor $\Mult(s_1, \dots, s_n ; -)$ is representable as soon as 
$\otimes^n \circ \underline s \colon \Cc^n \to \Grothco(\Fc)^n \to \Grothco(\Fc)$ admits a relative left Kan extension over $\Cc$ along the tensor product functor $\otimes^n: \Cc^n \to \Cc$.
The existence of such relative extension reduces, see 
  \cite[Cor. 4.3.1.11]{lurie-htt} and \cite[Thm. 10.3]{shah},   to the existence of the \emph{relative}
  colimits
  \[
  \hocolim_{(a:c_1\otimes\cdots\otimes c_n\to c)\in \otimes^n/c}\nolimits^{\on{Rel}.} 
  \Fc(a) (\Fc_{c_1,\dots, c_n}(s_1(c_1), \dots, s_n(c_n)))
  \]
  for all objects $c\in \Cc$. By \cite[Prop. 4.3.1.10]{lurie-htt}, such a relative colimit exists as soon as the usual (absolute) colimit exists in $\Fc(c)$ and is preserved by every functor of the form $\Fc(c \to d)$. Since by our assumptions,
  such  absolute colimits reduce to finite coproducts that exist in the $\Fc(c)$
  and the $\Fc(c \to d)$ preserves them,  the representability follows.
  
  \vskip .2cm
  
  The identification  \eqref{eq:mult} is obtained via
  \cite[Prop.2.2.6.6]{lurie-ha} which gives a pullback diagram
  \[
  \begin{tikzcd}
   \Fun_{\mathscr S^*}(\Delta^1, (\laxlim \Fc)^{\otimes}) \ar{r}{\ev_0} \ar{d} \ar[start anchor={[xshift=.5ex,yshift=1ex]south west}, bend right, end anchor={[xshift=3.5ex]north west}]{dd}[swap]{\ev_1} & \overset{\overset{\makebox[0pt][c]{$\Fun_\Cc(\Cc, \Grothco(\Fc))^n$}}{\rotatebox{90}{$=$}}}{(\laxlim \Fc)^n} \ar{d} \\
   \Fun_{\Cc^{\otimes}}\left(\Delta^1 \times_{\mathscr S^*} \Cc^{\otimes}, \Grothco(\Fc)^{\otimes}\right) \ar{r}{\ev'_0} \ar{d}{\ev'_1} & \Fun_{\Cc^n}(\Cc^n, \Grothco(\Fc)^n) \\
   \laxlim \Fc \simeq \Fun_{\Cc^{\otimes}}\left(\Cc, \Grothco(\Fc)^{\otimes}\right) \phantom{\laxlim \Fc \simeq\,}
  \end{tikzcd}
 \]
 
Here the map $\Delta^1\to \finsurj^*$ is given by $\alpha_n$. The arrows
 $\ev_0$ and $\ev_1$ are the evaluations at the two ends of $\Delta^1$,
 so $\ev_0$ takes values in $\pi^{-1}(\{*,1, \dots, n\})= (\laxlim\Fc)^n$ and $\ev_1$
 in $\pi^{-1}(\{*,1\})=\laxlim\Fc$, and similarly for $\ev'_0, \ev'_1$. 
  Using now the equivalence at the very bottom of page 229 in \cite{lurie-ha}, we find another pullback diagram:
\[
 \begin{tikzcd}
  \Fun_{\Cc^{\otimes}}\left(\Delta^1 \times_{\mathscr S^*} \Cc^{\otimes}, \Grothco(\Fc)^{\otimes}\right) \ar{r} \ar{d}{\ev'_1} & \Fun_{\Cc^{\otimes}}\left(\Delta^1 \times \Cc^{n}, \Grothco(\Fc)^{\otimes}\right) \ar{d}{\ev''_1} \\
  \makebox[0pt][r]{$\laxlim \Fc \simeq\,$} \Fun_{\Cc^{\otimes}}\left(\Cc, \Grothco(\Fc)^{\otimes}\right) \ar{r}[swap]{(\otimes^n \colon \Cc^n \to \Cc)^*} & \underset{\overset{\rotatebox{90}{$=$}}{\makebox[0pt][c]{$\Fun_{\Cc}\left(\Cc^{n}, \Grothco(\Fc)\right).$}}}{\Fun_{\Cc^{\otimes}}(\Cc^n, \Grothco(\Fc)^{\otimes})}
 \end{tikzcd}
\]
Here the functor $\Delta^1 \times \Cc^n \to \Cc^{\otimes}$ used in the upper right corner maps $\Delta^1 \times \{(c_1, \dots, c_n)\}$ to the coCartesian morphism of $\Cc^{\otimes}$:
\[
 \Cc^{\otimes} \supset \Cc^n \ni (c_1, \dots, c_n) \to c_1 \otimes \cdots \otimes c_n \in \Cc \subset \Cc^{\otimes}.
\]
In particular, the functor $\Cc^n \to \Cc^{\otimes}$ in the bottom right corner is the composite $ \Cc^n\buildrel \otimes^n\over \to \Cc \subset \Cc^{\otimes}$.

 By definition, the functor $\Mult(s_1, \dots, s_n ; - )$ is classified by the fibration 
 \[
\ev_0^{-1}(s_1, \dots, s_n) \hookrightarrow \Fun_{\mathscr S^*}(\Delta^1, (\laxlim F)^{\otimes})\buildrel \ev_1\over \lra \laxlim F.
 \]
  From the two pullback squares above, to understand $\Mult(s_1, \dots, s_n ; -)$, 
  it suffices to study the fibration
 \[
  {\ev''_0}^{-1}(s_1, \dots, s_n) \hookrightarrow \Fun_{\Cc^{\otimes}}\left(\Delta^1 \times \Cc^{n}, \Grothco(F)^{\otimes}\right) \buildrel {\ev''_1} \over \lra \Fun_{\Cc}\left(\Cc^{n}, \Grothco(F)\right).
 \]
  After unravelling the definitions, we find \eqref{eq:mult}.

  \vskip .2cm

  \noindent{\sl Proof of (B):}   the statement (A) established the existence of iterated 
  tensor products. Given that,
  the statement (B) means that such products are associative. 
 This is a straightforward
  manipulation with the colimits as above, using the fact that  
  each $\Fc_{c_1,\dots, c_n}$ as well as each $\Fc(a)$ preserves 
  finite coproducts in each 
  of their arguments. Further details are similar to those of   \cite[Prop.2.12]{glasman}.
\end{proof}

\begin{cor}\label{cor:day-oplaxlim}
Let $\Fc: \Cc \to \Cat_\oo$ be a {\em contravariant} non-unital lax monoidal $\oo$-functor. Assume that:
 
    \begin{itemize}
    \item[(1)]  $\Cc$ is of finite codecompositions.
    
    \item[(2)]  For any $c \in \Cc$, the $\infty$-category $\Fc(c)$ admits finite direct products.
    
    \item[(3)] For any morphism $b: c\to c_1\otimes\cdots\otimes c_n$ in $\Cc$, the composite functor
    \[
    \begin{tikzcd}
    \Fc(c_1)\times\cdots\times \Fc(c_n) \ar{rr}{\Fc_{c_1,\dots, c_n}} && \Fc(c_1\otimes\cdots\otimes c_n) \ar{r}{\Fc(b)} & \Fc(c)
    \end{tikzcd}
    \] preserves finite direct products in each variable. 
     \end{itemize}
    
   Then the  
   right Day convolution defines a non-unital symmetric monoidal structure $\rday$ on $\oplaxlim \Fc$. 
\end{cor}

\begin{proof} Let $\Fc^\circ: \Cc \to\Cat_\oo$ be the 
(still contravariant)  functor mapping $c\in \Cc$ to
the category $\Fc(c)^\op$, see  \autoref{def:lax-oplax-lim}. The functor
 $(-)^\op: \Cat_\oo \to\Cat_\oo$, being an equivalence,
  is therefore monoidal for the Cartesian monoidal structure. Therefore $\Fc^\circ$ inherits the lax monoidal
  structure from that of $\Fc$. Now, because of \eqref{eq:groth-cogroth-op}, the left Day convolution on
  $\laxlim(\Fc^\circ)$ induces the desired right Day convolution on $\oplaxlim(\Fc)$.
\end{proof}


 \paragraph{Effect of lax monoidal transformations.} 

\begin{prop}\label{prop:laxtransf}
    Given two lax monoidal $\infty$-functors $\Fc, \Gc \colon \Cc \to \Cat_\infty$ and 
    a lax monoidal transformation $\Fc \Rightarrow \Gc$, the induced $\infty$-functors
    \[
    \Grothco \Fc \to \Grothco \Gc \hspace{8mm} \text{and} \hspace{8mm} \Groth \Fc \to \Groth \Gc
    \]
    are lax monoidal. If the natural transformation is (strictly) monoidal, then 
    the induced $\infty$-functors are too.
\end{prop}
\begin{proof}
    The lax monoidal $\infty$-functors $\Fc$ and $\Gc$ correspond to $\infty$-functors 
    \[
    \wt \Fc, \wt \Gc \colon \Cc^\otimes \to \Cat_\infty.
    \]
    The lax monoidal 
    transformation $\Fc \Rightarrow \Gc$ then correspond to a lax natural 
    transformation $\wt \Fc \Rightarrow \wt \Gc$ and therefore to a 
    commutative diagram
    \[
    \xymatrix@R=1em{
        \displaystyle \Grothco \wt \Fc \ar[rr] \ar[dr] && \displaystyle \Grothco \wt 
        \Gc \ar[dl] \\ & \Cc^\otimes. &
    }
    \]
    The result follows.
    
    \vskip .2cm The case of the Cartesian Grothendieck construction is done using the duality between the Cartesian and the coCartesian constructions described in \cite{BGS}
   and  which we recall. For any $\oo$-functor $\Fc: \Cc\to\Cat_\oo$ 
    (not necessarily lax monoidal) the main result of \cite{BGS} recovers
    the Cartesian fibration $\ol p: \Groth(\Fc)\to\Cc^\op$ from the coCartesian
    fibration $p: \Grothco(\Fc)\to\Cc$. More precisely, $\Groth(\Fc)$
    is identified with $\Grothco(\Fc)_{\on{corr}}$,  the $\oo$-category of correspondences
    in $\Grothco(\Fc)$
    \be\label{eq:corresp}
    x\buildrel f\over\lla z \buildrel y\over\lra y
    \ee
    where $f$ is such that $p(f)$ is an equivalence in $\Cc$ and $g$ is $p$-coCartesian.
    This means that:
    \begin{itemize}
    \item $\Ob(\Grothco(\Fc)_{\on{corr}}) = \Ob(\Groth(\Fc))$.
    
    \item $1$-morphisms of $\Grothco(\Fc)_{\on{corr}}$
     are diagrams \eqref{eq:corresp}
    in $\Grothco(\Fc)$ as above. 
    
    \item ``Composition of correspondences is done in the standard way  by forming pullbacks'', i.e., $n$-simplices of  $\Grothco(\Fc)_{\on{corr}}$
    are defined to be pyramid mesh diagrams consisting of Cartesian squares, see
    \cite[\S1.2]{BGS}. 
    \end{itemize}
    
    \noindent Let now $\Fc$ be lax monoidal. In this case the above construction of
    $\Groth$
    from $\Grothco$ upgrades to a construction of the {\em coCartesian} fibration
    $\pi: \Groth(\Fc)^\otimes \to\finsurj^*$  giving the symmetric monoidal structure
    on $\Groth(\Fc)$ out of the coCartesian fibrations
    \[
    \Grothco(\wt\Fc) = \Grothco(\Fc)^\otimes \buildrel q\over\lra \Cc^\otimes
    \buildrel r\over\lra \finsurj^*
    \]
    giving the symmetric monoidal structures on $\Grothco(\Fc)$ and $\Cc$. 
    More precisely, $ \Groth(\Fc)^\otimes $ is identified with the $\oo$-category
    of the correspondence in $\Grothco(\Fc)^\otimes $ of the form  \eqref{eq:corresp}
    in which $f$ is such that $q(f)$ is $r$-coCartesian and $g$ is $q$-coCartesian
    with $(rq)(g)$ an equivalence. 
    
    This construction of $\Groth(\Fc)^\otimes$  out of $\Grothco(\Fc)^\otimes$
    reduces our proposition to the case of $\Grothco(\Fc)\to\Grothco(\Gc)$. 
\end{proof}

\begin{cor}\label{cor:laxmonoidalonlaxlimits}
    Let $\Fc, \Gc \colon \Cc \to \Cat_\infty$ be lax monoidal $\infty$-functors such that for any $c \in \Cc$ the $\infty$-categories $\Fc(c)$ and $\Gc(c)$ admit finite limits. Assume that $\Cc$ is of finite decomposition. For any lax monoidal transformation $\Fc \Rightarrow \Gc$, the induced $\infty$-functors
    \[
    \laxlim \Fc \to \laxlim \Gc \hspace{8mm} \text{and} \hspace{8mm} \oplaxlim \Fc \to \oplaxlim \Gc
    \]
    are lax monoidal.
\end{cor}

\begin{prop}\label{prop:restrictionlaxmonoidal}
    Consider two $\infty$-functors $\Fc \colon \Cc_1 \to \Cc_2$ and  $\Gc \colon \Cc_2 \to \Cat_\infty$. Assume that $\Fc$ is monoidal, that $\Gc$ is lax monoidal and that both $\Cc_1$ and $\Cc_2$ are of finite decomposition. The restriction $\infty$-functor $\laxlim \Gc \lra \laxlim \Gc \circ \Fc$ is lax monoidal.
\end{prop}

\begin{proof}
    Recall the pullback diagram
    \[
    \xymatrix{
        \Grothco(\Gc \circ\Fc) \ar[r]^-{\underline \Fc} \ar[d] & \Grothco(\Gc) \ar[d] \\ \Cc_1 \ar[r]_\Fc & \Cc_2.
    }
    \]
    Since $\Fc$ is monoidal, so is the projection $\underline \Fc$ and the above diagram is a pullback of symmetric monoidal $\infty$-categories. The result follows by  
    applying the left Day convolution and using its universal property. 
    \end{proof}


\subsection{Factorization algebras 
 as modules over \texorpdfstring{$X^\finsurj$}{XS}}\label{subsec:FA-finsurj}
 
Introduced for $X$ a curve by Beilinson and Drinfeld \cite{BD}, factorization algebras are structured $\Dc^!$-modules over $\Ran(X)$. Intuitively, a factorization structure on $E = (E^{(I)})$ is the data of compatible equivalences between $E^{(I)}$ and $\bigboxtimes_{i\in I} E^{(\{i\})}$ once restricted to the complement of the big diagonal in $X^I$.

In \cite{FG}, Francis and Gaitsgory generalized the definition to $X$ of any dimension.
In this subsection, we first recall the definitions for strict $\Dc^!$-modules, and then extend them to lax $\Dc^!$-modules (and lax $[\![\Dc]\!]$-modules) using Day convolutions.

\paragraph{Factorization structures on strict modules.} Recall that  $\DsM(X^\finsurj)$  denotes the category
of strict $\Dc^!$-modules on $X^\finsurj$. 
\begin{Defi}~
\begin{statements}[ref=Definition \theDefi(\alph*)]
 \item Let $\alpha \colon I \dra J$ be a surjection, seen as a partition $I = \coprod_j I_j$ where $I_j = \alpha^{-1}(j)$.
    We denote by $U(\alpha)$ the open subvariety of $X^I$ whose points are families $(x_i)$ such that the sets $\{x_i, i \in I_j\}$ indexed by $j \in J$ are pairwise disjoint.
    Equivalently:
    \[
    U(\alpha) = \{ (x_i)_{i \in I} ~|~ \forall i_1, i_2 \in I, ~ \alpha(i_1) \neq \alpha(i_2) \implies x_{i_1} \neq x_{i_2} \}.
    \]
    We denote by $j_\alpha \colon U(\alpha) \to X^I$ the open immersion.
 \item \label{def:jI1I2} Given two non-empty finite sets $I_1$ and $I_2$, we denote by $\gamma_{I_1,I_2}$ the surjection $I_1 \amalg I_2 \to \{1, 2\}$ mapping points of $I_j$ to $j$, for $j = 1,2$.
We denote by $j_{I_1, I_2} := j_{\gamma_{I_1, I_2}}$  the open immersion
\[
 U(\gamma_{I_1,I_2}) = \bigl\{((x_{i_1})_{i_1 \in I_1}, (x_{i_2})_{i_2 \in I_2}) ~|~
x_{i_1}\neq x_{i_2}, \,\forall i_1\in I_1, i_2\in I_2
 \bigl \} \,\, \hra \,\, X^{I_1} \times X^{I_2}. 
\]
\end{statements}
\end{Defi}

\begin{Defi}(Beilinson-Drinfeld, Francis-Gaitsgory)\label{Defi:factFG}~
\begin{statements}
 \item  We call the {\em chiral tensor product} and denote $\otimes^\mathrm{ch}$
 the non-unital symmetric monoidal structure on the $\oo$-category
  $\DsM(X^\finsurj)$  given at the lowest level of pairs of objects by 
   \[
    \left(E  \otimes^\mathrm{ch} F\right)^{(I)} = \bigoplus_{I = I_1 \amalg I_2} j_* j^* (E^{(I_1)} \boxtimes E^{(I_2)}),
    \quad j = j_{I_1,I_2},
    \]
 with the remaining structures defined in a standard way. 
    
 \item A {\em (strict) factorization algebra} on $X$ is a strict $\Dc^!$-module $E$  endowed with a coalgebra structure $E \to E \otimes^\mathrm{ch} E$ such that for any surjection $\alpha \colon I \dra J$, the morphism induced by the iterated comultiplication
    \[
    j^*_\alpha E^{(I)} \lra j^*_\alpha\biggl( \bigboxtimes_{j \in J} E^{(I_j)}\biggr)
    \]
    is an equivalence.
\end{statements}
\end{Defi}

\paragraph{The chiral tensor product on lax $\Dc^!$-modules via Day convolutions.}
We now consider $\Mod^!_\Dc(X^\finsurj)$, the $\oo$-category of lax $\Dc^!$-modules on the diagram
$X^\finsurj$. Recall (Definition \ref{def:D-mod-diagram}) that it is the op-lax limit
\[
\Mod^!_\Dc(X^\finsurj) \,:=\, \oplaxlim \Dc^!_{X^\finsurj}. 
\]
Here $\Dc^!_{X^\finsurj} \colon \finsurj \to \Cat_\infty$ maps $I$ to  the $\oo$-category $D(\QCoh_{\Dc,X^I})$ and a 
surjection
$\alpha: I \twoheadrightarrow J$ to the !-pullback $\infty$-functor
\[
\delta_\alpha^!: D(\QCoh_{\Dc,X^I}) \lra D(\QCoh_{\Dc,X^J}). 
\]
Since each $\delta_\alpha^!$ admits a left adjoint $\delta_{\alpha *}$, the Cartesian fibration $\Groth(\Dc^!_{X^\finsurj}) \to \finsurj^\op$ is also a coCartesian fibration. It is classified by the functor
$\Dc_*^{X^\finsurj} \colon \finsurj^\op \to \Cat_\oo$ mapping a finite set $I$ to $D(\QCoh_{\Dc,X^I})$ and a surjection 
$\alpha$ as above to the associated $*$-pushforward $\infty$-functor $\delta_{\alpha *}$.
This induces a tautological equivalence 
\[
  \LDsM(X^\finsurj)\, :=\,  \oplaxlim \Dc^!_{X^\finsurj} \,\simeq \, \laxlim \Dc_*^{X^\finsurj}.
\]
Further, it follows from \cite[\S 2.2.3]{FG} that $\Dc^!_{X^\finsurj}$ is endowed with 
the structure of a lax monoidal functor given (at the lowest level of pairs of objects) by
\begin{equation}\label{eq:laxmonoidal-chiral}
j_* j^*( - \boxtimes -) \colon D(\QCoh_{\Dc,X^{I_1}}) \times D(\QCoh_{\Dc,X^{I_2}}) \lra D(\QCoh_{\Dc,X^{I_1 \amalg I_2}})
\end{equation}
for $j = j_{I_1,I_2}$. Therefore, using \autoref{prop:grothmonod}, we get a non-unital symmetric monoidal structure on the Grothendieck construction $\Groth(\Dc^!_{X^\finsurj})$ compatible with the projection
\[
p \colon \Groth(\Dc^!_{X^\finsurj}) \to \finsurj^\op.
\]
It is given on two objects $(I_1, E_1)$ and $(I_2, E_2)$ in $\Groth(\Dc^!_{X^\finsurj})$ by the formula
\begin{equation}\label{eq:monoidal-Groth-D!}
(I_1, E_1) \otimes (I_2, E_2) = (I_1 \amalg I_2, j_*j^*(E_1 \boxtimes E_2)).
\end{equation}
The $\oo$-functor $\Dc^!_{X^\finsurj} \colon \finsurj \to \Cat_\infty$ is covariant. 
Writing $\finsurj$ as $(\finsurj^\op)^\op$, we can view  $\Dc^!_{X^\finsurj} $
as a contravariant functor from $\finsurj^\op$ to $\Cat_\oo$ and apply 
Corollary \ref{cor:day-oplaxlim}. Indeed,  $\finsurj^\op$ is of finite
codecompositions (i.e., $\finsurj$ is of finite decompositions)  by 
\ref{ex:finsurj-decomposition} and each $\oo$-category 
 $D(\QCoh_{\Dc,X^I})$ admits finite direct products (which reduce to direct sums). 
  \autoref{cor:day-oplaxlim} implies that the $\infty$-category $\LDsM(X^\finsurj)$ carries a right Day convolution product $\rday$.
  
 \vskip .2cm 
  
    On the other hand,  the functors \eqref{eq:laxmonoidal-chiral} also define a lax monoidal structure on the functor 
  $\Dc_*^{X^\finsurj}$.
By \autoref{prop:convonlaxlim} and \ref{ex:finsurj-codecomposition}, there is a \emph{left} Day convolution product $\lday$ on $\laxlim \Dc_*^{X^\finsurj}$ (and thus on $\LDsM(X^\finsurj)$). 

\begin{prop}\label{prop:left-right-Day-agree}
 The left and right Day convolution symmetric monoidal structures on $\LDsM(X^\finsurj)$ are equivalent.
\end{prop}

\begin{Defi}\label{Defi:chiral-lax}
   We denote by $\otimes^\mathrm{ch}$ the non-unital symmetric monoidal structure $\rday \simeq \lday$ on $\LDsM(X^\finsurj)$. We call it the \emph{chiral tensor product}.
\end{Defi}

To prove Proposition \ref{prop:left-right-Day-agree}, we analyze the two Day convolutions separately. 

\paragraph{The right Day convolution explicitly.} 
  
\begin{lem}\label{lem:rday-is-chiral}
 The above right Day convolution product $\rday$ is given on two lax $\Dc^!$-modules $E$ and $F$ by the formula
 \[
  \left(E \rday F\right)^{(I)} \simeq \bigoplus_{I = I_1 \amalg I_2} j_* j^* (E^{(I_1)} \boxtimes E^{(I_2)})
 \]
 where $j = j_{I_1,I_2}$ is as in \ref{def:jI1I2}.
\end{lem}
\begin{proof}
The lax monoidal structure on $\Delta$ is given by the same Eq. 
\eqref{eq:laxmonoidal-chiral} as for $\Dc^!_{X^\finsurj}$. 
For notational convenience, denote the monoidal operation in $\finsurj^\op$
by $\amalg^\circ$ (it is still given by disjoint union). 
 Then, for $I\in\Ob(\finsurj^\op)
=\Ob(\finsurj)$ we have
$I/\amalg^\circ = (\amalg/I)^\op$. 
 Applying \eqref{eq:right-day-sections} to $\Delta$, we get a representation 
 of $E\rday F$
as a limit over $I/\amalg^\circ$ which we write as a limit over 
$(\amalg/I)^\op$ in the form 
 \[
  \left(E \rday F\right)^{(I)} \simeq \holim_{(b: J_1 \amalg J_2 \twoheadrightarrow I)\in \amalg/I}
    G_b, \hspace{1cm} G_b\,=\, \delta_b^! j_* j^* (E^{(J_1)} \boxtimes F^{(J_2)}),
 \]
 where $\delta_b \colon X^I \to X^{J_1} \times X^{J_2}$ is the diagonal embedding and 
 $j= j_{J_1,J_2}$ is
 as above. 
 More precisely, 
   note that the diagram $X^\finsurj$, i.e., the functor $I\mapsto X^I$,
 is labelled by the category $\finsurj^\op$. Therefore, taking into
 account the definition of $\Dc^!$-modules on a diagram,
 see Remark \ref{rems:*!-D-modules}(b), 
 we see that the correspondence $b\mapsto G_b$ is a  $\oo$-functor from
  $(\amalg/I)^\op$ to  $D(\QCoh_{\Dc,X^I})$.   
This means  that
  \ref{ex:finsurj-decomposition} which gives a cofinal discrete subset in 
  $\amalg/I$ (coinitial in $(\amalg/I)^\op$), identifies    
  the above homotopy limit  with:
  \[
   \left(E \rday F\right)^{(I)} \simeq \bigoplus_{I = I_1 \cup I_2} \delta^! j_* j^* (E^{(I_1)} \boxtimes F^{(I_2)}).
  \]
 Here the limit is  taken over pairs $(I_1, I_2)$ of nonempty subsets of $I$
 such that $I=I_1\cup I_2$ (not necessarily disjoint) and 
  $\delta \colon X^I = X^{I_1 \cup I_2} \hookrightarrow X^{I_1} \times X^{I_2}$
  is the diagonal embedding associated to the surjection $I_1\amalg I_2\to I$. Now, if $I_1 \cap I_2 \neq \varnothing$, then we see
  from the definition of $j$ that 
  the image of $\delta$ lies in the closed complement of the image of $j$. In this case $\delta^! j_* \simeq 0$. The above  sum therefore simplifies into the announced formula over  disjoint $(I_1, I_2)$. 
\end{proof}

\paragraph{The left Day convolution and comparison.} 
 Using Eq.\eqref{eq:left-day-sections}    and  \ref{ex:finsurj-codecomposition}, we find directly the formula
\begin{equation}\label{eq:formula-lday}
  \left(E \lday F\right)^{(I)} \simeq \bigoplus_{I = I_1 \amalg I_2} j_* j^* \left( E^{(I_1)} \boxtimes F^{(I_2)}\right).  
\end{equation}
This is identical to what is given by Lemma \ref{lem:rday-is-chiral} for $E\rday F$
and will be the main ingredient in the proof below. 

\vskip .2cm

\begin{proof}[Proof of \autoref{prop:left-right-Day-agree}]
 We construct a non-unital symmetric monoidal structure on the equivalence 
 $\xi: \oplaxlim \Dc^!_{X^\finsurj}\buildrel \simeq \over\to\laxlim \Dc_*^{X^\finsurj}$. 
 That is, we extend it to an equivalence 
$\Xi: (\oplaxlim \Dc^!_{X^\finsurj})^\rday \buildrel \simeq \over\to
(\laxlim \Dc_*^{X^\finsurj})^\lday$ of coCartesian fibrations over $\finsurj^*$,
which means simply an equivalence of $\oo$-categories over $\finsurj^*$
(which will then automatically send coCartesian morphisms in the source to
coCartesian morphisms in the target). 

Let us first  understand what is involved in constructing $\xi$  as a {\em lax} monoidal functor.
By Definition \ref{def:lax-mo-functor},  this amounts to constructing  an $\oo$-functor
(not  claimed to be an equivalence)
 $\Xi$ as above which is a morphism of $\oo$-operads, i.e., 
  sends coCartesian morphisms in the source lying over
 inert morphisms in $\finsurj^*$, to coCartesian morphisms in the target. 
 By definition \eqref{eq:laxlim=norm}
 \[
 (\laxlim \Dc_*^{X^\finsurj})^\lday\,=\, \Nm_{(\finsurj^\op)^{\amalg^\op}/\finsurj^*}
 \bigl(\Grothco(\Dc_*^{X^\finsurj})^\otimes\bigr). 
 \]
 Here $(\finsurj^\op)^{\amalg^\op}\to\finsurj^*$,  is the coCartesian fibration classifying the symmetric monoidal category
 $(\finsurj^\op,\amalg^\op)$, while  resp. 
 $(\Grothco(\Dc_*^{X^\finsurj})^\otimes \to(\finsurj^\op)^{\amalg^\op}$ is the categorical fibration
 representing the monoidal functor 
$(\Grothco(\Dc_*^{X^\finsurj}), \otimes) \to (\finsurj^\op, \amalg^\op)$
whose underlying $\oo$-functor $\Grothco(\Dc_*^{X^\finsurj})\to \finsurj$
is the coCartesian fibration classifying the $\oo$-functor
$\Dc_*^{X^\finsurj}: \finsurj^\op\to\Cat_\oo$.

 Let us apply   the universal property of the left Day convolution 
 as formulated in \cite[Def. 2.2.6.1]{lurie-ha} to
 \[
 \begin{gathered}
 \Cc^\otimes = (\finsurj^\op)^{\amalg^\op}, \,\,\, \Oc^\otimes =\finsurj^*,\,\, \,
 \wt\Cc^\otimes = \Grothco(\Dc_*^{X^\finsurj})^\otimes, \\
  \wt\Oc^\otimes= (\laxlim \Dc_*^{X^\finsurj})^\lday,\,\,\,
   \Oc'{}^\otimes =  (\oplaxlim \Dc^!_{X^\finsurj})^\rday. 
 \end{gathered}
 \]
 It gives an equivalence of $\oo$-categories
 \[
 \Alg_{\Oc'/\Oc}(\wt\Oc) \buildrel\simeq\over\lra \Alg_{\Oc'\times_\Oc\Cc /\Cc}(\wt\Cc). 
 \]
 Here, by definition, the source $\oo$-category consists of morphisms of 
 $\oo$-operads
 $\Oc'{}^\otimes\to\wt\Oc^\otimes$ compatible with the projection to $\Oc^\otimes$,
 which in our situation amount to data
 $\Xi: (\oplaxlim \Dc^!_{X^\finsurj})^\rday  \to
(\laxlim \Dc_*^{X^\finsurj})^\lday$  of a lax monoidal functor which we seek. 
 The target $\oo$-category consists, in a similar interpretation, of
  lax monoidal structures on the evaluation functor
 \[
  \theta \colon \oplaxlim \Dc^!_{X^\finsurj} \times \finsurj 
   \lra \Grothco(\Dc_*^{X^\finsurj})\simeq \Groth(\Dc^!_{X^\finsurj})
 \]
 that commute with the monoidal structures of the projections to $\finsurj$.
 
 Now, the monoidal structure $\rday$ on the $\oo$-category $\oplaxlim \Dc^!_{X^\finsurj}$, i.e., the $\oo$-category $(\oplaxlim \Dc^!_{X^\finsurj})^\rday$, is defined by the norm
 construction for $\oo$-cooperads dual to that for $\oo$-operads. 
 In particular, we have the morphism, corresponding to $\alpha$ in 
 \cite[Def. 2.2.6.1]{lurie-ha}, which exhibits $(\oplaxlim \Dc^!_{X^\finsurj})^\rday$
 as the norm. After examination, we find that this $\alpha$ is nothing but
  a \emph{colax} monoidal structure  on $\theta$, which we denote $\Theta$. 
  Formally, $\Theta$ is an $\oo$-functor between Cartesian fibrations over $\finsurj^*$. 
  
  We now modify $\Theta$ to a lax monoidal structure $\Theta'$.  This 
  will imply a construction of a lax monoidal  
  structure on $\xi$, i.e., a morphism  $\Xi$  of coCartesian fibrations as above.  
    
  To modify $\Theta$  we first consider its action on  Cartesian edges
  lying over the arrow $\{1,2\}^*\to \{1\}^*$, i.e., its effect on pairs of objects
  $(E,I)$ and $(F,J)$, with  $E,F \in \oplaxlim \Dc^!_{X^\finsurj}$ and $I, J \in \finsurj$. 
  It is given by the projection
  \begin{multline}\label{eq:proj-dir-summand}
  \theta(E \rday F, I \amalg J) = (E \rday F)^{(I \amalg J)} \simeq \bigoplus_{I \amalg J = I_1 \amalg I_2} j_* j^* \left( E^{(I_1)} \boxtimes F^{(I_2)}\right) \\ \lra j_* j^* \left( E^{(I)} \boxtimes F^{(J)}\right)
 \end{multline}
  on a direct summand. This projection has an obvious canonical section   (embedding of a summand into the sum).
 
 Similarly, the action of $\Theta$ on  Cartesian edges lying over $\{1,\cdots, n\}^*\to\{1\}^*$
 corresponds to the obvious $n$-fold analog of the projection \eqref{eq:proj-dir-summand}, 
 for $n$ pairs $(E_1, I_1),\cdots, (E_n, I_n)$ and has a similar canonical section.

 The action of $\Theta$ on an arbitrary Cartesian edge in $\finsurj^*$ is obtained as an exterior product of the above actions, so it acquires a canonical section as well. These canonical sections
 define $\Theta'$ as required.
 
 Now,  \autoref{lem:rday-is-chiral} together with equation \eqref{eq:formula-lday} shows that the obtained lax monoidal structure on the tautological equivalence $\oplaxlim \Dc^!_{X^\finsurj} \simeq \laxlim \Dc_*^{X^\finsurj}$ is actually monoidal.
 \end{proof}

 \paragraph{The $*$-tensor product.}
The box-product
\[
  - \boxtimes - \colon D(\QCoh_{\Dc,X^{I_1}}) \times D(\QCoh_{\Dc,X^{I_2}}) \lra D(\QCoh_{\Dc,X^{I_1 \amalg I_2}})
\]
defines another lax monoidal structure on $\Dc_*^{X^\finsurj}$ and, by the same procedure, another left convolution on $\LDsM(X^\finsurj)$.
\begin{Defi}\label{Defi:*-product}
We denote by $\otimes^*$ and call the $*$-product the non-unital symmetric monoidal structure constructed above:
\[
  \left(E \otimes^* F\right)^{(I)} \simeq \bigoplus_{I = I_1 \amalg I_2} E^{(I_1)} \boxtimes F^{(I_2)}.
\]
\end{Defi}
The $*$-tensor product can be compared with the chiral tensor product:
\begin{lem}\label{lem:*-to-chiral}
     The identity functor $(\LDsM(X^\finsurj),\otimes^*) \to (\LDsM(X^\finsurj), \otimes^\mathrm{ch})$ carries a canonical colax monoidal structure given by
     \begin{multline*}
      \left(E \otimes^* F\right)^{(I)} \simeq \bigoplus_{I = I_1 \amalg I_2} E^{(I_1)} \boxtimes F^{(I_2)} \\ \lra \bigoplus_{I = I_1 \amalg I_2} (j_{I_1,I_2})_* j_{I_1,I_2}^* \left( E^{(I_1)} \boxtimes F^{(I_2)} \right) \simeq \left( E \otimes^\mathrm{ch} F \right)^{(I)}.
     \end{multline*}
\end{lem}
\begin{proof}
Denote by $(\Dc_*^{X^\finsurj},*)$ the functor $\Dc_*^{X^\finsurj}$ equipped with its lax monoidal structure coming from the box product, and by $(\Dc_*^{X^\finsurj}, \mathrm{ch})$ the same $\infty$-functor but equipped with the lax monoidal structure induced by equation \eqref{eq:laxmonoidal-chiral}.
The adjunction units $\Id \to j_* j^*$ (associated to open immersions of the form $j \colon U(\alpha) \to X^I$ for $\alpha \colon I \twoheadrightarrow J$) define a colax monoidal structure on the identical natural transformation
\[
  (\Dc_*^{X^\finsurj},*) \Rightarrow (\Dc_*^{X^\finsurj},\mathrm{ch}).
\]
By \autoref{cor:laxmonoidalonlaxlimits}, it endows the identity of $\LDsM(X^\finsurj)$ with the announced colax monoidal structure.
\end{proof}

\paragraph{Lax $\Dc^!$-factorization algebras:}

\begin{Defi}\label{Defi:factlax}
    A lax pre-$\Dc^!$-factorization algebra is a coalgebra for the chiral tensor product in the $\infty$-category $\LDsM(X^\finsurj)$.
    A lax $\Dc^!$-factorization algebra is a lax pre-$\Dc^*$-factorization algebra such that such that for any surjection $\alpha \colon I \dra J$, the induced morphism
    \[
    j^*_\alpha E^{(I)} \to j^*_\alpha\biggl( \bigboxtimes_{j \in J} E^{(I_j)}\biggr)
    \]
    is an equivalence.
    
    We denote by $\FA^!_\Dc(X)$ the $\infty$-category of lax $\Dc^!$-factorization algebras over $X$. We denote by $\strict\FA^!_\Dc(X)$ the full 
    $\infty$-subcategory of $\FA^!_\Dc(X)$ spanned by lax $\Dc^!$-factorization algebras on $X$ whose underlying lax $\Dc^!$-module is strict.
\end{Defi}

\begin{rems}~
 \begin{statements}[ref=Remark \therems (\alph*)]
  \item The $\infty$-category $\strict\FA^!_\Dc(X)$ is exactly that of factorization algebras in the sense of \autoref{Defi:factFG}.
  \item \label{rem:FA-colax-monoidal-sections} Since by definition, the chiral tensor product is a right Day convolution, it follow from the universal property of Day convolutions that a coalgebra structure amounts to a colax monoidal structure on the corresponding section.
  
  In particular, a lax pre-$\Dc^!$-factorization algebra is simply a colax monoidal section of the canonical projection $\Groth(\Dc^!_{X^\finsurj}) \to \finsurj^\op$ (where the monoidal structure on $\Groth(\Dc^!_{X^\finsurj})$ is that of equation \eqref{eq:monoidal-Groth-D!}).
 \end{statements}
\end{rems}

\paragraph{Lax $[\![\Dc]\!]$-factorization algebras:}
The constructions and definitions in this case are simply dual to those in the case of $\Dc^!$-modules. Since we will no formally need the various tensor structures in this case, we give only a definition following the idea of \ref{rem:FA-colax-monoidal-sections}.

The $\infty$-functor $\Dc^{[\![*]\!]}_{X^\finsurj}$ mapping $I$ to $\Pro(\Perf_{\Dc, X^I})$ admits a lax monoidal structure given by the formula $j_{[\![!]\!]} j^{[\![*]\!]}( - \boxtimes -)$:
\[
\Pro(\Perf_{\Dc, X^{I_1}}) \times \Pro(\Perf_{\Dc, X^{I_2}}) \lra \Pro(\Perf_{\Dc, X^{I_1 \amalg I_2}})
\]
with $j = j_{I_1,I_2}$. \autoref{prop:grothmonod} shows that the (coCartesian) Grothendieck construction $\Grothco(\Dc^{[\![*]\!]}_{X^\finsurj})$ admits a symmetric monoidal structure compatible with the projection
\[
q \colon \Grothco(\Dc^{[\![*]\!]}_{X^\finsurj}) \to \finsurj
\]

\begin{Defi}
    A lax $[\![\Dc]\!]$-factorization algebra is a lax $[\![\Dc]\!]$-module, seen as a section of $q$, endowed with a lax monoidal structure, such that for any surjection $\alpha \colon I \dra J$, the induced morphism
    \[
    j^{[\![*]\!]}_\alpha\left( \bigboxtimes_{j \in J} E^{(I_j)}\right) \to j^{[\![*]\!]}_\alpha E^{(I)}
    \]
    is an equivalence.
    
    We denote by $\FA_{[\![\Dc]\!]}(X)$ (resp. $\strict\FA_{[\![\Dc]\!]}(X)$) the $\infty$-category of lax (resp. strict) $[\![\Dc]\!]$-factorization algebras.
\end{Defi}

The following is obvious:
\begin{prop}
    Verdier duality induces equivalences
    \[
    \FA_{[\![\Dc]\!]}(X)^\op \simeq \FA^!_\Dc(X) \hspace{2em} \text{and} \hspace{2em} \strict\FA_{[\![\Dc]\!]}(X)^\op \simeq \strict\FA^!_\Dc(X)
    \] 
    compatible with the equivalence of \autoref{prop:Verdierduality}.
\end{prop}


\subsection{Definition in terms of arrow categories}\label{subsec:straightarrows}

In the above definitions, factorization structures for $\Dc^!$- and $[\![\Dc]\!]$-modules are dual to one another. The factorization structure is in one case a coalgebra structure, and in the other case an algebra structure.
In order to prove that covariant Verdier duality preserves the factorization structures, we will need alternative models for factorization algebras that do not make a choice between algebras and coalgebras.

\begin{Defi}
    We denote by $\finsurj^1$ the category of arrows in $\finsurj$. A morphism $\sigma$ from $\alpha \colon I \dra J$ to $\beta \colon S \dra T$ in $\finsurj^1$ is a 
    commutative diagram
    \[
    \xymatrix{
        I \ar@{->>}[r]^{\overrightarrow{\scriptstyle \sigma}} \ar@{->>}[d]_\alpha & S \ar@{->>}[d]^\beta \\ J \ar@{->>}[r]_{\underrightarrow{\scriptstyle \sigma}} & T \ar@{}[ul]|{\displaystyle\sigma}.
    }
    \]
    Disjoint union makes $\finsurj^1$ into a symmetric monoidal category.
    
    A morphism $\sigma$ as above is called an {\em open} (resp. {\em closed}) morphism if $\overrightarrow{\sigma}$ (resp. $\underrightarrow{\sigma}$) is a bijection.
\end{Defi}

\begin{lem}\label{lem:S1-finite-dec-codec}
 The category $\finsurj^1$ is both of finite decomposition and finite codecomposition. Moreover, for $\alpha \colon I \twoheadrightarrow J$, we have
 \begin{gather*}
  \Dec_n(\alpha) \simeq \{ (\alpha_1, \dots, \alpha_n) ~|~ \alpha_i \colon I_i \twoheadrightarrow J_i,~ I = I_1 \cup \cdots \cup I_n~, J = J_1 \cup \cdots \cup J_n \},
  \\
  \Codec_n(\alpha) \simeq \{ (\alpha_1, \dots, \alpha_n) ~|~ \alpha_i \colon I_i \twoheadrightarrow J_i,~ I = I_1 \amalg \cdots \amalg I_n~, J = J_1 \amalg \cdots \amalg J_n \}.
 \end{gather*}
\end{lem}
\begin{proof}
 Let $\beta_1, \dots, \beta_n \in \finsurj^1$ and write $\beta_i \colon S_i \twoheadrightarrow T_i$. A morphism $\alpha \to \beta_1 \amalg \cdots \amalg \beta_n$ is a commutative diagram
 \[
  \xymatrix{
   I \ar@{->>}[d]_{\alpha} \ar@{->>}[r] & S_1 \amalg \cdots \amalg S_n \ar@{->>}[d]^{\beta_1 \amalg \cdots \amalg \beta_n} \\
   J \ar@{->>}[r] & T_1 \amalg \cdots \amalg T_n.
  }
 \]
 Setting $I_i$ (resp. $J_i$) as the pre-image of $T_i$ in $I$ (resp. in $J$), we find the announced explicit description of $\Codec_n(\alpha)$.
 
 Similarly, a morphism $\beta_1 \amalg \cdots \amalg \beta_n \to \alpha$ is a commutative diagram
 \[
  \xymatrix{
   S_1 \amalg \cdots \amalg S_n \ar@{->>}[d]_{\beta_1 \amalg \cdots \amalg \beta_n} \ar@{->>}[r] & I \ar@{->>}[d]^{\alpha} \\
   \ar@{->>}[r] T_1 \amalg \cdots \amalg T_n & J.
  }
 \]
 Setting $I_i$ (resp. $J_i$) to be the image of $S_i$ in $I$ (resp. in $J$), we get the announced description on $\Dec_n(\alpha)$.
\end{proof}

Fix a commutative diagram $\sigma$ as above.
By definition, we have an open immersion $U(\alpha) \to U(\underrightarrow \sigma \alpha)$. Moreover, the closed immersion $\delta_{\overrightarrow{\scriptstyle \sigma}} \colon X^S \to X^I$ restricts to a closed immersion $U(\beta) \to U(\underrightarrow \sigma \alpha)$.
We denote by $\widehat U(f)$ the pullback
\begin{equation}\label{eq:hati-hatj}
    \vcenter{
        \xymatrix{
            \hatU(f) \ar[r]^{\hatj(\sigma)} \ar[d]_{\hati(\sigma)} \ar@{-}[]+(6,-2);[]+(6,-6)\ar@{-}[]+(2,-6);[]+(6,-6) & U(\beta) \ar[d] \\ U(\alpha) \ar[r] & U(\underrightarrow\sigma\alpha).
        }
    }
\end{equation}
Note that the horizontal maps are open immersions and the vertical ones are closed immersions.
If $\sigma$ is an open (resp. a closed) morphism, then $\hati(\sigma)$ (resp. $\hatj(\sigma)$) is an isomorphism.

If it exists, let $\gamma$ be a surjection $S \dra J$ such that $\gamma \overrightarrow \sigma = \alpha$. Since $\overrightarrow \sigma$ is surjective, such a map $\gamma$ is unique if it exists, and automatically satisfies $\underrightarrow\sigma \gamma = \beta$. We then have
\[
\hatU(f) := \begin{cases} U(\gamma) & \text{if } \gamma \text{ exists} \\ \emptyset & \text{else.} \end{cases}
\]

The assignment $\alpha \mapsto U(\alpha)$, $\sigma \mapsto (U(\beta) \leftarrow \hatU(\sigma) \to U(\alpha))$ defines a functor
\[
\hatU \colon (\finsurj^1)^\op \to \Var_\k^\corr.
\]
Recall the functor $\Dcorr \colon \Var_\k^\corr \to \Cat_\infty$ mapping a variety $Y$ to $D(\QCoh_{\Dc,Y})$ and a correspondence $Y_1 \underset a\leftarrow Z \underset b\to Y_2$ to $b_* a^!$.

\begin{Defi}
    We denote by $\Dcorr_{\hatU}$ the composite $\infty$-functor
    \[
    \Dcorr_{\hatU} := \Dcorr \circ \hatU \colon (\finsurj^1)^\op \to \Cat_\infty.
    \]
\end{Defi}

The functor $\hatU$ admits a lax-monoidal structure, given by the open immersions $U(\alpha_1 \amalg \alpha_2) \to U(\alpha_1) \times U(\alpha_2)$.
Composed with the lax monoidal structure on $\Dcorr$, we get a lax monoidal structure on $\Dcorr_\hatU$ and hence a monoidal structure on $\Grothco(\Dcorr_\hatU)$.

\begin{Defi}
    We say that a section of $\Grothco(\Dcorr_\hatU) \to (\finsurj^1)^\op$ is openly coCartesian if it sends every open morphisms in $\finsurj^1$ to a coCartesian morphism.
    We denote by $\Mod^\corr_\Dc(\hatU)$ the full sub-$\infty$-category of $\laxlim \Dcorr_\hatU$ spanned by openly coCartesian sections.
\end{Defi}
    
Notice that for any morphism $\sigma \colon \alpha \to \beta \in \finsurj^1$, the functor $\Dcorr_\hatU(\sigma) = \hati(\sigma)_* \hatj(\sigma)^*$ admits a right adjoint $\hatj(\sigma)_* \hati(\sigma)^!$. Denote by $\ol \Dc {}^{\on{corr}}_\hatU$ the corresponding functor $\finsurj^1 \to \Cat_\oo$ so that 
\[
 \laxlim \Dcorr_\hatU = \oplaxlim \ol \Dc {}^{\on{corr}}_\hatU.
\]
Since $\finsurj^1$ is of finite codecompositions (see \autoref{lem:S1-finite-dec-codec}), we get using \autoref{cor:day-oplaxlim} a right Day convolution structure on $\oplaxlim \ol \Dc {}^{\on{corr}}_\hatU$. This tensor structure preserves openly coCartesian sections and thus defines a monoidal structure on $\Mod^\corr_\Dc(\hatU)$.

\begin{prop}\label{prop:monod-!-arrows}
    Restriction along the functor $\eta \colon \finsurj \to \finsurj^1$ given by $I \mapsto (I \dra *)$ induces a symmetric monoidal equivalence
    \[
    B \colon \Mod^\corr_\Dc(\hatU) \simeq \LDsM(X^\finsurj).
    \]
\end{prop}

\begin{proof}
    Let $s$ be the source functor $\finsurj^1 \to \finsurj$, mapping a surjection $\alpha$ to its domain.
    It is a monoidal functor, and thus induces a colax monoidal $\infty$-functor (by \autoref{prop:restrictionlaxmonoidal})
    \[
    \oplaxlim \Dc^*_{X^\finsurj} \to \oplaxlim (\Dc^*_{X^\finsurj} \circ s).
    \]
    The canonical immersions $j_\alpha \colon U(\alpha) \to X^I$ define a monoidal natural transformation $j^*_\bullet \colon \Dc^!_{X^\finsurj} \circ s \to \ol \Dc {}^{\on{corr}}_\hatU$.
    In particular, it defines a monoidal $\infty$-functor (see \autoref{cor:laxmonoidalonlaxlimits})
    \[
    \oplaxlim (\Dc^!_{X^\finsurj} \circ s) \to \oplaxlim \ol \Dc {}^{\on{corr}}_\hatU
    \]
    whose image lies in $\on{Mod}^\corr_\Dc(\hatU)$.
    We find a monoidal $\infty$-functor
    \[
    A \colon \LDsM(X^\finsurj) \simeq  \oplaxlim \Dc^!_{X^\finsurj} \to \oplaxlim (\Dc^!_{X^\finsurj} \circ s) \to \on{Mod}^\corr_\Dc(\hatU).
    \]
    The restriction along $\eta \colon \finsurj \to \finsurj^1$ gives an $\infty$-functor
    \[
    B \colon \Mod^\corr_\Dc(\hatU) \to \LDsM(X^\finsurj)
    \]
    which is inverse to $A$.
\end{proof}

\begin{Defi}
    Let $\FA^\corr_\Dc(\hatU)$ denote the $\infty$-category of sections in $\Mod^\corr_\Dc(\hatU)$ endowed with symmetric monoidal structure.
\end{Defi}

\begin{cor}\label{cor:shriekvsstraightarrows}
    The functor $\finsurj \to \finsurj^1$ mapping $I$ to $I \dra *$ induces an 
    equivalence
    \[
    B \colon \FA^\corr_\Dc(\hatU) \simeq \FA_\Dc^!(X).
    \]
\end{cor}

\begin{proof}
    The equivalence $B$ of \autoref{prop:monod-!-arrows} being monoidal, it preserves coalgebras. Coalgebras for the (right) Day convolution are colax monoidal $\infty$-functor (see \cite[Prop. 2.12]{glasman}).
    Let $\Ac \in \FA^\corr_\Dc(\hatU)$. It is a section with a monoidal structure, and thus $B(\Ac)$ is is a colax monoidal section.
    The factorizing property of $B(\Ac)$ corresponds to the fact that $\Ac$ is (strictly) monoidal.
\end{proof}

The dual statements obviously hold for $[\![\Dc]\!]$-factorization algebras.
Consider the $\infty$-functor $\Dc^{[\![\corr]\!]}_V \colon (\finsurj^1)^\op \to \Cat_\infty$ mapping $\alpha$ to $\Pro (\Perf_{\Dc,U(\alpha)})$ and $\sigma \colon \alpha \to \beta$ to $\hati(\sigma)_{[\![*]\!]} \circ \hatj(\sigma)^{[\![*]\!]}$ (using the notations of Eq. \eqref{eq:hati-hatj}).
It is naturally lax monoidal and $\Groth(\Dc^{[\![\corr]\!]}_V)$ inherits a symmetric monoidal structure.

\begin{Defi}
    A section of $\Groth(\Dc^{[\![\corr]\!]}_V) \to \finsurj^1$ is called openly Cartesian if it maps open morphisms of $\finsurj^1$ to Cartesian morphisms.
    
    Denote by $\FA_{[\![\Dc]\!]}^\corr(\hatU)$ the $\infty$-category of openly Cartesian symmetric monoidal sections of $\Groth(\Dc^{[\![\corr]\!]}_V) \to \finsurj^1$.
\end{Defi}

\begin{prop}\label{prop:DfactoverS1}
    Restriction along $\eta$ induces an equivalence
    \[
    \FA_{[\![\Dc]\!]}^\corr(\hatU) \simeq \FA_{[\![\Dc]\!]}(X)
    \]
    compatible through Verdier duality with the equivalence of \autoref{cor:shriekvsstraightarrows}.
\end{prop}


\subsection{Definition in terms of twisted arrows}\label{subsec:twistedarrows}

\begin{Defi}
    We denote by $\TwS$ the category of twisted arrows in $\finsurj$. Its objects are morphisms in $\finsurj$, and its morphisms from $\alpha \colon I \dra J$ to $\beta \colon S \dra T$ are commutative diagrams
    \[
    \xymatrix{
        I \ar@{->>}[d]_\alpha \ar@{}[dr]|{\displaystyle \tau} & S \ar@{->>}[l]_{\overleftarrow{\scriptstyle \tau}} \ar@{->>}[d]^\beta \\ J \ar@{->>}[r]_{\underrightarrow{\scriptstyle \tau}} & T
    }
    \]
    in $\finsurj$. For any surjections $\alpha_1 \colon I_1 \to J_1$ and $\alpha_2 \colon I_2 \to J_2$, we set \[
    \alpha_1 \amalg \alpha_2 \colon I_1 \amalg I_2 \twoheadrightarrow J_1 \amalg J_2.
    \]
    It induces on $\TwS$ a symmetric monoidal structure.
\end{Defi}

\begin{Defi}
    Let $\tau \colon \alpha \to \beta$ be a morphism in $\TwS$ corresponding to a diagram as above.
    We say that $\tau$ is {\em open} if the map $\overleftarrow\tau$ is a bijection. We say that $\tau$ is 
    {\em closed} if the map $\underrightarrow\tau$ is a bijection.
\end{Defi}

\begin{lem}\label{lem:TwS-finite-dec-codec}
 The category $\TwS$ is both of finite decompositions and finite codecompositions. Moreover, for any $\alpha \colon I \twoheadrightarrow J$, we have
 \begin{gather*}
  \Dec_n(\alpha) \simeq \{ (\alpha_1, \dots, \alpha_n) ~|~ \alpha_i \colon I_i \twoheadrightarrow J_i,~ I = I_1 \amalg \cdots \amalg I_n~, J = J_1 \cup \cdots \cup J_n \},
  \\
  \Codec_n(\alpha) \simeq \{ (\alpha_1, \dots, \alpha_n) ~|~ \alpha_i \colon I_i \twoheadrightarrow J_i,~ I = I_1 \amalg \cdots \amalg I_n~, J = J_1 \amalg \cdots \amalg J_n \}.
 \end{gather*}

\end{lem}

\begin{proof}
 Let $\beta_1, \dots, \beta_n \in \TwS$ and write $\beta_i \colon S_i \twoheadrightarrow T_i$. A morphism $\alpha \to \beta_1 \amalg \cdots \amalg \beta_n$ is a commutative diagram
 \[
  \xymatrix{
   I \ar@{->>}[d]_{\alpha} & S_1 \amalg \cdots \amalg S_n \ar@{->>}[d]^{\beta_1 \amalg \cdots \amalg \beta_n} \ar@{->>}[l] \\
   J \ar@{->>}[r] & T_1 \amalg \cdots \amalg T_n.
  }
 \]
 Setting $I_i$ (resp. $J_i$) as the pre-image of $T_i$ in $I$ (resp. in $J$), we find the announced explicit description of $\Codec_n(\alpha)$.
 
 Similarly, a morphism $\beta_1 \amalg \cdots \amalg \beta_n \to \alpha$ is a commutative diagram
 \[
  \xymatrix{
   S_1 \amalg \cdots \amalg S_n \ar@{->>}[d]_{\beta_1 \amalg \cdots \amalg \beta_n} & I \ar@{->>}[d]^{\alpha} \ar@{->>}[l] \\
   \ar@{->>}[r] T_1 \amalg \cdots \amalg T_n & J.
  }
 \]
 Setting $I_i$ to be the pre-image of $T_i$ in $I$, and $J_i$ to be its image in $J$, we find the announced description on $\Dec_n(\alpha)$.
\end{proof}

Consider the diagram $U_\Tw \colon \TwS \to \Var_\k$ mapping $\alpha \colon I \to J$ to $U(\alpha)$. It maps a commutative diagram $\tau$ as above to the natural immersion $U(\alpha) \to U(\beta)$.
Note that it maps open (resp. closed) morphisms in $\TwS$ to open (resp. closed) immersions of varieties.
The $\infty$-functor $\Dc^!_{U_\Tw} \colon \TwS^\op \to \Cat_\infty$ (recall the notation from Section \ref{sec:laxmodules} \S \ref{par:modulesondiagrams}) has a lax monoidal structure, given by
\[
j^*(- \boxtimes -) \colon D(\QCoh_{\Dc,U(\alpha_1)}) \times D(\QCoh_{\Dc,U(\alpha_2)}) \to D(\QCoh_{\Dc, U(\alpha_1 \amalg \alpha_2)}).
\]
with $j \colon U(\alpha_1 \amalg \alpha_2) \to U(\alpha_1) \times U(\alpha_2)$ the open embedding.
We get from \autoref{prop:grothmonod} a symmetric monoidal structure on $\Groth(\Dc^!_{U_\Tw})$.

\begin{Defi}
    A section of $\Groth(\Dc^!_{U_\Tw}) \to \TwS$ is called openly Cartesian if it maps open morphisms in $\TwS$ to Cartesian morphisms.
    We denote by $\Mod^{!o}_\Dc(U_\Tw)$ the full sub-$\infty$-category of $\oplaxlim \Dc^!_{U_\Tw}$ spanned by openly Cartesian sections. Since $\TwS$ is of finite decomposition (see \autoref{lem:TwS-finite-dec-codec}), \autoref{cor:day-oplaxlim} defines on $\oplaxlim \Dc^!_{U_\Tw}$ a right Day convolution. The full sub-$\infty$-category $\Mod^{!o}_\Dc(U_\Tw)$ is stable by this tensor product, and thus inherits a symmetric monoidal structure.
\end{Defi}

\begin{Defi}
    We denote by $\FA_{\Dc}^!(U_\Tw)$ the $\infty$-category of openly Cartesian symmetric monoidal sections of $\Groth(\Dc^!_{U_\Tw})$. Forgetting the monoidal structure defines an $\infty$-functor
    \[
    \FA_{\Dc}^!(U_\Tw) \lra \Mod^{!o}_\Dc(U_\Tw).
    \]
\end{Defi}

\begin{prop}\label{prop:twistedshriek}
    Restriction along the functor $\bar \eta \colon I \mapsto (I \dra *)$ induces a symmetric monoidal equivalence of $\infty$-categories
    \[
    \Mod_{\Dc}^{!o}(U_\Tw) \simeq \LDsM(X^\finsurj).
    \]
    It induces an equivalence
    \[
    \FA_{\Dc}^!(U_\Tw) \simeq \FA^!_\Dc(X^\finsurj).
    \]
\end{prop}

\begin{proof}
    Let us denote by $\bar s \colon \TwS \to \finsurj^\op$ the monoidal functor mapping a surjection to its domain.
    The canonical open immersions $j_\alpha \colon U(\alpha) \to X^I$ (for $\alpha \colon I \dra J$) induce a monoidal natural transformation $j^!_\bullet \colon \Dc^!_{X^\finsurj} \circ \bar s \to \Dc^!_{U_\Tw}$. We find a monoidal $\infty$-functor
    \[
    \LDsM(X^\finsurj) := \oplaxlim \Dc^!_{X^\finsurj} \to \oplaxlim \Dc^!_{X^\finsurj} \circ \bar s \to \oplaxlim \Dc^!_{U_\Tw}
    \]
    whose image lie in $\Mod^{!o}_\Dc(U_\Tw)$:
    \[
    C \colon \LDsM(X^\finsurj) \lra \Mod^{!o}_\Dc(U_\Tw).
    \]
    Restriction along $\bar \eta$ gives an inverse $\infty$-functor $D \colon \Mod_\Dc^{!o}(U_\Tw) \to \LDsM(X^\finsurj)$ to $C$.
    This equivalence preserves the factorization structures.
\end{proof}


\subsection{Coherent factorization algebras}
Recall that a coherent (lax) $[\![\Dc]\!]$-modules is a lax $[\![\Dc]\!]$-module 
$(E^{[I)})$ such that $E^{(I)}$ is a coherent $\Dc$-module over 
$X^I$, for any $I$.

Equivalently, a coherent $[\![\Dc]\!]$-module can be seen as an object in the 
oplax-limit of the $\infty$-functor $\Coh_{*}^{X^\finsurj}$ mapping $I$ to $\Coh_{\Dc, 
    X^I}$ and a surjection $\alpha \colon I \to J$ to the associated pushforward 
$\infty$-functor (which preserves coherent $\Dc$-modules).

\begin{Defi}
    A lax $[\![\Dc]\!]$-factorization algebra is called lax if its underlying lax $[\![\Dc]\!]$-module is coherent. We denote by $\FA^{\Coh}_{[\![\Dc]\!]}(X)$ the full sub-$\infty$-category of $\FA_{[\![\Dc]\!]}(X)$ spanned by coherent factorization algebras.
\end{Defi}

Fix a morphism $g \colon Y \to Z$ of varieties. 
\begin{itemize}
    \item If $g$ is an open immersion, then the $\infty$-functors $g^*$ and $g^{[\![*]\!]}$ both preserve coherent $\Dc$-modules and they coincide on coherent $\Dc$-modules.
    \item If $g$ is proper, then the $\infty$-functors $g_*$ and $g_{[\![*]\!]}$ both preserve coherent $\Dc$-modules and coincide on coherent $\Dc$-modules.
\end{itemize}

In particular, the lax monoidal $\infty$-functors $\Dcorr_{\hatU}$ and $\Dc^{[\![\corr]\!]}_\hatU$ admit a common lax monoidal full sub-$\infty$-functor $\Coh^{\corr}_\hatU$ mapping a surjection $\alpha$ to $\Perf_{\Dc,U(\alpha)}$ and a morphism $\sigma \colon \alpha \to \beta$ in $\finsurj^1$ to the $\infty$-functor $\hati(\sigma)_* \circ \hatj(\sigma)^* \simeq \hati(\sigma)_{[\![*]\!]} \circ \hatj(\sigma)^{[\![*]\!]}$:
\[
\xymatrix{
    {}\overbrace{\vphantom{\coprod^d}D(\QCoh_{\Dc,U(\beta)})}^{\displaystyle \Dcorr_\hatU} \ar[d]_{\hati(\sigma)_* \circ \hatj(\sigma)^*} & {}\overbrace{\vphantom{\coprod^d}\Perf_{\Dc,U(\beta)}}^{\displaystyle \Coh^{\corr}_\hatU} \ar[d] \ar[l] \ar[r] \ar[]+(7,12.3);[r]+(-7,12.3) \ar[]+(-7,12.3);[l]+(6,12.3) & {}\overbrace{\vphantom{\coprod^d}\Pro (\Perf_{\Dc,U(\beta)})}^{\displaystyle\Dc^{[\![\corr]\!]}_\hatU} \ar[d]^{\hati(\sigma)_{[\![*]\!]} \circ \hatj(\sigma)^{[\![*]\!]}}
    \\
    D(\QCoh_{\Dc,U(\alpha)}) & \Perf_{\Dc,U(\alpha)} \ar[l] \ar[r] & \Pro (\Perf_{\Dc,U(\alpha)}).
}
\]
Applying the Grothendieck construction, we get symmetric monoidal and fully faithful $\infty$-functors
\[
\Groth(\Dcorr_\hatU) \lla \Groth(\Coh^{\corr}_\hatU) \lra \Groth(\Dc^{[\![\corr]\!]}_\hatU)
\]
over $\finsurj^1$. We find using \autoref{prop:DfactoverS1}:
\begin{prop}\label{prop:coherentfact}
    A coherent lax $[\![\Dc]\!]$-factorization algebra is tantamount to any of the following equivalent datum.
    \begin{enumerate}
        \item An openly Cartesian symmetric monoidal section of $\Groth(\Dc^{[\![\corr]\!]}_\hatU) \to \finsurj^1$ mapping any $\alpha \in \finsurj^1$ to a coherent $\Dc$-module.
        \item An openly Cartesian symmetric monoidal section of $\Groth(\Coh^{\corr}_\hatU) \to \finsurj^1$.
        \item An openly Cartesian symmetric monoidal section of $\Groth(\Dcorr_\hatU) \to \finsurj^1$ mapping any $\alpha \in \finsurj^1$ to a coherent $\Dc$-module.
    \end{enumerate}
\end{prop}

\begin{Defi}
    We denote by $\widebar\FA_\corr^\Dc(\hatU)$ the $\infty$-category of openly Cartesian symmetric monoidal sections of $\Groth(\Dcorr_\hatU)$.
    The above proposition gives a fully faithful $\infty$-functor
    \[
    \FA^\Coh_{[\![\Dc]\!]}(X) \lra \widebar\FA_\corr^\Dc(\hatU).
    \]
\end{Defi}

We shall now give another model for coherent lax $[\![\Dc]\!]$-factorization algebras. 
Consider the $\infty$-functor $\Dc_*^{U_\Tw} \colon \TwS \to \Cat_\infty$ mapping a surjection $\alpha$ to $D(\QCoh_{\Dc,U(\alpha)})$ and a morphism of twisted arrows $\tau \colon \alpha \to \beta$ to the pushforward $\infty$-functor $D(\QCoh_{\Dc,U(\alpha)}) \lra D(\QCoh_{\Dc,U(\beta)})$.
It has a lax monoidal structure, given by
\[
j^*(- \boxtimes -) \colon D(\QCoh_{\Dc,U(\alpha_1)}) \times D(\QCoh_{\Dc,U(\alpha_2)}) \to D(\QCoh_{\Dc, U(\alpha_1 \amalg \alpha_2)}).
\]
with $j \colon U(\alpha_1 \amalg \alpha_2) \to U(\alpha_1) \times U(\alpha_2)$ the open embedding.
Its Cartesian Grothendieck construction thus admits a symmetric monoidal structure.

Fix a section $E \colon \TwS^\op \to \Groth(\Dc_*^{U_\Tw})$ and an open morphism $\tau \colon \alpha \to \beta$ in $\TwS$. The transition morphism $E(\beta) \to U_\Tw(\tau)_* E(\alpha)$ induces by adjunction a morphism
\[
U_\Tw(\tau)^* E(\beta) \to E(\alpha).
\]

\begin{Defi}
    Let $\wb \Mod_*^\Dc(U_\Tw)$ denote the full sub-$\infty$-category of $\oplaxlim \Dc_*^{U_\Tw}$ spanned by sections $E$ such that for any open morphism $\tau \colon \alpha \to \beta$ in $\TwS$, the induced morphism $U_\Tw(\tau)^* E(\beta) \to E(\alpha)$ is an equivalence.
    
    Let $\wb \FA_*^\Dc(U_\Tw)$ denote the $\infty$-category of symmetric monoidal sections of $\Groth(\Dc_*^{U_\Tw})$ that belong to $\wb \Mod_*^\Dc(U_\Tw)$.
\end{Defi}

Arguments similar to those used in sections \ref{subsec:straightarrows} and \ref{subsec:twistedarrows} give:
\begin{prop}\label{prop:simplerCohFA}
    The $\infty$-categories $\widebar \FA_*^\Dc(U_\Tw)$ and $\widebar\FA_\corr^\Dc(\hatU)$ are equivalent. In particular, there is a fully faithful $\infty$-functor
    \[
    \mu \colon \FA^\Coh_{[\![\Dc]\!]}(X) \lra \widebar \FA_*^\Dc(U_\Tw)
    \]
    whose image consists of sections mapping every $\alpha \in \TwS$ to a coherent $\Dc$-module.
\end{prop}


\subsection{Covariant Verdier duality}\label{subsec:CVD-FA}

We can now prove the following
\begin{thm}\label{thm:coVerdier-fact}
    The covariant Verdier duality $\infty$-functor
    \[
    \psi \colon \Coh_{[\![\Dc]\!]}(X^\finsurj) \to \DsM(X^\finsurj)
    \]
    preserves factorization structures. In other words, it extends to an $\infty$-functor
    \[
    \psi \colon \FA^\Coh_{[\![\Dc]\!]}(X) \to \FA^!_\Dc(X).
    \]
\end{thm}

By Propositions \ref{prop:simplerCohFA} and \ref{prop:twistedshriek}, we are reduced to building an $\infty$-functor
\[
\widebar \FA_*^\Dc(U_\Tw) \lra \FA^!_\Dc(U_\Tw)
\]
that coincides with $\psi$ once restricted to coherent factorization algebras.

\vspace{2mm}

Recall the $2$-category $\Var_\k^{\on{corr}}$ of correspondences between 
$\k$-varieties (see \S \ref{subsec:corr}).
It naturally contains (as a non-full subcategory) a copy of (the 
$1$-category) $\Var_\k$ and a copy of $\Var_\k^\op$, through the functors:
\begin{align*}
    \Var_\k &\to \Var_\k^{\on{corr}} &  \Var_\k^\op &\to \Var_\k^{\on{corr}} \\
    X &\mapsto X & X &\mapsto X \\
    (X \to Y) & \mapsto (X \underset{=}{\leftarrow} X \to Y) & (X \leftarrow Y) 
    & \mapsto (X \leftarrow Y \underset{=}{\to} Y).
\end{align*}

\begin{Defi}
    We denote by $\UTwc$ and $\wb \UTwc$ the composite functors
    \begin{align*}
        \UTwc &\colon \TwS \overset{U_\Tw}\lra \Var_\k \lra \Var_\k^\corr \\[1mm]
        \wb\UTwc &\colon \TwS^\op \underset{U_\Tw}\lra \Var_\k^\op \lra \Var_\k^\corr.
    \end{align*}
    They both admit a lax monoidal structure given by the correspondences
    \[
    U(\alpha_1) \times U(\alpha_2) \lla U(\alpha_1 \amalg \alpha_2) \overset{=}\lra U(\alpha_1 \amalg \alpha_2).
    \]
\end{Defi}

For any two surjections $\alpha \colon I \to J$ and $\beta \colon S \to T$, we 
denote by $\Delta(\alpha,\beta)$ the subvariety of $X^I \times X^S$ 
obtained by intersecting $U(\alpha) \times U(\beta)$ with $\Delta(I,S)$ (recall 
that $\Delta(I,S) \subset X^I \times X^S$ is the closed subvariety spanned by 
those families $((x_i),(x_s))$ such that $\{x_i,i\in I\} = \{x_s,s\in S\}$).

\begin{prop}\label{prop:transformationT}
    The data of the $\Delta(\alpha,\beta)$'s define a lax monoidal natural transformation
    \[
    \xymatrix{
        \TwS \times \TwS^\op \phantom{\TwS} \ar@/^20pt/[rr]^{\UTwc \circ \rho_1} \ar@/_20pt/[rr]_{\wb\UTwc \circ \rho_2} & \ar@{=>}[]+(-9,6);[]+(-9,-6)^T & \Var_\k^\corr
    }
    \]
    where $\rho_1$ and $\rho_2$ are the projections.
\end{prop}

\begin{proof}
    To any pair $(\alpha,\beta)$ we associate the correspondence
    \[
    \UTwc(\alpha) = U(\alpha) \leftarrow \Delta(\alpha,\beta) \to U(\beta) = \wb\UTwc(\beta).
    \]
    We start by showing it defines a natural transformation. Let $\tau \colon \alpha \to \alpha'$ and $\xi \colon \beta \to \beta'$ be morphisms in $\TwS$. Consider the following commutative diagram:
    \[
    \xymatrix{
        U(\alpha) & \Delta(\alpha,\beta') \ar[r] \ar[l] & U(\beta') \\
        U(\alpha) \ar[u]^{=} \ar[d] &
        \Delta(\alpha,\beta) \ar[r] \ar[l] \ar[u] \ar[d] \ar@{}[dl]|{\displaystyle \mathrm{(1)}} \ar@{}[ur]|{\displaystyle \mathrm{(2)}} &
        U(\beta) \ar[u] \ar[d]^{=} \\
        U(\alpha') & \Delta(\alpha',\beta) \ar[r] \ar[l] & U(\beta).
    }
    \]
    It follows from \autoref{prop:Delta-cart} that the squares (1) and (2) are pullbacks. We have thus indeed defined a natural transformation.
    
    To show it is lax monoidal, we have to provide, for any $\alpha_1, \alpha_2, \beta_1,\beta_2$ in $\TwS$, a transformation in the $2$-category $\Var_\k^{\on{corr}}$
    \[
    \xymatrix{
        U(\alpha_1) \times U(\alpha_2) \ar[r]^{T} \ar[d] &  U(\beta_1) \times 
        U(\beta_2)
        \ar[d] 
        \\ U(\alpha_1 \amalg \alpha_2) \ar[r]_{T} \ar@{<=}[ur] 
        & U(\beta_1 \amalg \beta_2).
    }
    \]
    Unfolding the definition of this $2$-category, we have to find a commutative 
    diagram
    \begin{equation}\label{eq:laxmonoidal-UTw}
    \vcenter{\hbox{\xymatrix{
        U(\alpha_1) \times U(\alpha_2) & \Delta(\alpha_1, \beta_1) \times 
        \Delta(\alpha_2,\beta_2) \ar[r] \ar[l] & U(\beta_1) \times 
        U(\beta_2)
        \\
        U(\alpha_1 \amalg \alpha_2) \ar[d]^-{=} \ar[u] & Z \ar[r] \ar[u] \ar[l]^-f 
        \ar[d]^g
        \ar@{}+(0,0);[ur]+(0,0)|-{\displaystyle (1)} & U(\beta_1 \amalg \beta_2) 
        \ar[d]_-{=} 
        \ar[u]
        \\
        U(\alpha_1 \amalg \alpha_2) & \Delta(\alpha_1 \amalg \alpha_2, \beta_1 
        \amalg \beta_2) \ar[r] \ar[l] & U(\beta_1 \amalg \beta_2)
    }}}
    \end{equation}
    in which the square $(1)$ is a pullback and $g$ is proper. We therefore pick 
    $Z$ such that $(1)$ is a pullback. The image of $Z$ into $U(\alpha_1) \times U(\alpha_2)$ lies in $U(\alpha_1 \amalg \alpha_2)$ and the map $f$ is thus canonically defined.
    Finally, the map $g$ is the natural closed immersion: $Z$ sits in $\Delta(\alpha_1 \amalg 
    \alpha_2, \beta_1 \amalg \beta_2)$, as a closed subvariety of $U(\alpha_1 \amalg \alpha_2) \times U(\beta_1 \amalg \beta_2)$.
    In particular, $g$ is proper and the above diagram indeed defines a 
    transformation in $\Var_\k^{\on{corr}}$.
    
    We then check those transformations behave coherently so that they define a lax 
    symmetric monoidal structure on the transformation.
\end{proof}

\begin{proof}[Proof of \autoref{thm:coVerdier-fact}]

    We first observe the equalities 
    \[
        \Dc_*^{U_\Tw} = \Dcorr_{\UTwc} \hspace{1em} \text{and} \hspace{1em} \Dc^!_{U_\Tw} = \Dcorr_{\wb\UTwc}.
    \]
    By composing the transformation $T$ of \autoref{prop:transformationT} with the lax monoidal $\infty$-functor $\Dcorr$, we get a lax monoidal natural transformation
    \[
    \begin{tikzcd}[column sep=large]
     \Tw(\finsurj) \times \Tw(\finsurj)^\op \ar[start anchor=north east, end anchor=north west, bend left, ""{name=U, below}]{r}{\Dc_*^{U_\Tw} \circ \rho_1} \ar[start anchor=south east, end anchor=south west, bend right, ""{name=D, above}]{r}[swap]{\Dc^!_{U_\Tw} \circ \rho_2} & \Cat_\oo \ar[Rightarrow, from=U, to=D]
    \end{tikzcd}
    \]
    By \autoref{prop:grothmonod}, it induces a lax monoidal $\infty$-functor $\Groth(\Dc_*^{U_\Tw} \circ \rho_1) \to \Groth(\Dc^!_{U_\Tw} \circ \rho_2)$. Using the equivalences $\Groth(\Dc_*^{U_\Tw} \circ \rho_1) \simeq \Groth(\Dc_*^{U_\Tw}) \times \Tw(\finsurj)$ and $\Groth(\Dc^!_{U_\Tw} \circ \rho_2) \simeq \Tw(\finsurj)^\op \times \Groth(\Dc^!_{U_\Tw})$, we find a commutative diagram of lax monoidal $\infty$-functors
    \[
        \begin{tikzcd}[column sep=-3em]
        \Groth(\Dc_*^{U_\Tw}) \times \Tw(\finsurj) \ar{rr}{T'} \ar{dr}[swap, near start]{\tau_1 = (\pi_1, \Id)} && \Tw(\finsurj)^\op \times \Groth(\Dc^!_{U_\Tw}) \ar{dl}[near start]{\tau_2 = (\Id,\pi_2)}
        \\
        & \Tw(\finsurj)^\op \times \Tw(\finsurj)
        \end{tikzcd}
    \]
    (where $q_1$ and $q_2$ are actually monoidal).
    Denoting by $\Sect^\otimes_\mathrm{lax}(\mathrm{pr})$ the $\infty$-category of lax monoidal sections of $\mathrm{pr}$, for $\mathrm{pr} = \pi_1, \tau_1, \tau_2$ or $\pi_2$, we get by composing with $T'$:
    \[
    \Sect^\otimes_\mathrm{lax}(\pi_1) \lra \Sect^\otimes_\mathrm{lax}(\tau_1) \lra \Sect^\otimes_\mathrm{lax}(\tau_2).
    \]
    Composing with the projection $\Tw(\finsurj)^\op \times \Groth(\Dc^!_{U_\Tw}) \to \Groth(\Dc^!_{U_\Tw})$ and taking limits over $\Tw(\finsurj)^\op$ yields a functor $\Sect^\otimes_\mathrm{lax}(\tau_2) \to \Sect^\otimes_\mathrm{lax}(\pi_2)$.
    Let us denote by $\psi_\Tw \colon \Sect^\otimes_\mathrm{lax}(\pi_1) \to \Sect^\otimes_\mathrm{lax}(\pi_2)$ the composite functor.
    
    It is by construction given by the announced formula of \autoref{def:coVerdier-psi}.
    Arguments similar to those of section \ref{subsec:coVerdieranddiag} prove that the image of $\psi_\Tw$ lies in $\DsM(U_\Tw)$.
    It remains to prove that it maps $\widebar \FA_*^\Dc(U_\Tw)$ to $\FA^!_\Dc(U_\Tw)$.
    
    Let thus $E \in \widebar \FA_*^\Dc(U_\Tw)$. We only have to check that $\psi_\Tw(E)$ is actually monoidal.    
    Let $\alpha_1 \colon I_1 \to J_1$ and $\alpha_2 \colon I_2 \to J_2$ be 
    surjections. We set $\alpha := \alpha_1 \amalg \alpha_2$. We denote by $j \colon 
    U(\alpha) \to U(\alpha_1) \times U(\alpha_2)$ the open immersion.
    For any surjection $\beta \colon S \to T$, the variety $\Delta(\alpha, \beta)$ 
    is the disjoint union, over all decompositions $\beta = \beta_1 \amalg \beta_2$, 
    of the product $\Delta(\alpha_1 , \beta_1) \times \Delta(\alpha_2, \beta_2)$ 
    (pulled back along $j$).
    We get
    \begin{multline*}
        \psi_\Tw(E)^{(\alpha)} = \lim_{\beta} p_* q^! E^{(\beta)} \simeq 
        j^* \lim_{\beta_1, \beta_2} (p_{1*} q^!_1 E^{(\beta_1)} \boxtimes p_{2*} q_2^! 
        E^{(\beta_2)})\\ \simeq j^*(\psi_\Tw(E)^{(\alpha_1)} \boxtimes \psi_\Tw(E)^{(\alpha_2)})
    \end{multline*}
    where
    \begin{align*}
        U(\alpha) \overset{p}{\lla} \Delta(\alpha&, \beta) \overset{q}{\lra} U(\beta) \\
        U(\alpha_1) \overset{p_1}{\lla} \Delta(\alpha_1&, \beta_1) \overset{q_1}{\lra} U(\beta_1) \\
        U(\alpha_2) \overset{p_2}{\lla} \Delta(\alpha_2&, \beta_2) \overset{q_2}{\lra} U(\beta_2)
    \end{align*}
    are the projections. It follows that the lax monoidal structure of $\psi_\Tw(E)$ is actually monoidal.
\end{proof}


\section{Gelfand-Fuchs cohomology in algebraic geometry}\label{sec:GFAG}

Recall that  $X$ is a  fixed smooth algebraic variety over $\k$.

\subsection{Chevalley-Eilenberg factorization algebras}\label{subsec:CEFA}

\paragraph{Homological $\Dc^!$-module.}

Let $L$ be a local Lie algebra on $X$, i.e., a vector bundle with a Lie algebra structure on the sheaf of sections  given by a bi-differential operator.
Then the right $\Dc_X$-module $\Lc = L\otimes_{\Oc_X}\Dc_X$ is  a Lie$^*$-algebra, see \cite{BD} \S 2.5. 
Using the determinantal factor $\det(\k^2)$, 
we can write the antisymmetric Lie$^*$-bracket in $\Lc$ as  a permutation equivariant morphism of $\Dc$-modules on $X\times X$
\be\label{eq:*-bracket} 
\eta: (\Lc\boxtimes \Lc)  \otimes_\k  \det(\k^2) \lra \delta_* \Lc.
\ee
Here $\delta: X\to X\times X$ is the diagonal embedding. Let us list the most important examples.

\begin{exas}~
    \begin{statements}
        \item $L=T_X$ is the tangent bundle of $X$.  
        \item Let $G$ be an algebraic group over $\k$ with Lie algebra $\gen$ and $P$ be a principal $G$-bundle on $X$. 
        The data of $P$ gives rise to two local Lie algebras on $X$. First, we have
        the $\Oc_X$-linear Lie algebra
        $P^{\Ad}$ (infinitesimal symmetries of $P$).
        Second, we have the {\em Atiyah Lie algebroid} $\At(P)$ (infinitesimal symmetries of the pair $(X,P)$) fitting into a short
        exact sequence
        \[
        0\to P^\Ad \lra \At(P) \lra T_X \to 0. 
        \]
        For $P=X\times G$ the trivial bundle, $P^\Ad= \gen\otimes\Oc_X$ and $\At(P)$ is the semi-direct product of $T_X$ acting on $\gen\otimes \Oc_X$
        via the second factor.
    \end{statements}
\end{exas}

Given a local Lie algebra $L$, we have the dg-Lie algebra $\len = R\Gamma(X,L)$ and we are interested in its Lie algebra cohomology with constant coefficients. 
It is calculated by the (reduced) Chevalley-Eilenberg chain complex of $\len$ which we denote by
\[
\CE_\bullet(\len) \,\,=\,\,\bigl( \Sym^{\bullet \geq 1} (\len[1]), d_\CE\bigr). 
\]
Applying the K\"unneth formula, we see that 
\[
\CE_\bullet(\len) \,\,=\,\,\Tot\biggl\{ \cdots \to R\Gamma(X^3, L^{\boxtimes 3})_{-S_3} \to R\Gamma(X^2, L^{\boxtimes 2})_{-S_2} \to R\Gamma(X,L)\biggr\} 
\]
is the total complex of the   obvious double complex with horizontal grading ending in degree $(-1)$. Here $S_p$ is the symmetric group and the
subscript ``$-S_p$'' means the space of (graded) coinvariants under the skew-symmetric action of $S_p$ on $R\Gamma(X^p, L^{\boxtimes p})$.

\vskip .2cm

Following \cite{BD} we represent $\CE_\bullet(\len)$ as the factorization homology of an appropriate lax $\Dc^!$-module $\Cc_\bullet$ on $X^\finsurj$. 
We first define a lax $\Dc^!$-module $\Cc_1$ by putting $\Cc_{1}^{(I)} = (\delta_{I})_* \, \Lc$
for any nonempty finite set $I$.  Here $\delta_I: X\to X^I$ is the diagonal embedding. Given
a surjection $g: I\to J$ with the corresponding diagonal embedding $\delta_g: X^J\to X^I$, we define structure map 
(in the dual form)
$(\delta_g)_* (\delta_J)_* \, \Lc \to (\delta_I)_*\, \Lc$ to be the canonical isomorphism arising from the equality $\delta_g\circ\delta_J=\delta_I$. 

\begin{rems}~
    \begin{statements}
        \item After passing to the colimit,  $\Cc_1$  becomes the pushforward of $\Lc$ under the embedding of $X$ into $\Ran(X)$. 
        
        \item Note that the structure maps for $\Cc_1$ in the form
        $\Cc_{1, J} \to \delta_g^! \, \Cc_{1}^{(I)}$ are not, in general, isomorphisms, so $\Cc_1$ is not a strict $\Dc^!$-module.
    \end{statements}
\end{rems}

Recall that the $\infty$-category  $\LDsM(X^\finsurj)$ has a symmetric monoidal structure $\otimes^*$
(see \autoref{Defi:*-product} above and \cite[4.2.5]{BD}). The fact that $\Lc$ is a Lie$^*$-algebra means that $\Cc_1$ is a Lie algebra with respect to $\otimes^*$.

\begin{Defi}
    We define
    \[
    \Cc_\bullet \,\ = \,\, \bigl(\Sym^{\bullet \geq 1}_{\otimes^*} (\Cc_1[1]), d_\CE\bigr)
    \]
    to be the Chevalley-Eilenberg complex of $\Cc_1$ as a Lie algebra in $(\LDsM(X^\finsurj), \otimes^*)$.   
\end{Defi}

For convenience of the reader let us describe $\Cc_\bullet$ more explicitly.
From the formula for $\otimes^*$ given in \autoref{Defi:*-product} above
we obtain directly
\begin{equation}\label{eq:sum-EQ}
    \Cc_{q}^{(I)} = 
    \bigoplus_{R\in\Eq_q(I)} (\delta_R)_* \bigl( \Lc^{\boxtimes (I/R)} \otimes \det(\k^{I/R})\bigr).
\end{equation}
Here $\Eq_q(I)$ is the set of equivalence relations $R$ on $I$ with exactly $q$ equivalence classes, i.e., such that $|I/R|=q$, and
$\delta_R: X^{I/R}\to X^I$ is the diagonal embedding. In particular, $\Cc_q$ is concentrated in degree $0$. 

Given a surjection $g: I\to J$ and any surjection $f: J\to Q$ with $|Q|=q$, we have the surjection $fg: I\to Q$ with 
$\delta_g\circ\delta_f =\delta_{fg}$, and so we have an identification
\[
(\delta_g)_* (\delta_f)_* (\Lc^{\boxtimes Q} \otimes\det(\k^Q)) \lra (\delta_{fg})_*\, ( \Lc^{\boxtimes Q} \otimes\det(\k^Q))
\]
of (an arbitrary) term 
of the colimit for $(\delta_g)_*\, \Cc_{q}^{(J)}$ with a certain term  
of the colimit for $\Cc_{q}^{(I)}$. The structure map (in the dual form) $(\delta_g)_*\,  \Cc_{q}^{(J)} \to \Cc_{q}^{(I)}$ is induced by these identifications. 

\vskip .2cm
Next, we have the differential $d=d_\CE: \Cc_{q}\to\Cc_{q-1}$ defined as follows. 
Let $g \colon Q \twoheadrightarrow S$ be a surjection between finite sets such that  that $|Q|  = |S|+1 = q$,
so $g$ has exactly one fiber of cardinality $2$, all other fibers being of cardinality $1$. Applying the bracket 
\eqref {eq:*-bracket} to this fiber and substituting copies of the identity for the other fibers, we get a map
\[
\eta^g \colon  \Lc^{\boxtimes Q} \otimes \det(\k^Q) \lra (\delta_g)_*\left( \Lc^{\boxtimes S} \otimes \det(\k^S)\right).
\]
We define  $d \colon \Cc_q \to \Cc_{q-1}$  by summing,  as $g$ varies,  the induced maps
\[
(\delta_f)_* \left(( \Lc^{\boxtimes Q} \otimes \det(\k^Q)\right) \to  (\delta_{gf})_*\left( \Lc^{\boxtimes S} \otimes \det(\k^S)\right).
\]
The differential squares to zero by the Jacobi identity.

\begin{Defi}
    Let $\Cc^{(I)}$ be the complex of $\Dc$-modules on $X^I$ given by $\Cc_q^{(I)}$ in homological degree $q$ and the above differential.
    We denote by $\Cc_{\leq q}^{(I)}$ its truncation $\bigoplus_{p\leq q} \Cc_{p}^{(I)}[p]$ (with the same differential).
    The differential is compatible with the transition maps, and we get lax $\Dc^!$-modules   $\Cc_{\leq q}$ on $X^{\finsurj}$.
\end{Defi}

\begin{prop}
    We have
    \[
    \int_X \Cc_q \simeq \Lambda^q(\len), \quad   \int_X \Cc_{\leq q}  \simeq \CE_{\leq q}(\len), \quad\int_X \Cc \simeq \CE_\bullet(\len).
    \]
\end{prop}

\begin{proof} 
    We show the first identification, the compatibility with the differentials will be clear. 
    For any $I \in \finsurj$, we have, by \eqref{eq:sum-EQ} and \eqref{eq:DR-Induced}:
    \begin{multline*} 
        R\Gamma_{\DR}(X^I, \Cc_{q}^{(I)})\,= \bigoplus_{R\in\Eq_q(I)} R\Gamma_\DR\bigl(X^{I/R}, \Lc^{\boxtimes (I/R)} \otimes\det(\k^{I/R})\bigr)
        \\
        =    \bigoplus_{R\in\Eq_q(I)} R\Gamma\bigl(X^{I/R}, L^{\boxtimes (I/R)} \otimes\det(\k^{I/R})\bigr) \,\,=  \bigoplus_{R\in\Eq_q(I)} \Lambda^{|I/R|} (\len).
    \end{multline*} 
    Now, $\int_X \Cc_q$ is the $\hocolim$ of this over $I$ in $\finsurj$, and so is identified with $\Lambda^q(\len)$.
\end{proof}

\begin{prop}
    The lax $\Dc^!$-module $\Cc_\bullet$ is factorizing (i.e. admits a factorization structure). 
\end{prop}

\begin{proof}
    By construction, $\Cc_\bullet$ is the homological Chevalley-Eilenberg complex of $\Cc_1$ for the $*$-tensor product. It is therefore a coalgebra in the non-unital symmetric monoidal $\infty$-category $(\LDsM(X^\finsurj), \otimes^*)$. By \autoref{lem:*-to-chiral}, it is also a coalgebra in $(\LDsM(X^\finsurj), \otimes^\mathrm{ch})$, and thus a pre-$\Dc^!$-factorization algebra.
    We see using equation \eqref{eq:sum-EQ} that $\Cc_\bullet$ is actually factorizing.
\end{proof}

\paragraph{Strictification of $\Cc_\bullet$ and chiral envelopes.}\label{par:chiral-envelops}
While the components $C_q^{(I)}$ of the lax $\Dc^!$-module $\Cc_\bullet$ are very simple $\Dc_{X^I}$-modules constructed
out of $\Lc^{\boxtimes q}$, the strictification $\strict{\Cc_\bullet}$ is highly non-trivial. More precisely, let $\pt$
denote the $1$-element set, so $X^\pt=X$.
For $x\in X$ let $\wh x = \Spec \wh\Oc_{X,x}$ be the formal disk around $x$ and $\wh x^\circ = \wh x - \{x\}$ the punctured
formal disk. 

Consider the left $\Dc_X$-module $\omega_X^{-1}\otimes_{\Oc_X} \strict{\Cc_\bullet}^{(\pt)}$ and its $\Oc_X$-module fiber
\[
\bigl(\omega_X^{-1}\otimes_{\Oc_X} \strict{\Cc_\bullet}^{(\pt)}\bigr)_x \,\,=\,\,  \bigl(\omega_X^{-1}\otimes_{\Oc_X} \strict{\Cc_\bullet}^{(\pt)}\bigr)
\otimes_{\Oc_X} \k_x. 
\]
\begin{prop}
    One has a canonical identification
    \[
    \bigl(\omega_X^{-1}\otimes_{\Oc_X} \strict{\Cc_\bullet}^{(\pt)}\bigr)_x \,\,\simeq \,\,\Ind_{\Gamma(\wh x, L)} ^{R\Gamma(\wh x^\circ, L)} \k
    \]
    where on the right hand side we have the vacuum (dg-)module of the dg-Lie algebra $R\Gamma(\wh x^\circ, L)$.
    In other words, $\omega_X^{-1}\otimes_{\Oc_X} \strict{\Cc_\bullet}^{(\pt)}$ is the chiral envelope of the Lie$^*$-algebra $\Lc$,
    see \cite[\S 1.2.4]{G} and \cite[\S 3.7.1]{BD}. 
\end{prop}

Since we will not need this result in the present paper, we leave its proof to the reader. 
In this way one can see the validity of Theorem 4.8.1.1 of \cite{BD} (the chiral homology of the chiral envelope is the same
as  $H_\bullet^\Lie(\len)$) for any smooth variety $X$, not necessarily 1-dimensional or proper.

\paragraph{Cohomological $[\![\Dc]\!]$-module and diagonal filtration.}

Using Verdier duality, we get $[\![\Dc]\!]$-modules 
\[
\wc \Cc ^q = (\Cc_q)^\vee, \quad \wc\Cc^{\leq q} = (\Cc_{\leq q})^\vee, \quad \wc \Cc^\bullet = (\Cc_\bullet)^\vee
\]
on $X^\finsurj$ such that 

\[
\oint_X^{[\![\c]\!]} \, \wc \Cc^\bullet \,\, \simeq \,\, \CE^\bullet( \len).
\]
This complex comes with the {\em diagonal filtration} (cf. section \ref{subsec:coVerdieranddiag} or \cite[Ch.2 \S4]{fuks})  which is a sequence of complexes and morphisms given by
\[
R\Gamma^{[\![\c]\!]}_{X_1^\finsurj\!,\,\DR}(X^\finsurj, \wc \Cc^\bullet) \to \cdots \to 
R\Gamma^{[\![\c]\!]}_{X_d^\finsurj\!,\,\DR}(X^\finsurj, \wc \Cc^\bullet) \to \cdots \quad  \to \oint^{[\![\c]\!]}_X \wc \Cc^\bullet, 
\]
where 
\[
R\Gamma^{[\![\c]\!]}_{X_d^\finsurj\!,\,\DR}(X^\finsurj, \wc \Cc^\bullet)\,\, =\,\,  \oint_X^{[\![\c]\!]}\,  (i_d)_*  i_d^!\, \wc\Cc^\bullet ~:=~ \holim_I R\Gamma^{[\![\c]\!]}_\DR(X^I_d, i^!_d\, \wc\Cc^{(I)})
\]
with $i_d\colon X_d^\finsurj \to X^\finsurj$ the pointwise closed immersion.

\vskip .2cm

For future use we introduce the following.

\begin{Defi}\label{def:diagonal}
    We call the {\em diagonal $\Dc$-module} associated to $L$ the complex of $\Dc$-modules $\wc \Cc^\bullet_\Delta = i_1^!\wc \Cc^\bullet$ on $X$:
    \[
    \wc \Cc^\bullet_\Delta ~=~ i_1^!\wc \Cc^\bullet ~:=~ \holim_J (i_1^{(J)})^!\,\wc \Cc^\bullet_J \simeq \psi(\wc \Cc^\bullet)^{(\pt)},
    \]
    where $i_1^{(J)} \colon X \to X^J$ is the diagonal embedding.
    
    The de Rham complex of $\wc \Cc^\bullet_\Delta$ will be denoted by $\Fc^\bullet_\Delta$ and called 
    {\em the diagonal complex} of $L$. 
    We will denote its compactly supported cohomology
    \[
    H^\bullet_\Delta (L) \,=\, H^\bullet_{[\![\c]\!]} (X, \Fc^\bullet_\Delta) \,=\,  R\Gamma^{[\![\c]\!]}_{X_1^\finsurj}(\wc \Cc^\bullet)
    \]
    and call it the {\em diagonal cohomology} of $\len$. It comes with a canonical map $H^\bullet_\Delta (L)\to H^\bullet_\Lie(\len)$. 
\end{Defi} 

Explicitly,
\be\label{eq:complex-F}
\Fc^\bullet_\Delta  \,=\, \biggl\{ L^\vee \to \ul H^n_X((L^\vee)^{\boxtimes 2})^{\Sigma_2} \to 
\ul  H^{2n}_X((L^\vee)^{\boxtimes 3})^{\Sigma_3} \to \cdots\biggr\},
\ee
with grading normalized so that $L^\vee$ is in degree $1-n$. 

\begin{rem}
    One can see
    $\Fc^\bullet_\Delta$   as  a sheafified, algebro-geometric
    version of the diagonal complex of
    Gelfand-Fuchs \cite{gelfand-fuks} \cite{fuks}, see also \cite{CG2} \S 4.2. 
    Instead of distributions supported on the diagonal $X\subset X^p$, as in the $C^\infty$-case,
    our construction involves coherent cohomology of $X^p$ with support in $X$ which is a well known
    analog  of the space of distributions  (``holomorphic hyperfunctions'').  
    
\end{rem}


\subsection{The diagonal filtration in the affine case}\label{subsec:diag-aff}

In this section, we assume that $X$ is a smooth affine variety. 
As before, $L$ is a local Lie algebra on $X$ and $\len = \Gamma(X, L)$.
Recall that $\Cc = \Cc_L$ is the factorizing $!$-sheaf computing $\CE_\bullet(\len)$. 
We will prove the following theorem.

\begin{thm}\label{thm:affinecase}
    The canonical map $\int_X \Cc \to \oint_X \phi(\Cc) \simeq \holim_d \int^{\leq d}_X \Cc$ is an equivalence.
\end{thm}

Recall (see equation \eqref{eq:sum-EQ}) that $\Cc$ comes with a natural filtration $\Cc_{\leq q}$ where $\Cc_q = \hocofib(\Cc_{\leq q-1} \to \Cc_{\leq q})[-q]$ is given by
\[
\Cc_{q}^{(I)} = 
\bigoplus_{R\in\Eq_q(I)} (\delta_R)_* \bigl( \Lc^{\boxtimes (I/R)} \otimes \det(\k^{I/R})\bigr).
\]
\begin{lem}
    The canonical map $\int_X \Cc_q \to \oint_X\phi(\Cc_q)$ is an equivalence.
\end{lem}
\begin{proof}
    We compute explicitly both sides and find $\Lambda^q \Gamma(X,L) = \Lambda^q \mathfrak{l}$.
\end{proof}

From this, we deduce by induction on $q$:
\begin{lem}
    The canonical map $\int_X \Cc_{\leq q} \to \oint_X \phi(\Cc_{\leq q})$ is an equivalence.
\end{lem}

\begin{proof}[Proof of \autoref{thm:affinecase}]
    Since $\hocolim_q \int_X \Cc_{\leq q} \simeq \int_X \Cc$, it is now enough to prove that the map
    \[
    \hocolim_q \oint_X \phi(\Cc_{\leq q}) \to \oint_X \phi(\Cc)
    \]
    is an equivalence. Rephrasing with the diagonal filtration, we get
    \[
    \hocolim_q \holim_d \int^{\leq d}_X \Cc_{\leq q} \to \holim_d \hocolim_q \int^{\leq d}_X \Cc_{\leq q}.
    \]
    Fix an integer $p$. It is enough to prove that for $q$ big enough (independently of $d$), the map $H^p_{\DR}(\int^{\leq d}_X \Cc_{\leq q}) \to H^p_{\DR}(\int^{\leq d}_X \Cc_{\leq q+1})$ is an isomorphism. This amounts to proving that $H^{p+q}_{\DR}(\int_X^{\leq d} \Cc_q)$ vanishes for $q$ big enough.
    
    By definition, we have $\int_X^{\leq d} \Cc_q = \hocolim_I R\Gamma_{\DR}(X^I_d, i_d^{[\![*]\!]} \Cc_q^{(I)})$. Let us fix $I$. For any positive integer $s$, we denote by $Y^{(s)}$ the $s^\mathrm{th}$ infinitesimal neighborhood of $X^I_d$ in $X^I$ and by $i^{(s)} \colon Y^{(s)} \to X^I$ the canonical inclusion. We get
    \[
    R\Gamma_{\DR}(X^I_d, i^{[\![*]\!]}_d \Cc_q^{(I)}) \simeq \holim_s R\Gamma_{\DR}(Y^{(s)}, (i^{(s)})^* \Cc_q^{(I)}).
    \]
    Since $Y^{(s)}$ is affine (because $X$ is) and $\Cc_q^{(I)}$ is induced from a quasicoherent sheaf concentrated in degree $0$, the complex $R\Gamma_{\DR}(Y^{(s)}, (i^{(s)})^*\Cc_q^{(I)})$ only has cohomology in degrees lower or equal to $0$. The homotopy limit indexed by $s$ satisfies the Mittag-Leffler condition. We deduce that the cohomology of $R\Gamma_{\DR}(X^I_d, i_d^{[\![*]\!]} \Cc_q^{(I)})$ is concentrated in degree lower or equal to $0$. It follows that $H^{p+q}_{\DR}(\int_X^{\leq d} \Cc_q)$ vanishes for $q \geq 1 - p$.
    This concludes the proof of \autoref{thm:affinecase}.
\end{proof}
\vskip 1em
Applying Verdier duality to \autoref{thm:affinecase}, we get:

\begin{cor}\label{cor:affine-L}
    For a smooth affine variety $X$, the diagonal filtration on Chevalley-Eilenberg cohomology is complete, i.e.:
    \[
    \CE^\bullet(\len) \,  \simeq\,  \oint^{[\![\c]\!]}_X \wc \Cc^\bullet  \,\simeq \, \int^{[\![\c]\!]}_X \psi(\wc \Cc^\bullet) 
    \,\simeq\,  \hocolim_d R \Gamma^{[\![\c]\!]}_{X^\finsurj_d\!,\,\DR}(X^\finsurj, \wc\Cc^\bullet).
    \]
\end{cor}


\section{Relation to the topological picture}\label{sec:top-pic}

\subsection{\texorpdfstring{$C^\infty$}{C∞} factorization
    algebras in \texorpdfstring{$\infty$}{∞}-categories}\label{subsec:C-oo-infty}
    
We extend here the notion of $C^\infty$ factorization algebras introduced in Chapter \ref{sec:C-inf-GF} to 
$\infty$-categorical contexts. The content of this section is by no means original and can be found for instance in \cite{lurie-ha} or \cite{ginot}.

\begin{Defi}[{see \cite[5.5.2.1]{lurie-ha}}]
Let $M$ be a $C^\infty$-manifold. 
 Denote by $\Uc(M)^\otimes$ the (non-unital) colored operad
  whose colors are open subsets of $M$ and such that for $U, U_1, \dots, U_n$ open subset of $M$, we have
 \[
  \on{Mult}(U_1, \dots, U_n ; U) = \begin{cases} * & \text{if } U_i \subset U,\,  \forall i \text{ and } U_i \cap U_j = \varnothing,
  \, \forall i \neq j; 
   \\ \varnothing & \text{otherwise.} \end{cases}
 \]
 In particular, the full subcategory $\Uc(M) \subset \Uc(M)^\otimes$ of colors is the category of open subsets of $M$.
\end{Defi}

\begin{Defi}\label{Defi:FAinftyCat}
 Let $\Cc$ be  a (non-unital) symmetric monoidal $\infty$-category.
 A $\Cc$-valued $C^\oo$ {\em factorization algebra} over a $C^\oo$-manifold 
 $M$ is a $\Uc(M)^\otimes$-algebra $\Fc$ in $\Cc$ 
 (i.e., a morphism of $\oo$-operads $\Fc: \Uc(M)^\otimes\to \Cc^\otimes$)
 whose restriction $\Uc(M) \to \Cc$ sends factorizing covers (see \autoref{Defi:factorizing-cover}) to the realizations of their multiplicative \v Cech complexes (as in \eqref{eq:codescent}).
 The factorization homology of $\Fc$ is by definition
 \[
  \int_M \Fc := \Fc(M) \,\in \, \Ob \Cc. 
 \]
\end{Defi}

\begin{rem}
 The above definition is a direct extension of its counterpart in Chapter \ref{sec:C-inf-GF}.
\end{rem}

\subsection{From factorizing \texorpdfstring{$\Dc^!$}{D!}-modules to \texorpdfstring{$C^\infty$}{C∞} factorization
    algebras}\label{subsec:D!-e2n}

In this section we assume $\k=\CC$. Thus $X_\an := X(\CC)$ is a complex analytic manifold. 

For any complex  analytic manifold $M$ 
we denote by $M^\finsurj$ the Ran diagram of complex manifolds $M^I$
and diagonal embeddings $\delta_g: M^J\to M^I$. 

\paragraph{Holonomic regular factorization algebras}

\begin{Defi}
 A lax $\Dc^!$-module $\Ec$ over $X^\finsurj$ is called holonomic regular (h.r.) if each $\Dc$-module $\Ec^{(I)}$ has holonomic regular cohomology over $X^I$.
\end{Defi}

\begin{prop}\label{prop:e_1hr-ehr}
    Suppose $\Ec$ is factorizable. Then, $\Ec$ is h.r. if and  only if the $\Dc$-module over $X$ has holonomic regular cohomology.
\end{prop}

\begin{proof}
    Suppose  $\Ec^{(\pt)}$ is h.r. on $X$. We prove, by induction on $|I|$, that $\Ec^{(I)}$ is h.r. on $X^I$. 
    First of all, the restriction of $\Ec^{(I)}$ to the open subset $X^I_\neq$ (complement to all the diagonals) is, by the factorization
    structure, identified with the restriction to $X^I_\neq$ of the h.r. module $(\Ec^{(\pt)})^{\boxtimes I}$. Further, the complement
    $X^I - X^I_\neq$ is stratified into locally closed subvarieties isomorphic to $X^J_{\neq}$ with $2\leq |J| < |I|$, or to $X^{\pt} = X $,
    if $|J|=1$. So our statement follows by stability of h.r. modules under extensions.
\end{proof}

\autoref{prop:twistedshriek} allows for another description of h.r. $!$-factorization algebras. Denote by $\Dc^{!,\hr}_{U_\Tw}$ the $\infty$-functor $\Tw(\finsurj)^\op \to \Cat_\oo$ that maps $\alpha \in \Tw(\finsurj)$ to the $\infty$-category of complexes of $\Dc$-modules with holonomic regular cohomology over $U(\alpha)$, and morphisms of twisted arrows to the corresponding $!$-pull-back functor.
It is a subfunctor of $\Dc^!_{U_\Tw}$. Moreover, since a tensor product of holonomic regular $\Dc$-modules is also holonomic regular, the lax monoidal structure of $\Dc^!_{U_\Tw}$ restricts to such a structure on $\Dc^{!,\hr}_{U_\Tw}$.
From \autoref{prop:twistedshriek}, a h.r. $!$-factorization algebra amounts to a symmetric monoidal and Cartesian section of the projection
\[
 \Groth(\Dc^{!,\hr}_{U_\Tw}) \to \Tw(\finsurj)
\]

Further, the Riemann-Hilbert correspondence yields an equivalence of lax-monoidal $\infty$-functors between $\Dc^{!,\hr}_{U_\Tw}$ and the $\infty$-functor $\mathrm{Sh}_{U_\Tw^\an}^{!,\mathrm{const}} \colon (\Tw \finsurj)^\op \to \Cat_\oo$ that maps $\alpha$ to the $\infty$-category of complexes of sheaves on the analytification $U(\alpha)_\an$ that are constructible for the diagonal filtration.
This implies
\begin{prop}\label{prop:RHmonoidalsections}
 The $\infty$-category of h.r. $!$-factorization algebras is equivalent to that of symmetric monoidal and Cartesian sections of
 \[
  \Groth(\mathrm{Sh}_{U_\Tw^\an}^{!,\mathrm{const}}) \to \Tw(\finsurj).
 \]
 We denote this equivalence by $\mathrm{RH} \colon \FA^!(X) \simeq \Sect_\mathrm{Cart}^\otimes(\Groth(\mathrm{Sh}_{U_\Tw^\an}^{!,\mathrm{const}}))$.
 \qed
\end{prop}

\paragraph{Associated $C^\oo$ factorization algebra}

Consider now the following functor
\[
 R\Gamma_\mathrm{c} \colon \Uc(X_\an) \times \Groth(\mathrm{Sh}_{U_\Tw^\an}^{!,\mathrm{const}}) \lra C(\k)
\]
that associates to $V \subset X_\an$ and $(\alpha \colon I \twoheadrightarrow J,\Fc)$ (with $\Fc$ a (constructible) complex of sheaves over $U(\alpha)_\an$) the complex $R\Gamma_\mathrm{c}(U(\alpha) \cap V^I, \Fc_\alpha)$.
To simplify the notations, we set
\[
 U^V_\alpha := U(\alpha) \cap V^I \subset X_\an^I.
\]
Fix a family $V_1, \dots, V_n$ of pairwise disjoint open subsets of $X_\an$ and $\alpha_1, \dots, \alpha_n \in \Tw(\finsurj)$. Since the $V_p$'s are pairwise disjoint, we have
\[
 U^{V_1}_{\alpha_1} \times \cdots \times U^{V_n}_{\alpha_n} \subset U^{V_1 \amalg \cdots \amalg V_n}_{\alpha_1 \amalg \cdots \amalg \alpha_n}.
\]
Fixing now $\Fc_1, \dots, \Fc_n$ with $\Fc_p$ constructible over $U(\alpha_p)$, we get a natural extension by $0$ morphism
\begin{multline*}
 R\Gamma_c\left(U^{V_1}_{\alpha_1}, \Fc_1\right) \otimes \cdots \otimes R\Gamma_c\left(U_{\alpha_n}^{V_n}, \Fc_n\right) \simeq R\Gamma_\mathrm{c}\left(U^{V_1}_{\alpha_1} \times \cdots \times U^{V_n}_{\alpha_n}, \Fc_1 \boxtimes \cdots \boxtimes \Fc_n\right) 
 \\
 \lra R\Gamma_\mathrm{c}\left(U^{V_1 \amalg \cdots \amalg V_n}_{\alpha_1 \amalg \cdots \amalg \alpha_n}, \Fc_1 \boxtimes \cdots \boxtimes \Fc_n\right).
\end{multline*}
Altogether, those morphisms extend the functor $R\Gamma_\mathrm{c}$ to a morphism of $\oo$-operads 
\[
 \Uc(X_\an)^\otimes \times_{\finsurj^*} \Groth(\mathrm{Sh}_{U_\Tw^\an}^{!,\mathrm{const}})^\otimes
 \lra C(\k)^\otimes.
\]
 Let $\Ec$ be a h.r. $!$-factorization algebra over $X$.
 Let
 $
q \colon \Uc(X_\an)^\otimes \times_{\finsurj^*} (\Tw \finsurj)^\amalg \to  \Uc(X_\an)^\otimes
 $
be the projection.
We denote by $\Ac_\Ec$ the $q$-free $\Uc(X_\an)^\otimes$-algebra (see \cite[3.1.3.1]{lurie-ha}) generated by the composite
 \[
   \Uc(X_\an)^\otimes \times_{\finsurj^*} (\Tw \finsurj)^\amalg \overset {\mathrm{RH}(\Ec)}\lra \Uc(X_\an)^\otimes \times_{\finsurj^*} \Groth(\mathrm{Sh}_{U_\Tw^\an}^{!,\mathrm{const}})^\otimes
 \lra C(\k)^\otimes.
 \]
 The  $\Uc(X_\an)^\otimes$-algebra  $\Ac_\Ec$ is in particular a morphism of $\oo$-operads
 which we denote by the same symbol
\[
 \Ac_\Ec \colon \Uc(X_\an)^\otimes \lra C(\k)^\otimes.
\]
The definition of  $\Ac_\Ec$ as a  $q$-free $\Uc(X_\an)^\otimes$-algebra
automatically takes care of all the higher coherences. The following proposition
connects this definition with the more immediate and intuitive approach,
in which such coherences may not be obvious a priori.

\begin{prop}\label{prop:h-r-fact-anal}
Let $\Ec$ be a h.r. $!$-factorization algebra and $\Ac_\Ec$ as above.
\begin{assertions}
 \item\label{ass:AE-value-on-V} For  $V \subset X_\an$ open, we have $\Ac_\Ec(V) \simeq \hocolim_I R\Gamma_\mathrm{c}\left(V^I, \DR(\Ec^{(I)})_\an\right)$.
 \item\label{ass:AE-is-FA} $\Ac_\Ec$ is a $C^\oo$-factorization algebra in the sense of \autoref{Defi:FAinftyCat}.
 \item\label{ass:FH-of-AE} We have
        \[
        \int_X^{[\![\c]\!]} \Ec \,\,\simeq \,\, \int_{X_\an} \Ac_\Ec.
        \]
\end{assertions}
\end{prop}
\begin{proof}
 We start by proving \ref{ass:AE-value-on-V}. Let $V$ be an open subset of $X_\an$.
 Let $\Uc^\otimes_{\mathrm{act}/V}$ denote the (ordinary)
 category whose objects are
 pairs  $(I \in \finsurj^*, (V_i)_{i \in I^o})$ where the $V_i$'s are pairwise disjoint open subsets of $V$.
 By definition, morphisms $((I, V_i)) \to (J, (W_j))$  in  $\Uc^\otimes_{\mathrm{act}/V}$  are active surjections $f \colon I \twoheadrightarrow J$ (so that $f^{-1}(*) = \{*\}$) such that $\forall i \in I^o$
  we have $V_i \subset W_{f(i)}$. We denote by $\Cc_V$ the fiber product $\Uc^\otimes_{\mathrm{act}/V} \times_{\finsurj^*} (\Tw \finsurj)^\amalg$. Thus  objects $\Cc_V$ are triples
  $(I, (V_i), (\alpha_i))$, where the pair $(I, (V_i))$ is an object of $\Uc^\otimes_{\mathrm{act}/V}$
  as before and $(\alpha_i)_{i\in I^o }$ is a system of arrows in $\finsurj$ (which are objects
  of $\Tw\finsurj$) labelled by $I^o$.

Unwinding the definition of $q$-free algebra, see \cite[3.1.3.1]{lurie-ha}, we have\footnote{Technically speaking, the colimit in \eqref{eq:formulaAEV} should be an operadic colimit \cite[\S3.1.1]{lurie-ha}. However, in virtue of \cite[3.1.1.16]{lurie-ha} and because the tensor product of complexes preserves all small colimits in each variable, the standard homotopy colimit is operadic.} for any open subset $V \subset X_\an$:
 \begin{equation}\label{eq:formulaAEV}
  \Ac_\Ec(V) = \hocolim_{(I, (V_i), (\alpha_i)) \in \Cc_V}
  \bigotimes_{i \in I^o}R\Gamma_\mathrm{c}\left(U_{\alpha_i}^{V_i}, \DR(\Ec^{(\alpha_i)})_\an\right).
 \end{equation}
 We then observe that the inclusion 
 \[
 \finsurj \hookrightarrow \Cc_V~;~J \mapsto \bigl(\{*,1\}, V_1 = V, \alpha_1 \colon J \twoheadrightarrow *
 \bigr)
 \]
  is cofinal, so that $\Ac_\Ec(V) \simeq \hocolim_J R\Gamma_\mathrm{c}\left(V^J, \DR(\Ec^{(J)})_\an\right)$.
  This proves (a). 
 
 We now prove \ref{ass:AE-is-FA}. Let $V$ and $W$ be disjoint open subsets of $X_\an$. There is an equivalence
 \[
  \begin{aligned}
   \Cc_V \times \Cc_W & \to \Cc_{V \amalg W} \\
  \bigl(  (I, (V_i)_{i\in I^o} , (\alpha_i)_{i\in I^o} ),(J,(V_j)_{j\in J^o},(\alpha_j)_{j\in J^o})\bigr) & \mapsto (I \amalg_* J, (V_i)_{i \in I^o \amalg J^o}, (\alpha_i)_{i \in I^o \amalg J^o}).
  \end{aligned}
 \]
 Its inverse decides whether each $V_i$ is a subset of $V$ or of $W$. This equivalence allows us to split in two the colimit of \eqref{eq:formulaAEV}. It yields
 \begin{equation}\label{eq:AE-multiplicative}
  \Ac_\Ec(V  \amalg W)
  \simeq \Ac_\Ec(V) \otimes \Ac_\Ec(W).
 \end{equation}
 Let now $\Uen = (U_i)_{i\in I}$ be a factorizing cover of an open subset $W \subset X_\an$. Using \eqref{eq:AE-multiplicative}, the multiplicative nerve (see \eqref{eq:codescent}) of $\Uen$ becomes
 \[
\Cen_\bullet(\Uen, \Ac_\Ec) \,\simeq\, \left\{
 \xymatrix@C=2em{
 \cdots     
          \ar@<.6ex>[r] \ar@<-.6ex>[r] \ar[r] & \displaystyle
          \bigoplus_{\alpha,\beta \in P\Uen} \Ac_\Ec\biggl(  \coprod_{U\in\alpha, V\in\beta} U\cap V\biggr)
     \ar@<.4ex>[r] \ar@<-.4ex>[r] & \displaystyle
   \bigoplus_{\alpha\in P\Uen} \Ac_\Ec \biggl( \coprod_{U\in\alpha}U \biggr)
   }
   \!\right\}.
\]
As in the proof of Theorem \ref{thm:CE-fact}, 
this identifies with the  usual \v Cech nerve of the cover $\Ven := (V_\alpha)_{\alpha \in P\Uen}$ of $W$ where $V_\alpha := \coprod_{U \in \alpha} U$:
\[
 \Cen_\bullet(\Uen, \Ac_\Ec) \simeq \Nc_\bullet(\Ven, \Ac_\Ec).
\]
Observe now that $\Ven^I := (V_\alpha^I)_{\alpha \in P\Uen}$ is a cover of $W^I$ for every $I \in \finsurj$.
We get
\[
 |\Nc_\bullet(\Ven, \Ac_\Ec)| \simeq \Ac_\Ec(W)
\]
by using \ref{ass:AE-value-on-V} and the fact that $\DR(\Ec^{(I)})_\an$ is a sheaf on $W^I$ for every $I$. This shows \ref{ass:AE-is-FA}.

We now prove \ref{ass:FH-of-AE}. By assumption, each $\Ec^{(I)}$ is holonomic regular over $X^I$. By the Riemann-Hilbert correspondence, we have $R\Gamma_\mathrm{c}\left((X_\an)^I, \DR(\Ec^{(I)})_\an\right) \simeq R\Gamma_\DR^{[\![\c]\!]}\left(X^I, \Ec^{(I)}\right)$
In particular, we get using \ref{ass:AE-value-on-V}
\[\Ac_\Ec(X^\an) \simeq \hocolim_I R\Gamma_\mathrm{c}\left((X_\an)^I, \DR(\Ec^{(I)})_\an\right) \simeq \hocolim_I R\Gamma_\DR^{[\![\c]\!]}\left(X^I, \Ec^{(I)}\right) = \int_X \Ec.
\]
\end{proof}

\subsection{The case of the tangent bundle}\label{subsec:case-tan}

We now specialize  the considerations of Chapter \ref{sec:GFAG} and \S \ref{subsec:D!-e2n} to $\Ec = \psi(\wc \Cc^\bullet)$
being  the covariant Verdier dual to the cohomological Chevalley-Eilenberg $*$-sheaf $\wc\Cc^\bullet = \wc\Cc^\bullet_L$.
We further specialize  the local Lie algebra  to    $L=T_X$,    the tangent bundle of $X$. 

\paragraph{The diagonal $\Dc$-module.} Recall from \autoref{def:diagonal} the diagonal $\Dc$-module $\wc\Cc_\Delta^\bullet$ on $X$. 

\begin{lem}\label{lem:C-delta-rh}
    $\wc\Cc_\Delta^\bullet$ is regular holonomic. 
\end{lem}

\begin{proof}
    We first show that $\wc\Cc_\Delta^\bullet$ is   holonomic. 
    Let $x\in X$ be any point with the embedding $i_x:\{x\}\to X$. It suffices to show that for any $x$
    the $!$-fiber $i_x^! \wc\Cc_\Delta^\bullet$ (a complex of $\Dc$-modules on $\{x\}$, i.e., of
    vector spaces) has bounded and finite-dimensional cohomology. 
    
    \vskip .2cm

    Consider the Verdier dual complex  to $\wc\Cc^\bullet_\Delta$ and denote it $\Cc_\bullet^\Delta$. Then the $[\![*]\!]$-fiber
    $i_x^{[\![*]\!]}\Cc_\bullet^\Delta$ is dual to $i_x^! \wc\Cc_\Delta^\bullet$ and so it suffices to prove finite-dimensionality of
    the cohomology of all such fibers. Now, unravelling the definitions
    and using Example \ref{ex:f^*=completion} shows that $i_x^{[\![*]\!]} \Cc^\Delta_\bullet\,=\, \CE_\bullet (W_x)$
    is the homological Chevalley-Eilenberg complex of the topological Lie algebra $W_x = \Der (\wh\Oc_{X,x})$
    of formal vector fields near $x$. Here the Chevalley-Eilenberg complex  is understood in the completed sense. 
    
    Since the homology of $W_x$ is finite-dimensional  (Theorem \ref{thm:GF-Wn}), the holonomicity follows.

    \vskip .2cm
    
    Next, we show that  $ \wc \Cc^\bullet_\Delta $ is regular. For this, we denote this $\Dc_X$-module by $N_X$ and study its dependence on $X$.
    That is, if $j: X\to X'$ is an open embedding of smooth algebraic varieties, then $N_X = j^* N_{X'}$. Now, if $X'$ is compact,
    then $N_{X'}$, being a local system on $X'$, is regular. Embedding any $X$ into a smooth proper $X'$ we see
    that $N_X$ is also regular, being the pullback of a regular $\Dc$-module. 
    The lemma is proved.
\end{proof}

\begin{rem}
    \autoref{lem:C-delta-rh} shows that  the diagonal complex $\Fc^\bullet_\Delta = \DR(\wc \Cc^\bullet_\Delta)$
    is the {\em local system of the Gelfand-Fuchs cohomology}. 
    
\end{rem}

\paragraph{Comparison with analytification.}
Recall that  $\k=\CC$.

\begin{thm}\label{thm:reghol}~
    \begin{statements}
        \item The factorization algebra $ \Ac_{\psi(\wc \Cc^\bullet)}$ on $X_\an$ is locally constant.
        \item The canonical map
        \[
        \int_X^{[\![\c]\!]} \psi(\wc\Cc^\bullet ) \lra  R\Gamma^{\c}\left(X_\an^\finsurj , \, \DR(\psi(\wc\Cc^\bullet))_\an\right) \,=\, \int_{X_\an} \Ac_{\psi(\wc \Cc^\bullet)}
        \]
        is an equivalence.
    \end{statements}
\end{thm}

\begin{proof} 
    We know that $\wc\Cc^\bullet$ is a factorizing coherent $[\![\Dc]\!]$-module on $X^\finsurj$. By \autoref{thm:coVerdier-fact},
    $\psi(\wc\Cc^\bullet)$ is factorizable. Now 
    \autoref{lem:C-delta-rh} implies that $\wc\Cc^\bullet_\Delta \simeq \psi(\wc\Cc^\bullet)^{(\pt)}$ is holonomic regular.
    \autoref{prop:e_1hr-ehr} implies that $\psi(\wc\Cc^\bullet)$ is a h.r. factorizing  $\Dc^!$-module on $X^\finsurj$.
    After this the theorem becomes an application  of 
    \autoref{prop:h-r-fact-anal}.
\end{proof}

Combining \autoref {cor:affine-L}
with \autoref{thm:reghol}, we find:

\begin{cor}\label{cor:CET}
    For $X$ a smooth affine variety over $\k = \CC$, we have
    \[
    \CE^\bullet(T(X)) \, \simeq\,   \int_{X_\an} \Ac_{\psi(\wc \Cc^\bullet)}. 
    \]
\end{cor}

\paragraph{The  structure of the factorization algebra  $\Ac$.}
Denote $\Ac = \Ac_X$ the locally constant factorization algebra $\Ac_{\psi(\wc \Cc^\bullet)}$. 
We note, first of all, that $\Ac$ is naturally a factorization algebra with values in $\cdga$. 
This is because all the steps in constructing $\Ac$ can be done in $\cdga$. 
So by the original version of \autoref{prop:fact-cdga}
due to Ginot \cite[Prop.48]{ginot},  $\Ac$ is a cosheaf of cdga's on $X_\an$. 

\vskip .2cm

The complex manifold $X_\an$ can be seen as a $C^\infty$-manifold of dimension $2n$ with $GL_n(\CC)$-structure in the sense of 
\autoref{def:G-str}. Let $\GL_{n,\CC}$ be the algebraic group $GL_n$  with field of definition $\CC$. 
Any cdga $A$ with a $\GL_{n,\CC}^*$-action 
(\autoref{def:G*}) has a BL-action of  the Lie group
$GL_n(\CC)$  and so gives rise to a locally constant cosheaf of cdga's $ \ul A_{X_\an}$ on $X_\an$,
see \autoref{prop:assoc-alg}. 

\vskip .2cm

We recall from \S \ref{subsec:clas-GF}C that the cdga $\CE^\bullet(W_n(\CC))$ 
has a natural $\GL_{n,\CC}^*$-action
and so we have the cosheaf of cdga's $\ul{ \CE^\bullet(W_n(\CC))}_{X_\an}$ on $X_\an$.

\begin{prop}\label{prop:A-X-A-n}
    The cosheaf of cdga's $\Ac_X$ on $X_\an$ is weakly equivalent to $\ul{ \CE^\bullet(W_n(\CC))}_{X_\an}$.
    In particular,
    \[
    \int_{X_\an} \Ac_X \,\simeq \, \int_{X_\an} (\CE^\bullet(W_n(\CC)))
    \] 
    is the factorization homology of the complex manifold $X_\an$ with coefficients in $\CE^\bullet(W_n(\CC))$. 
\end{prop}

\begin{proof}  
    Let $U\subset X_\an$ be a disk.
    Applying Theorem 5.5.4.14 of \cite{lurie-ha}, we see that the natural arrow
    \[
    R\Gamma_c\bigl(U, \DR(\psi(\wc \Cc^\bullet)^{(\pt)}_\an)\bigr) \lra \Ac(U)  := \,
    \hocolim_{I\in\finsurj} 
    R\Gamma_c\bigl(U^I, \DR(\psi(\wc \Cc^\bullet)^{(I)}_\an)\bigr)
    \]
    is a quasi-isomorphism.  Further, by definition, 
    \[
    R\Gamma_c\bigl(U, \DR(\psi(\wc \Cc^\bullet)^{(\pt)}_\an)\bigr)\,=\, R\Gamma_c(U, \Fc_\Delta^\bullet)
    \]
    is the compactly supported cohomology with coefficients in the diagonal complex, see \eqref{eq:complex-F}. 
    
    From this point on the proof proceeds similarly to that  of \autoref{prop:TM-Wn}. 
    Let $A=\CE^\bullet(W_n(\CC))$. 
    Our cosheaf $\ul A_{X_\an}$ is the inverse of the locally constant sheaf of cdga's
    $\gamma_{TX}^{-1}|\Sigma_{GL}|$, where $\gamma_{TX}: X_\an\to BGL_n(\CC) = |N_\bullet GL_n(\CC)|$ is the classifying map for $TX$ and $\Sigma_{GL}$ is the locally constant sheaf on $N_\bullet GL_n(\CC)$
    associated to the $\GL_{n,\CC}^*$-action on $A$. This action comes from a $\JJ_n^*$-action,
    where $\JJ_n=\ul\Aut\, \CC[\![z_1,\dots, z_n]\!]$, see Example \ref {ex:group-J}.
    So we have a locally constant sheaf $\Sigma_J$ on $N_\bullet\JJ_n(\CC)$ associated to it.
    We further have the Gelfand-Kazhdan scheme $p: SX\to X$ of formal coordinate systems on $X$ which is
    a principal $\JJ_n$-bundle equipped with a formally  flat connection. So we have a classifying map
    $\gamma_{SX}: X_\an\to B\JJ_n(\CC)$. The homotopy equivalence $GL_n(\CC)\to \JJ_n(\CC)$
    and homotopy invariance of BL-actions give an identification $\ul A_{X\an}\= \gamma_{SX}^{-1}|\Sigma_J|$. 
    
    On the other hand, by definition, 
    $\Fc^\bullet_\Delta \= \ul\Omega^\bullet_{X_\an} \otimes(\ul\CEW^\bullet_{X})_\an$,
    where $\ul\CEW_X^\bullet$ is the (dg-)vector bundle on $X$ associated to $SX$ via the
    $\JJ_n$-action on $A$ and equipped with the associated connection. Equivalently,
    the fiber of $\ul\CEW_X^\bullet$ at any $x\in X$ is $\CE^\bullet(W_x)$. 
    
    Now, representing $\gamma_{SX}$ as the natural map from the simplicial  quotient $SX(\CC)/\!/ \JJ_n(\CC)
    \buildrel \=\over\to X$
   to $N_\bullet \JJ_n(\CC)$, we compare $\ul\Omega^\bullet_{X_\an} \otimes(\ul\CEW^\bullet_{X})_\an$
   with $\gamma_{SX}^{-1}|\Sigma_J|$ directly as in the proof of Proposition \ref{prop:gamma_SM=CEW}. 
   \end{proof}


\subsection {Main result}\label{subsec:main-res}

Let now $X$ be a smooth variety over $\CC$ of complex dimension $n$. Using the $GL_n(\CC)$-action on $Y_n$, we form the
{\em holomorphic Gelfand-Fuchs fibration} $\ul Y_X \to X_\an$ with fiber $Y_n$.  Note that $Y_n$ is $2n$-connected,
and so the  non-Abelian Poincar\'e duality \autoref{thm:NAPD} applies to $\ul Y_X\to X_\an$ (the real dimension of $X_\an$
is also $2n$).

Combining \autoref{thm:W-Y} with \autoref {prop:A-X-A-n}
and \autoref{thm:NAPD}, we obtain:

\begin{thm}\label{thm:top-approx}~
    \begin{statements}
        \item Let $X$ be any smooth algebraic variety over $\CC$. Then we have identifications
        \[
        \int_X^{[\![\c]\!]} \psi(\wc\Cc^\bullet) \,\simeq \, \int_{X_\an} \Ac_{\psi(\wc\Cc^\bullet)} \,\simeq \,
        H^\bullet_\top(\Sect(\ul Y_X/X_\an), \CC),
        \]
        where on the right we have the space of continuous sections of $\ul Y_X$ over $X_\an$ (considered as just a topological space). 
        \item In particular, the canonical arrow $\int_x^{[\![\c]\!]} \psi(\wc\Cc^\bullet) \to \int_x^{[\![\c]\!]} \wc\Cc^\bullet$ gives rise to a natural morphism
        of commutative algebras
        \[
        \tau_X: H^\bullet_\top(\Sect(\ul Y_X/X_\an),\CC)\lra H^\bullet_\Lie R\Gamma(X, T_X),
        \]
        compatible with pullbacks under \'etale maps $X'\to X$. 
    \end{statements}
\end{thm}

Applying now \autoref{cor:CET}, we obtain our main result:

\begin{thm}\label{thm:main}~
    \begin{statements}
        \item Let $X$ be a smooth affine variety over $\CC$. Then $\tau_X$ is an isomorphism, i.e., 
        we have a commutative algebra isomorphism 
        \[
        H^\bullet_\Lie(T(X)) \, \simeq \, H^\bullet_\top(\Sect(\ul Y_X/X_\an), \CC). 
        \]
        
        \item In particular, $H^\bullet_\Lie(T(X))$ 
        is finite-dimensional in each degree and is an invariant of $n$, of  the rational homotopy type of $X_\an$ 
        and of the  Chern classes $c_i(T_X)\in H^{2i}(X, \QQ)$. \qed
    \end{statements}
\end{thm}

\begin{exas} Let us give two counterexamples in the case of non-affine (projective) varieties.
    \begin{assertions}
        \item For $X = \PP^1$, we have
        \begin{align*}
            & H^i_\Lie(T(\PP^1)) \simeq H^i_\Lie(\mathfrak{sl}_2(\CC)) \simeq 
            \begin{cases}
                \CC & \text{if }i = 0, 1 \\ 0 & \text{otherwise;}
            \end{cases}
            \\
            & H^i_\top(\Sect(\ul Y_{\PP^1}/\CC\PP^1), \CC) \simeq H^i_\Lie(\gl_2(\CC)) \simeq
            \begin{cases}
                \CC & \text{if }i = 0, 1, 3, 4 \\ 0 & \text{otherwise;}
            \end{cases}
        \end{align*}
        Indeed, recall that $Y_1$ is the $3$-sphere $S^3$. By the Smale conjecture (proved by Hatcher \cite{hatcher}), $\operatorname{Diff}(S^3) \simeq O(4)$. In particular
        \[
        \pi_1(\operatorname{Diff}(S^3)) \simeq \pi_1(O(4)) \simeq \ZZ/2.
        \]
        The bundle $\ul Y_{\PP^1}$ on $\CC\PP^1$ is induced from the tangent bundle $T = \Oc(2)$ along the morphism $U(1) \to \operatorname{Diff}(S^3)$. It is therefore trivial (as $2 = 0$ in $\pi_1(\operatorname{Diff}(S^3))$. We get $\ul Y_{\PP^1} \simeq Y_1 \times \CC\PP^1 \simeq S^3 \times \CC\PP^1$ and thus
        \begin{align*}
            H^\bullet_\top(\Sect(\ul Y_{\PP^1}/\CC\PP^1), \CC) &\simeq H^\bullet(\mathrm{Map}(S^2, S^3)) 
            \\ &\simeq H^\bullet(S^3, \mathbb C) \otimes H^\bullet(\Omega^2 S^3, \mathbb C)
            \\ &\simeq (\mathbb C \oplus \mathbb C[-3]) \otimes (\mathbb C \oplus \mathbb C[-1])
            \\ &\simeq \mathbb C \oplus \mathbb C[-1] \oplus \mathbb C[-3] \oplus \mathbb C[-4]
            \\ &\simeq H^\bullet_\mathrm{Lie}(\gl_2(\mathbb C)).
        \end{align*}
        \item Let $X$ be an elliptic curve. The tangent bundle $T_X$ is trivial, so $\ul Y_X = X_\an \times Y_1$.
        Further, $Y_1$ is is homotopy equivalent to the $3$-sphere 
        $S^3$, while $X_\an$ is homeomorphic to $S^1\times S^1$. 
        So $\Sect(\ul Y_X/X_\an)$ is homotopy equivalent to $\Map(S^1\times S^1, S^3)$ and has cohomology
        in infinitely many degrees but finite-dimensional in each given degree. On the other hand, 
        \[
        R\Gamma(X, T_X) \, \simeq \, \CC \oplus \CC[-1] 
        \]
        is an abelian dg-Lie algebra and so $H^0_\Lie R\Gamma(X, T_X) = \CC[\![q]\!]$ is only pro-finite. 
        This shows that $\tau_X$ cannot be an isomorphism.
    \end{assertions}
\end{exas}


\subsection{Examples of explicit calculations of \texorpdfstring{$H^\bullet_\Lie(T(X))$}{H*(T(X))}}\label{subsec:explicit}

\paragraph{Curves: Krichever-Novikov algebras.} Let $X$ be a smooth affine curve. 
Assume that $X$ is   of genus $g\geq 0$ with $m\geq 1$ punctures.
The Lie algebra $T(X)$ is known
as a {\em Krichever-Novikov algebra}, see \cite{krichever-novikov}  \cite{schlich-book}. 

\vskip.2cm

\autoref{thm:main}  in this case gives the following. The Gelfand-Fuchs skeleton $Y_1$ is homotopy equivalent to the $3$-sphere 
$S^3$. The space $X_\an$ is homotopy equivalent to a bouquet of $\nu = 2g+m-1$ circles, and so 
the 
complex tangent bundle $T_X$ is topologically trivial. Therefore the fibration $\ul Y_X\to X$ is trivial, identified,
up to homotopy equivalence,  with $X\times S^3$.
The space of sections $\Sect(\ul Y_X/X)$ is therefore identified with the mapping space $\Map(X, S^3)$
and 
we obtain:
\be\label{eq:ex-curve}
H^\bullet_\Lie (T(X)) \,\simeq \, H^\bullet_\top\biggl(\Map\bigl (\bigvee_{i=1}^\nu S^1, S^3\bigr), \CC\biggr).
\ee
The analytic version of this statement (involving all analytic vector fields and their continuous cohomology) has been proved earlier in \cite{kawazumi}.

\paragraph{Complexifications.} An interesting class of examples is obtained by considering $n$-dimensional complex affine 
varieties $X$ which are in fact defined over $\RR$ so that the space of $\RR$-points $M= X(\RR)$ is a smooth compact $C^\infty$-manifold
of dimension $n$, homotopy equivalent to $X_\an$. In such cases the algebro-geometric cohomology $H^\bullet_\Lie(T(X))$
is, by \autoref{thm:main}, identified with the $C^\infty$ cohomology $H^\bullet_\Lie(\Vect(M))\otimes_\RR \CC$. 
Examples include:
\begin{assertions}
    \item$X= \AAA^1-\{0\}$, $M=S^1$.
    
    \item $X= GL_n$, $M= U(n)$.
    
    \item $X$ is the affine quadric $\sum_{i=0}^n z_i^2=1$, $M$ is the sphere $S^n$.
\end{assertions}

\paragraph{$\PP^{n}$ minus a hypersurface.} Suppose $X=\PP^{n}-Z$ where $Z$ is a smooth hypersurface of degree $d$. 
In this case we have, first of all:

\begin{prop}
    The Chern classes of $T_X$ vanish rationally. 
\end{prop}

\begin{proof} 
    Indeed, they are the restrictions of the Chern classes of $T_{\PP^{n}}$ which lie
    in
    \[
    H^*(\PP^{n}_\an, \CC) = \CC[h]/h^{n+1},    h=c_1(\Oc(1)).
    \]
    
    Now, $dh = c_1 (\Oc(d))$ vanishes on $X$ since $Z$ is the zero locus of a section of $\Oc(d)$.
    Therefore $h|_X =0$ as well,  and similarly for all powers of $h$.
\end{proof}

Now, all the information about the fibration $\ul Y_X\to X_\an$ which we use, is contained in the Chern classes of $T_X$,
as we are dealing with rational homotopy types. 
Therefore \autoref{thm:main} gives that
\[
H^\bullet_\Lie (T(X)) \,=\, H^\bullet_\top (\Map(X_\an, Y_n), \CC). 
\]
Let us now identify the rational homotopy type of $X_\an$. Let us think of $\PP^{n}$ as the projectivization of $\CC^{n+1}$  and let
$f(x_0,\dots, x_n)$ be the homogeneous polynomial of degree $d$ defining $Z$. 
Without loss of generality, we can take $f= x_0^d +\cdots + x_n^d$. 
Let $W\subset \CC^{n+1}$ be given by $f=1$.
We then have the  Galois covering $p:W\to X_\an$ with Galois group $\ZZ/d$ of $d^\mathrm{th}$ roots of $1$ acting diagonally on $\CC^{n+1}$. 

Now $W$ is the ``Milnor fiber" for the isolated singularity $f=0$.
(We could define $W$ by $f=\eps$ for any small  $\eps$ with the same effect).
Therefore by Milnor's theorem \cite{milnor}, $W$  is homotopy equivalent to a bouquet of
$\mu$ spheres $S^{n}$.  Here $\mu=(d-1)^{n+1}$  is the Milnor number of the singularity.
So the cohomology of the space $W$ with values $\CC$ is  $\CC^\mu$ in degree $n$ and $0$ elsewhere
(except $H^0=\CC$). 

Further, the vanishing of  the  higher cohomology of the group $\ZZ/d$ with coefficients in any $\CC$-module
and the Leray spectral sequence of the Galois covering $p$, combined with the theory of rational homotopy type,
imply:

\begin{prop}
    The rational homotopy type of $X_\an$ is that of a bouquet of $\nu$ spheres, where $\nu$ is the dimension of
    the invariant subspace $H^{n}(W)^{\ZZ/d}$. \qed
\end{prop}

The number $\nu$ can be found explicitly by using the  fact, standard in the theory of singularities \cite{wall},  that  the space
$H^n(W)$ (the space of vanishing cycles for $f$)   has the same $\ZZ/d$-character
as the Jacobian quotient of the module of volume forms
\[
\Omega^{n+1} (\AAA^{n+1})\bigl/ \bigl( \partial f/\partial x_i\bigr)_{i=0}^n \,=\, \det(\CC^{n+1})\otimes
\CC[x_0, \dots, x_n] \bigl/ (x_0^{d-1}, \dots, x_n^{d-1}).
\]
So $\nu$ is equal to the number of monomials 
\[
x_0^{i_0} \cdots x_n^{i_n}, \quad 0\leq i_k\leq d-2, \quad \sum i_k \equiv - n-1 \mod d.
\]
We get a statement of the form similar to  \eqref{eq:ex-curve}:

\begin{cor}
    \[
    H^\bullet_\Lie (T(X)) \,\simeq \, H^\bullet\biggl( \Map\bigl(\bigvee_{i=1}^\nu S^n, Y_n\bigr), \CC\biggr). 
    \]
\end{cor}

\begin{ex}
    Let $d=2$, i.e., $Z$ is a smooth quadric hypersurface.
    Then $W$ 
    (complex affine quadric) is homotopy equivalent to $S^{n}$ and so 
    $X_\an= \CC\PP^{n}-Z_\an$ is homotopy equivalent to $\RR\PP^{n}$.   The rational homotopy type of $\RR\PP^{n}$ is that of a point, if $n$ is even and is that of $S^{n}$,  if $n$ is odd.
    This means that $\nu$ is equal to $0$ or $1$ in the corresponding cases. Accordingly
    \[
    H^\bullet_\Lie(T(\PP^{n}-Z)) = \begin{cases}
        H^\bullet(Y_n,\CC), & \text{ if } n \text{ is even};
        \\
        H^\bullet(\Map(S^{n}, Y_{n}),\CC), & \text{ if } n \text{ is odd}.
    \end{cases}
    \]
\end{ex}


\vskip 1cm

\begin{description}
\item[B.H.:]
Universit\'e Paris-Saclay, Laboratoire de math\'ematiques d'Orsay,\\
91405 Orsay, France.\\
Email:
{\tt benjamin.hennion@universite-paris-saclay.fr}

\item[M.K.:]Kavli IPMU, 5-1-5 Kashiwanoha, Kashiwa, Chiba, 277-8583 Japan. Email: 
{\tt mikhail.kapranov@protonmail.com}
\end{description}

\ed